\documentclass[10pt]{article}
\usepackage{amssymb}
\usepackage{amsmath}
\usepackage{theorem}
\usepackage{epsfig}
\usepackage{verbatim}
\usepackage{graphicx}
\usepackage{enumerate}
\usepackage{enumitem}
\usepackage{cancel}
\usepackage{mathtools}
\usepackage{hyperref}
\usepackage{datetime}

\usepackage{color}

\newlength{\hchng}
\newlength{\vchng}

\newcommand {\bc} {\begin{center}}
\newcommand {\ec} {\end{center}}

\theoremstyle{plain}

\newtheorem{thm}{Theorem}[section]
\newtheorem{lem}[thm]{Lemma}
\newtheorem{defn}[thm]{Definition}
\newtheorem{prop}[thm]{Proposition}
\newtheorem{cor}[thm]{Corollary}
\newtheorem{rem}[thm]{Remark}

\allowdisplaybreaks[2]

\newenvironment{proof}[1]{\begin{trivlist} \item[] {\em Proof of #1:}}{\hfill $\Box$
                      \end{trivlist}}

\newcommand{\tmtextsc}[1]{{\scshape{#1}}}

\newcommand{\ud}{\,\mathrm{d}}
\newcommand{\la}{\lambda}

\newcommand{\R}{\mathbb{R}}

\newcommand{\Om}{\Omega}

\newcommand{\pa}{\partial}
\newcommand{\eps}{\epsilon}

\newcommand{\ty}{\tilde{y}}

\newcommand{\lqu}{\textquoteleft}

\newcommand{\eigx}{\psi^{(x)}_1(y)}

\newcommand{\paeigx}{\pa_x\psi^{(x)}_1(y)}

\newcommand{\tphi}{\tilde{\phi}}

\title{The Shape of the Level Sets of the First Eigenfunction of a Class of Two Dimensional Schr\"odinger Operators}

\date{\today}     
\author{Thomas Beck}

\begin{document}
\maketitle

\begin{abstract}

\noindent We study the first Dirichlet eigenfunction of a class of Schr\"odinger operators with a convex potential $V$ on a domain $\Om$. We  find two length scales $L_1$ and $L_2$, and an orientation of the domain $\Om$, which determine the shape of the level sets of the eigenfunction. As an intermediate step, we also establish bounds on the first eigenvalue in terms of the first eigenvalue of an associated ordinary differential operator.

\end{abstract}

\section{Introduction}

We are interested in studying a class of Schr\"odinger operators
\begin{align*}
\mathcal{L} = - \Delta_{x,y} + V(x,y).
\end{align*}
This operator acts on functions defined on the bounded, convex domain $\Om \subset \R^2$, and $V(x,y)$ is a convex potential.

The operator $\mathcal{L}$ has an increasing sequence of Dirichlet eigenvalues
\begin{align*}
\la_1 < \la_2 \leq \cdots \leq \la_j   \nearrow \infty,
\end{align*}
with corresponding eigenfunctions $u_j(x,y)$ satisfying
\begin{eqnarray*}
    \left\{ \begin{array}{rlcc}
    (-\Delta_{x,y} + V(x,y))u_j(x,y) & = \la_j u_j(x,y) && \text{in } \Om \\
   u_j(x,y) & = 0&& \text{on } \pa \Om.
    \end{array} \right.
\end{eqnarray*}
Our main focus will be to study the first eigenvalue $\la = \la_1$ and eigenfunction $u(x,y) = u_1(x,y)$. The first eigenfunction $u(x,y)$ does not change sign inside $\Om$ and so we normalise $u(x,y)$ so that it is positive inside $\Om$, and attains a maximum of $1$. In Definitions \ref{def:Om} and \ref{def:V} below, we will define the class of convex domains $\Om$ and potentials $V(x,y)$ that we are interested in. We will see that one consequence of the assumptions on $\Om$ and $V(x,y)$ is that it ensures that the superlevel sets of $u(x,y)$, 
\begin{align*}
W_c \coloneqq \{ (x,y) \in \Om: u(x,y) \geq c \},
\end{align*}
are convex subsets of $\Om$ for all $0 \leq c \leq 1$.

A theorem of John, \cite{Jo}, therefore implies that for each $c$ we can find an ellipse $E_c$ contained within this superlevel set $W_c$, such that a dilate of $E_c$, with scaling factor bounded by an absolute constant contains $W_c$. We are interested in determining the shape of the level sets of $u(x,y)$, and to do this we will study the lengths and orientation of the axes of the ellipse $E_c$. One of the main steps in establishing the shape of the level sets of $u(x,y)$ will be to prove sufficiently precise bounds on the first eigenvalue $\la$.

We know that the level set $\{(x,y)\in\Om:u(x,y) = 0\}$ is equal to the boundary, $\pa\Om$, and so in particular the shape of this level set is determined solely by the geometry of $\Om$. However, we will see that, in general, for the intermediate level sets, for example $\{(x,y)\in\Om:u(x,y) = \tfrac{1}{2}\}$, it is not solely the shape of $\pa\Om$ that governs its shape, but instead the two length scales $L_1$ and $L_2$. These length scales $L_1$ and $L_2$ will be given in Definitions \ref{def:L1} and \ref{def:L2}, but the key feature of their definitions is the following: The length scale $L_1$ will be defined purely in terms of the geometry of $\Om$ and properties of the potential $V(x,y)$, but the length scale $L_2$ will also depend on a family of associated one dimensional Schr\"odinger operators. Moreover, the definition of $L_2$ will also describe the orientation of these level sets of $u(x,y)$.

Our motivation for studying this problem is as follows: First, $\la$ and $\Psi(t,x,y) = e^{\la t}u(x,y)$ are the lowest energy and ground state eigenfunction of the quantum system governed by the Schr\"odinger operator
\begin{align*}
\pa_t \Psi(t,x,y) +  \mathcal{L}\Psi(t,x,y) = 0 .
\end{align*}
The main motivation comes from the series of papers \cite{J1}, \cite{GJ1}, \cite{GJ2}. There, the authors study the first two Dirichlet eigenfunctions on two dimensional convex domains $\Om$, normalised so that the inner radius is comparable to $1$, and the diameter is equal to the large parameter $N$. We will describe their results and techniques in more detail below, but for now we will briefly describe one of the techniques used that is most relevant for us: Using their normalisation of the domain $\Om$, they write it as 
\begin{align*}
\Om = \{ (x,y): f_1(x) < y  < f_2(x), a<x<b \},
\end{align*}
for functions $f_1(x)$ and $f_2(x)$, which are convex and concave respectively, and they consider the concave \textit{height function} $h(x)$,
\begin{align*}
h(x) = f_2(x) - f_1(x),
\end{align*}
with $\max_{x\in[a,b]}h(x) = 1$. This allows us to define a large parameter $L$, purely in terms of the function $h(x)$ (and hence just depending on the geometry of the domain). This number $L$ is the largest value such that
\begin{align} \label{eqn:L}
h(x) \geq 1 - L^{-2}
\end{align}
on an interval $I$ of length at least $L$. Rather than the length of the diameter $N$, this parameter $L$ is the relevant length scale to study the low energy eigenfunctions. Since the inner radius of their domain is comparable to $1$, while the projection of the domain onto the $x$-axis is large compared to $1$, it is natural to study the two dimensional problem via an approximate separation of variables. For each fixed $x$, the domain $\Om$ consists of the interval $[f_1(x),f_2(x)]$ of length $h(x)$, which has first eigenvalue $\pi^2h(x)^{-2}$. Thus, the ordinary differential operator on the interval $[a,b]$, which is naturally associated with this separation of variables is
 \begin{align} \label{eqn:1dimJ}
 -\frac{d^2}{dx^2} + \frac{\pi^2}{h(x)^2},
 \end{align}
 with zero boundary conditions. In \cite{J1} the eigenvalues and eigenfunctions of this operator are used to generate appropriate test functions to provide bounds on the first eigenvalue in terms of $L$, and to estimate the location and width of the nodal line of the second eigenfunction. In \cite{GJ1}, they give a sharper estimate on the nodal line, and in \cite{GJ2} they study the location of the maximum of the first eigenfunction of $\Om$, and its behaviour near this maximum where they use this approximate separation of variables to relate it to the first eigenfunction of the one dimensional operator. As a straightforward consequence of their work, it is this length scale $L$ and orientation of the domain $\Om$ given above, which determines the shape of the level sets of the eigenfunction $u(x,y)$ in this special case.

The papers \cite{J1}, \cite{GJ1}, \cite{GJ2} also provide more motivation for studying the operators $\mathcal{L}$. In the same way that the one dimensional Schr\"odinger operator in \eqref{eqn:1dimJ} is used in a crucial way to study the eigenfunctions of two dimensional convex domains, it will be important to understand the properties of the eigenfunctions of $\mathcal{L}$ when considering the eigenfunctions of three (and higher) dimensional convex domains.

Before stating our results, let us define precisely the class of domains $\Om$ and potentials $V(x,y)$ that we will be considering here.

\begin{defn}[The Domain $\Om$] \label{def:Om}
The domain $\Om$ is a bounded, convex two dimensional domain with inner radius $N_1$, and diameter $N_2$. We assume that the diameter is large compared to an absolute constant, while the inner radius is bounded below by an absolute constant.
\end{defn}

\begin{rem}
Throughout, the constants that appear will depend on these absolute constants, but the dependence of any bounds on  the diameter and inner radius themselves (and the other parameters introduced below) will be explicitly stated.
\end{rem}
We now state the class of potentials of interest.

\begin{defn}[The Potential $V(x,y)$] \label{def:V}
The potential $V(x,y)$ on the domain $\Om$ satisfies
\begin{align*}
V(x,y) = \frac{1}{h(x,y)^{2}},
\end{align*}
where $h(x,y)$ is a concave function with $0\leq h(x,y) \leq 1$ and $\max_{\Om}h(x,y) = 1$. In other words, $V(x,y)^{-1/2}$ is concave on $\Om$ and
\begin{align*}
\min_{\Om}V(x,y) = 1.
\end{align*}
In particular, this also ensures that $V(x,y)$ is convex.
\end{defn}
We see that this ensures that the first derivatives of $V$ are bounded almost everywhere, and that the second derivatives of $V$ are positive measures. However, we do not impose any further regularity assumptions on the potential. Before continuing, let us briefly discuss the motivation behind Definition \ref{def:V}.
\begin{enumerate}
\item One allowed potential is the constant potential $V(x,y) = 1$. In this case, our operator is analogous to the purely two dimensional operator studied in \cite{J1}. In particular, we can renormalise our domain $\Om$ to ensure that the inner radius is comparable to $1$. Note in general, our potential $V(x,y)$ is not scale invariant, and so this is not as useful a normalisation for us.

\item The assumption that $V(x,y)^{-1/2}$ is concave is a natural one when we recall the motivation for studying this class of Schr\"odinger operators. In the same way that the operator in \eqref{eqn:1dimJ} has been used to study the eigenfunctions of two dimensional domains, the potential $V(x,y)$ that we are considering is naturally related to the three dimensional domain with height function proportional to $h(x,y)$. This assumption that $V(x,y)^{-1/2}$ is concave also appears in the work of Borell, \cite{B1}, \cite{B2}, when studying the concavity properties of the Green's functions associated to these Schr\"odinger operators. 

\item We do not claim that this is the only class of potentials for which the results below will be valid. In fact, many of the results can be restated to hold for a more general class of convex potentials (including those related to the harmonic oscillator). However, at times we will see that it is convenient to restrict to those potentials given in Definition \ref{def:V}, and so we will only state the results for this class of potentials.

\end{enumerate}

We can now introduce the crucial parameters $L_1$ and $L_2$ that will appear as important length scales in our study of the first eigenfunction $u(x,y)$. For each $c\geq0$, let us define the sublevel sets of $V(x,y)$ by
\begin{align*}
\Om_{c} \coloneqq \{ (x,y) \in \Om: V(x,y) \leq 1+c\}.
\end{align*}
Since $V(x,y)$ is convex, these sublevel sets $\Om_{c}$ are convex subsets of $\Om$.

\begin{defn}[The Parameter $L_1$] \label{def:L1}
Let $L_1$ be the largest value such that the sublevel set $\Om_{L_1^{-2}}$ has inner radius at least equal to $L_1$.
\end{defn}
\begin{rem}
This definition is analogous to the definition of the parameter $L$ from \cite{J1} described above, and roughly speaking is equal to the largest length scale $L_1$ on which the potential increases by at most $L_1^{-2}$ from its minimum.
\end{rem}

With $L_1$ fixed, we let $\tilde{L}_1$ be the diameter of the set $\Om_{L_1^{-2}}$. If $L_1$ and $\tilde{L}_1$ are comparable in size, then we define $L_2$ to be equal to $L_1$, but if
\begin{align*}
\tilde{L}_1 \gg L_1,
\end{align*}
then we now describe how to find $L_2$.

\begin{rem}
Throughout, the notation $A \gg B$ denotes $A\geq \tilde{C}B$, for some large fixed absolute constant $\tilde{C}>0$, and if this, and  the converse $B \gg A$, do not hold then we say that $A$ and $B$ are comparable. In particular, we  are not interested in the exact values of $L_1$ and $L_2$, but instead are interested in knowing whether any length scale is, or is not, comparable to $L_1$ and $L_2$. We will use the notation $C$ to represent an absolute constant, that is small compared to $\tilde{C}$, which may change from line to line.
\end{rem}

To obtain a value for $L_2$, we first rotate our domain $\Om$, so that  the projection of $\Om_{L_1^{-2}}$ onto the $y$-axis is of the smallest length amongst the projections onto any line. In particular, this means that the projection of $\Om_{L_1^{-2}}$ onto the $x$-axis is comparable to $\tilde{L}_1$, while the projection of $\Om_{L_1^{-2}}$ onto the $y$-axis is comparable to $L_1$. This also fixes the orientation of $\Om$.

For each fixed $x$, let the interval $\Om(x)$ be the cross-section of $\Om$ at $x$, and consider the ordinary differential operator
\begin{align} \label{eqn:L(x)}
 \mathcal{L}(x)\coloneqq -\frac{d^2}{dy^2} + V(x,y),
\end{align}
with zero boundary conditions on $\Om(x)$. We let $\mu(x)$ be the first eigenvalue of $\mathcal{L}(x)$, and define the minimum of these eigenvalues,
\begin{align*}
\mu^* \coloneqq \min_{x}\mu(x).
\end{align*}
We can now define the parameter $L_2$.
\begin{defn}[The Parameter $L_2$] \label{def:L2}
We define $L_2$ to be the largest value such that 
\begin{align*}
\mu^* \leq \mu(x) \leq \mu^* + L_2^{-2},
\end{align*}
for all $x$ in an interval $I$ of length at least $L_2$.
\end{defn}
\begin{rem} \label{rem:L2def}
Note that in this definition of $L_2$, we have used the orientation of $\Om_{L_1^{-2}}$ fixed above. Therefore, from now on, whenever we consider any property of the eigenvalue or eigenfunction that depends on the value of $L_2$, we will have to use this orientation of $\Om_{L_1^{-2}}$. In contrast, the definition of $L_1$ does not depend on the orientation of $\Om_{L_1^{-2}}$.
\end{rem}

Our main aim in the study of the first eigenfunction is to give precise information about the shape of the level sets $\{(x,y)\in\Om:u(x,y)=c\}$ which are near to the point where $u(x,y)$ attains its maximum of $1$. Since the potential $V(x,y)$ is a convex function and $\Om$ is a convex set, Theorem 6.1 in \cite{BL2}  tells us that $u(x,y)$ is log concave. Alternative proofs of this result have also been given in \cite{CF}, \cite{K}, \cite{KL}. In particular, this tells us that the superlevel sets are all convex. Since $\{(x,y)\in\Om :u(x,y) \geq 0 \} = \Om$, one way of viewing this result is that 
\begin{align*}
\{(x,y)\in\Om :u(x,y)\geq0 \} \text{ convex }  \Rightarrow  \{(x,y)\in\Om :u(x,y) \geq c \} \text{ convex}
\end{align*}
for all $0\leq c \leq 1$.

We will use the convexity of the superlevel sets of $u(x,y)$ in a crucial way to describe their shape near its maximum.

\begin{thm} \label{thm:shape}
Let $\Om$ and $V(x,y)$ be a domain and potential from Definitions \ref{def:Om} and \ref{def:V}. Fix a small absolute constant $c_1>0$, and let $L_1$ and $L_2$ be as in Definitions \ref{def:L1} and \ref{def:L2}. In particular, this means that we have fixed the orientation of the set $\Om_{L_1^{-2}}$. Then, for any fixed absolute constant $c$, with $c_1<c<1-c_1$, the level set $\{(x,y)\in\Om:u(x,y) = c\}$ has the following shape: There exists an ellipse $E$ with minor axis in the $y$-direction of length comparable to $L_1$ and major axis in the $x$-direction of length comparable to $L_2$, such that $E$ is contained inside this level set, and a dilate of $E$, with a scaling factor bounded by an absolute constant, contains this level set.
\end{thm}

\begin{rem}
The level set $\{(x,y)\in\Om:u(x,y) = 0\}$ is equal to $\pa\Om$, the boundary of $\Om$. We will see that in general the parameters $L_1$ and $L_2$ are not comparable to the inner radius and diameter of the original domain $\Om$. Thus, the result of Theorem \ref{thm:shape} does not remain valid when $c$ becomes close to $0$.
\end{rem}

\begin{cor}
For a convex set $W$, we define the eccentricity of $W$, \emph{ecc}$(W)$ in the usual way:
\begin{align*}
\emph{ecc}(W) = \frac{\emph{diam}(W)}{\emph{inradius}(W)}.
\end{align*}
For $c=0$, the eccentricity of the superlevel set $\{(x,y)\in\Om:u(x,y) \geq c\}$ is equal to the eccentricity of $\Om$, but as $c$ increases (while bounded above by $1-c_1$), the eccentricity of the superlevel set becomes comparable to $L_2/L_1$.
\end{cor}

The log concavity of the eigenfunction, and resulting convexity of its superlevel sets has been used previously in various situations. For example, in \cite{AC} moduli of convexity and concavity are introduced. Under certain conditions on the potential $V$, it is then possible to strengthen the log concavity of the first eigenfunction by finding an appropriate modulus of concavity. This allows the spectral gap for a class of Schr\"odinger operators to be compared to the case where the potential is identically zero,  and allows them to prove the Fundamental Gap Conjecture. In \cite{FJ} the convexity of the superlevel sets of the Green's function are used in a crucial way to prove third derivative estimates on the eigenfunction which are valid up to the boundary of the convex domain.

As well as the convexity of the superlevel sets of $u(x,y)$, a very important part of the proof of Theorem \ref{thm:shape} will be to obtain sufficiently precise eigenvalues bounds for the first eigenvalue $\la$. For $\mu(x)$ equal to  the first eigenvalue of the operator $\mathcal{L}(x)$, we consider the ordinary differential operator
\begin{align} \label{eqn:A}
\mathcal{A} = - \frac{d^2}{dx^2} + \mu(x),
\end{align}
and let $\mu$ be the first eigenvalue of this operator. Our eigenvalue bounds relate the value of $\la$ to this eigenvalue $\mu$.

\begin{thm} \label{thm:eigenvalue}
Let $\Om$ and $V(x,y)$ be a domain and potential from Definitions \ref{def:Om} and \ref{def:V}. If $L_2$ is defined as in Definition \ref{def:L2} and $\mu$ is the first eigenvalue of the operator $\mathcal{A}$ in \eqref{eqn:A}, then the first eigenvalue $\la$ of the operator $\mathcal{L}$ satisfies
\begin{align*}
\mu \leq \la \leq \mu + CL_2^{-2},
\end{align*}
for an absolute constant $C$.
\end{thm}

\begin{rem} \label{rem:uniformconstants}
Theorems \ref{thm:shape} and \ref{thm:eigenvalue} are valid for all domains and potentials satisfying the assumptions of Definitions \ref{def:Om} and \ref{def:V}, and the bounds are uniform for domains $\Om$ and potentials $V$ leading to the same values for $L_1$ and $L_2$.
\end{rem}

While it is much more straightforward to locate the eigenvalue $\la$ to an interval of length comparable to $L_1^{-2}$, we will see that the more precise bound obtained in Theorem \ref{thm:eigenvalue} is necessary to obtain sharp information about the length scale on which the eigenfunction $u(x,y)$ decays in the $x$-direction, and hence prove Theorem \ref{thm:shape}.

Theorem \ref{thm:eigenvalue} locates the first eigenvalue $\la$ to an interval of length comparable to $L_2^{-2}$, provided we know the value of $\mu$. However, $\mu$ is also an eigenvalue of a differential operator, and so it may seem like we have only been able to locate the unknown $\la$ in terms of another unknown $\mu$. Another reason why this theorem still has value is that whereas $\la$ is the first eigenvalue of a two dimensional partial differential operator (with a potential), $\mu$ is the first eigenvalue of an ordinary differential operator $\mathcal{A}$. Thus, from a computational standpoint, it is much easier to accurately approximate the value of $\mu$ compared to $\la$. Also, we notice that the parameter $L_2$ depends on the geometric properties of the domain $\Om$ and potential $V(x,y)$, together with the eigenvalues of the differential operator $\mathcal{L}(x)$ given in \eqref{eqn:L(x)}. In other words, $L_2$ also only depends on knowledge of ordinary differential operators. Thus, the bound given in Theorem \ref{thm:eigenvalue} gives information about the eigenvalue of a  two dimensional partial differential operator purely in terms of ordinary differential operators.

The idea of relating the eigenfunctions and eigenvalues of a two dimensional problem to an associated ordinary differential operator has also been used extensively by Friedlander and Solomyak in \cite{FS1}, \cite{FS2}, \cite{FS3}. In these papers, they use this approximate separation of variables to obtain asymptotics for the eigenvalues, and the resolvent of the Dirichlet Laplacian. They use a semiclassical method by sending a small parameter $\eps$ to $0$ in order to give a one-parameter of \lqu narrow' domains, and then write asymptotics in terms of this small parameter.
\\

Let us now describe how we will proceed in the sections below.

In Section \ref{sec:L1} we study the parameters $L_1$ and $L_2$ from Definitions \ref{def:L1} and \ref{def:L2} in more detail. In particular, we will obtain bounds on $L_1$ and $L_2$ in terms of the diameter and inner radius of the domain and the potential, and construct domains $\Om$ and potentials $V(x,y)$ to show to what extent these estimates are sharp. We will also give a straightforward bound on $\la$ in terms of $L_1$ by using the variational formulation for the first eigenvalue.

In Section \ref{sec:la} we will prove the eigenvalue bounds in Theorem \ref{thm:eigenvalue}. For each fixed $x$, $u(x,y)$ is an admissible test function for the operator $\mathcal{L}(x)$ from \eqref{eqn:L(x)}, and the lower bound on $\la$ will follow straightforwardly from this. The proof of the upper bound on $\la$ in Theorem \ref{thm:eigenvalue} is more involved. The starting point of the proof is to use the first eigenfunction, $\psi^{(x)}(y)$, of the operator $\mathcal{L}(x)$ to construct a suitable test function in the variational formulation for the first eigenvalue. To obtain the required upper bound on $\la$ it will be necessary to study the first variation of $\psi^{(x)}(y)$ in the cross-sectional variable $x$. To do this, we will derive the ordinary differential equation that this first variation satisfies for each fixed $x$. The bounds then follow from using the method of variation of parameters. It will be particularly important to have estimates on the relative size of the first derivative of the potential $V(x,y)$ and the size of $\psi^{(x)}(y)$.

Once we have established the bounds on $\la$ in Theorem \ref{thm:eigenvalue}, in Section \ref{sec:L2} we use them to study the first eigenfunction $u(x,y)$ itself. Our first aim is to prove a $L^2(\Om)$-bound on $u(x,y)$ which is consistent with the shape of the level sets required in Theorem \ref{thm:shape}. We begin by using Theorem \ref{thm:eigenvalue} to prove a Carleman-type estimate to show how the $L^2(\Om(x))$-norm of the cross-sections of $u(x,y)$,
\begin{align*}
H(x) = \int_{\Om(x)}u(x,y)^2 \ud y,
\end{align*}
decays from its maximum exponentially on a length scale comparable to $L_2$. To find the required bound on the $L^2(\Om)$-norm of $u(x,y)$, we then need to estimate the size of the maximum of $H(x)$. We will do this by proving $L^2(\Om)$-bounds on the first derivatives of $u(x,y)$, which are again consistent with Theorem \ref{thm:shape}. We finish Section \ref{sec:L2} by proving an Agmon-type estimate to give an indication of the behaviour of $u(x,y)$ at points at a large distance from its maximum.

In Section \ref{sec:shape} we study the shape of the level sets of $u(x,y)$ and complete the proof of Theorem \ref{thm:shape}. To do this we will use the results of Section \ref{sec:L2} on the $L^2(\Om)$-norms of $u(x,y)$ itself, and also its first derivatives. We will also use the log-concavity of the eigenfunction $u(x,y)$ in a crucial way, since it is this that ensures that the superlevel sets are convex.

Theorem \ref{thm:shape} gives information about the level sets, $\{(x,y)\in\Om:u(x,y) = c\}$ whenever $c$ is bounded away from $0$ and $1$. In Section \ref{sec:max}, we want to study the behaviour of the eigenfunction $u(x,y)$ near its maximum. In particular, we will relate the location of the maximum to the region where $V(x,y) - \la$ is bounded above by $-c^*L_1^{-2}$, for an absolute constant $c^*>0$. We will do this by first using a maximum principle to restrict attention to the part of $\Om$ where $V(x,y)-\la$ is at most comparable to $L_1^{-2}$. This will then be used to convert the $L^2(\Om)$-bounds on $\nabla_{x,y}u(x,y)$ from Section \ref{sec:L2} into pointwise bounds near the maximum of $u(x,y)$. These bounds are then  in turn used to prove the sharper estimate on the location of the maximum. We finish by giving two consequences of this estimate of the location of the maximum. The first is that we obtain sharper bounds on the derivative $\pa_yu(x,y)$ as we approach the maximum, and we also obtain an improved pointwise bound on $\pa_xu(x,y)$ in a region around the maximum of height comparable to $L_1$ in the $y$-direction, and  length comparable to $L_2$ in the $x$-direction.

 \subsection{Acknowledgements}
 
I would like to thank David Jerison for suggesting this problem to me and for many enlightening conversations. I would also like to thank my advisor Charles Fefferman for many useful discussions and for his help in improving the exposition in this paper.
%\textbf{Acknowledgements} 

\section{The Parameters $L_1$ and $L_2$} \label{sec:L1}

Before proving Theorems \ref{thm:eigenvalue} and \ref{thm:shape}, we first give some more properties of the parameters $L_1$ and $L_2$ defined in Definitions \ref{def:L1} and \ref{def:L2}. 

We first want to give upper and lower bounds for $L_1$, where we recall that $L_1$ is the largest value for which the sublevel set $\{(x,y)\in\Om: V(x,y) \leq 1+ L_1^{-2} \}$ has inner radius at least $L_1$. We can think of this as being analogous to the parameter $L$ from \cite{J1}, which we described earlier in \eqref{eqn:L}. In \cite{J1}, it was shown that this parameter $L$ satisfies
\begin{align*}
N^{1/3} \leq L \leq N,
\end{align*}
where $N$ is the diameter of the two dimensional domain. The upper bound on $L$ is attained by an exactly  rectangular domain, $[0,N]\times[0,1]$, and the lower bound is attained by a right triangle of height $1$ and length $N$. Moreover, any intermediate value for $L$ can be attained by interpolating between these two extreme cases and forming the appropriate trapezoidal shape.

We now give an analogous description for the possible values of $L_1$. Rather than the potential $V(x,y)$, it will be more convenient to work with the \textit{height function}
\begin{align} \label{eqn:h}
h(x,y) = V(x,y)^{-1/2},
\end{align}
which, by the assumptions on the potential, is a concave function, satisfying
\begin{align*}
0 \leq h(x,y) \leq 1,
\end{align*}
and attaining its maximum of $1$ at the minimum of $V(x,y)$.
\begin{prop} \label{prop:L1bounds}
Recalling that $N_1$ is the inner radius of the domain $\Om$, we have the bounds
\begin{align*}
cN_1^{1/5} \leq L_1 \leq N_1,
\end{align*}
for some absolute constant $c>0$.
\end{prop}
\begin{rem}
We will see in the proof of the proposition, that we are using the stronger assumption that $h(x,y) = V(x,y)^{-1/2}$ is concave, instead of just the convexity of $V(x,y)$.
\end{rem}
\begin{proof}{Proposition \ref{prop:L1bounds}}
The proposition follows easily when the inner radius $N_1$ is comparable to a constant, and so throughout we will assume that $N_1\gg 1$.

The upper bound follows trivially from the definition of $L_1$, and is attained, for example, when $V(x,y)$ (and hence $h(x,y)$) is identically equal to $1$.

Before proving the lower bound, we recall the following theorem of John, \cite{Jo}:
\begin{thm} \label{thm:john}
Let $K \subset \R^m$ be a convex domain. Then, there exists an ellipsoid $E$ such that if $c^* \in \R^m$ is the centre of $E$, then we have
\begin{align*}
E \subset K \subset c^*+m(E-c^*).
\end{align*}
That is, the ellipsoid $E$ is contained within the convex set $K$, but if it is dilated by a constant depending only on the dimension, then it contains $K$.
\end{thm}

We will also need the following simple property of concave functions:
\begin{lem} \label{lem:concave1}
Suppose $g(x)$ is a concave function on an interval of length $M$, with $0 \leq g(x) \leq 1$, and $g(0) = 1$. Let $0<\beta<1$ and suppose that $g(z) = 1 - \beta$ at some point $z \in (0,M)$. Then, we have the bound
\begin{align*}
M \leq \beta^{-1} z .
\end{align*}
\end{lem}
\begin{proof}{Lemma \ref{lem:concave1}}
By the assumptions on the function $g(x)$, it decreases by at most $1$ over an interval of length $M$. Thus, since it is a concave function, it must satisfy
\begin{align*}
g(x) \geq 1-\frac{x}{M}.
\end{align*}
Since $g(z) = 1- \beta$, this gives
\begin{align*}
1- \beta \geq 1 - \frac{z}{M}, \qquad \text{or equivalently} \qquad  M \leq \beta^{-1} z,
\end{align*}
as required.
\end{proof}

\begin{figure}[h!]
\begin{center}
\includegraphics[width=0.45\textwidth]{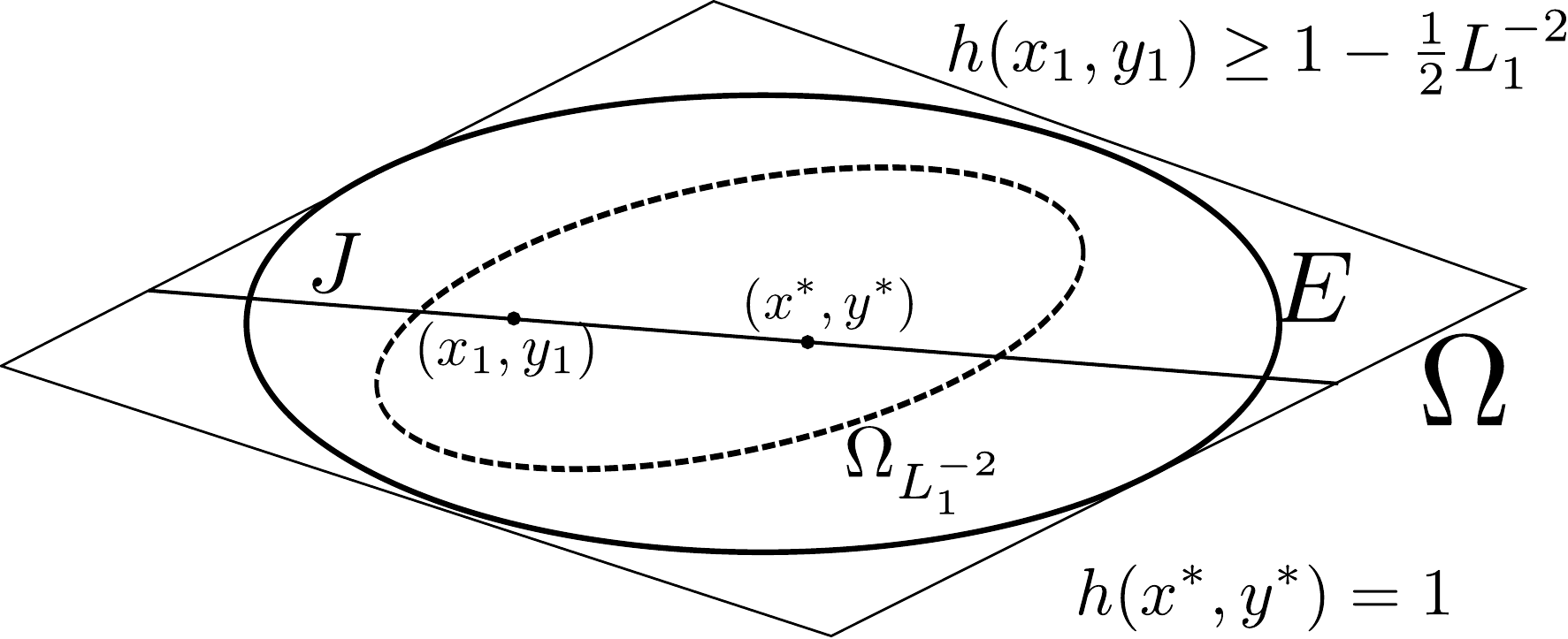}
\caption{The Domain $\Om$, and other sets appearing in the proof of Proposition \ref{prop:L1bounds}}
 \label{fig:L1bounds}
\end{center}
\end{figure}

We can now prove Proposition \ref{prop:L1bounds}. Let $E$ be the ellipse coming from Theorem \ref{thm:john} for our two dimensional domain $\Om$, and let $(x^*,y^*)$ be a point where $h(x,y)$ attains its maximum of $1$. Consider the ray $J$ which is the intersection of our domain $\Om$, and the line containing the point $(x^*,y^*)$ and the centre of the ellipse $E$ (see Figure \ref{fig:L1bounds}).

Since $\Om$ has inner radius equal to $N_1$, by the properties of the ellipse $E$, we know that the ray $J$ has length $M$ with
\begin{align} \label{eqn:concave1}
M \geq c_1 N_1,
\end{align}
for some small absolute constant $c_1>0$. Now consider the intersection of $J$ with the interior of the sublevel set
\begin{align*}
\Om_{L_1^{-2}} = \{(x,y)\in\Om: V(x,y) \leq 1 + L_1^{-2} \}.
\end{align*}
Let $J_1$ be this interval. If $V(x,y) = 1+ L_1^{-2}$, then $1- h(x,y) = 1- V(x,y)^{-2}$ will be comparable to $L_1^{-2}$, and so applying Lemma \ref{lem:concave1} with $\beta = L_1^{-2}$, we see that $J_1$ will be of length $A$, where
\begin{align} \label{eqn:concave2}
M \leq C_1L_1^{2} A,
\end{align}
for a large absolute constant $C_1$.

Combining \eqref{eqn:concave1} and \eqref{eqn:concave2} gives us
\begin{align} \label{eqn:concave2a}
c_1N_1 \leq M \leq C_1L_1^{2} A.
\end{align}
Thus, the lower bound of the proposition is established unless
\begin{align} \label{eqn:concave3}
A \geq C_2L_1^3,
\end{align}
for a large constant $C_2>0$.

Therefore, we will assume that \eqref{eqn:concave3} holds, and so in particular, $A$ is large compared to $L_1$. Let $E_{L_1^{-2}}$ be the ellipse from  Theorem \ref{thm:john} for the set $\Om_{L_1^{-2}}$, and rotate so that the minor axis of $E_{L_1^{-2}}$ lies in the $y$-direction. Then, by the definition of $L_1$, the minor axis of $E_{L_1^{-2}}$ has length comparable to $L_1$.

This means that the ray of length $A$ must approximately lie in the $x$-direction. $\Om$ is a convex set with inner radius $N_1$, and the original ray, $J$, through $\Om$ is of length $M$. Therefore, if we pick a point $(x_1,y_1)$ in the interval $J_1$, which is  at a distance of at least $A/4$ from the ends of $J_1$, then the height of $\Om$ in the $y$-direction at $x=x_1$ must be at least
\begin{align} \label{eqn:concave4}
 c_2 A N_1/M,
\end{align}
for a constant $c_2>0$. In contrast, the height of $\Om_{L_1^{-2}}$ at $x=x_1$ must be bounded above by $C_3 L_1$, since the minor axis of $E_{L_1^{-2}}$ lies in the $y$-direction and has length comparable to $L_1$.

Moreover, the concave function $h(x,y)$ varies from $1$ to $1- L_1^{-2}$  in the interval $J_1$ of length $A$. Thus, using Lemma \ref{lem:concave1} again, we have
\begin{align} \label{eqn:concave5}
h(x_1,y_1) \geq 1 - \frac{3}{4L_1^2},
\end{align}
at this point on the ray. 

Thus, combining \eqref{eqn:concave4} and \eqref{eqn:concave5}, we see that, for $x=x_1$ fixed, $h(x_1,y)$ is a concave function of $y$, which decreases by at most $1$ on an interval of length comparable to $AN_1/M$, and decreases by $\tfrac{1}{4}L_1^{-2}$ on an interval of length comparable to $L_1$. Thus, using Lemma \ref{lem:concave1} one more time, we see that
\begin{align} \label{eqn:concave6}
\frac{AN_1}{M} \leq C_4L_1^2L_1 = C_4L_1^3, 
\end{align}
for a constant $C_4$. Combining \eqref{eqn:concave2a} and \eqref{eqn:concave6} we see that
\begin{align*}
 M \leq C_1L_1^{2} A \leq C_1L_1^{2}C_4L_1^3 \frac{M}{N_1} = C_5 L_1^5 \frac{M}{N_1} ,
\end{align*}
for a constant $C_5>0$. Rearranging this inequality gives the desired lower bound on $L_1$.

\end{proof}

We noted in the proof of Proposition \ref{prop:L1bounds} that it is straightforward to give an example showing that the upper bound on $L_1$ is sharp. We now want to construct an example showing that the lower bound on $L_1$ is also optimal.

\begin{lem} \label{lem:sharpL1}
We can find a domain $\Om$ and potential $V(x,y)$ satisfying the assumptions of Definitions \ref{def:Om} and \ref{def:V} such that
\begin{align*}
L_1 \geq cN_1^{1/5},
\end{align*}
for some absolute constant $c>0$.
\end{lem}

\begin{figure}[h!]
\begin{center}
\includegraphics[width=0.45\textwidth]{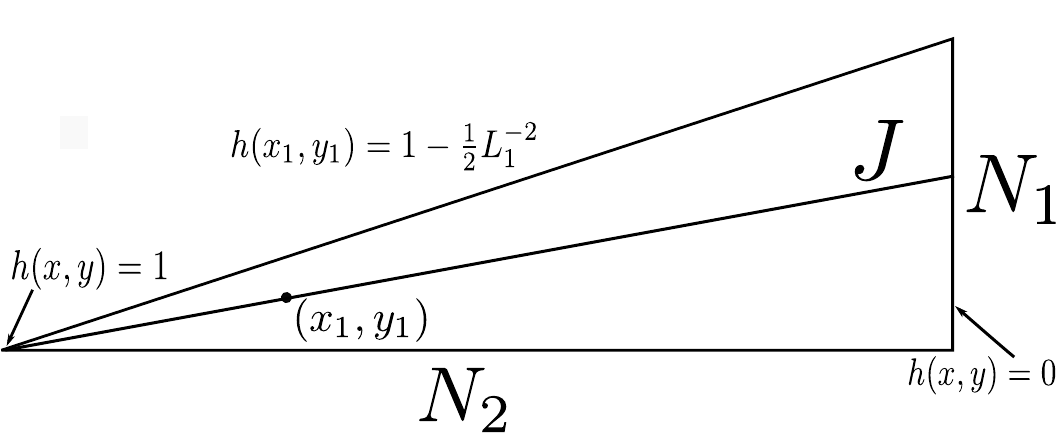}
\caption{The Domain in Lemma \ref{lem:sharpL1}}
 \label{fig:sharpL1}
\end{center}
\end{figure}
%file -> document properties -> fit page to selection

\begin{proof}{Lemma \ref{lem:sharpL1}}
We first construct the domain $\Om$. We have remarked earlier, that for the two dimensional domain case in \cite{J1}, a right triangle gives the smallest possible value for $L$. Motivated by this, we let $\Om$ be a right  triangle of side lengths $N_1$ in the $y$-direction, and side length $N_2$ in the $x$-direction (see Figure \ref{fig:sharpL1}). We note that while the inner radius of this domain is not identically to $N_1$, it is comparable to $N_1$ (independently of the size of $N_2$), and this is all we need.

We now define the potential $V(x,y)$, via the function $h(x,y) = V(x,y)^{-2}$. We let $h(x,y) = 1$ at the point where the hypotenuse joins the side of length $N_2$, and set $h(x,y) = 0$ at the midpoint of the side of length $N_1$. We then require $h(x,y)$ to decay linearly on the interval connecting these two points. Finally,  $h(x,y)$ decays linearly to $0$ in the $y$-direction as we move away from this interval. This defines $h(x,y)$ everywhere on $\Om$, and also ensures that $h(x,y)$ is a concave function. Thus the potential $V(x,y)$ satisfies the required properties.

We define $L_1$ as usual from Definition \ref{def:L1} for this domain $\Om$ and potential $V(x,y)$. Consider the line segment $J$ joining the vertex where $h(x,y) = 1$ to the midpoint of the opposite side, and let $M$ be the length of the line segment $J_1\subset J$ on which $h(x,y) \geq 1 - L_1^{-2}$. Then, since $h(x,y)$ decays linearly, and the whole of $J$ has length comparable to $N_2$, it is easy to see that
\begin{align} \label{eqn:sharp1}
M = c_1L_1^{-2}N_2,
\end{align}
for a constant $c_1>0$.

By the definition of $L_1$, the set $\{ (x,y)\in\Om: h(x,y) = 1-L_1^{-2}\}$ has inner radius comparable to $L_1$. Thus, at the point $(x_1,y_1)$ on the line segment $J$ with 
\begin{align} \label{eqn:sharp2}
h(x_1,y_1) = 1 - \tfrac{1}{2}L_1^{-2},
\end{align}
this set has height comparable to $L_1$ in the $y$-direction for $x=x_1$ fixed. Moreover, the point $(x_1,y_1)$ is at a distance comparable to $M$ from the vertex where $h(x,y) = 1$, and so the height of $\Om$ at this point is equal to
\begin{align} \label{eqn:sharp3}
c_2M \frac{N_1}{N_2},
\end{align}
for $c_2>0$. Thus, for $x=x_1$ fixed, $h(x_1,y)$ decays linearly to $0$ on an interval of length comparable to $L_1N_2/N_1$, and by \eqref{eqn:sharp2} decreases linearly by $\tfrac{1}{2}L_1^{-2}$ on an interval of length comparable to $L_1$. This tells us that
\begin{align} \label{eqn:sharp4}
L_1^3 = c_3 M \frac{N_1}{N_2}. 
\end{align}
Combining \eqref{eqn:sharp1} and \eqref{eqn:sharp4} gives
\begin{align*}
L_1^3 = c_3c_1L_1^{-2}N_2 \frac{N_1}{N_2} = c_3c_1L_1^{-2}N_1,
\end{align*}
and rearranging gives the desired estimate for $L_1$.
\end{proof}

\begin{rem}
By combining the two examples which show that the upper and lower bounds on $L_1$ from Proposition \ref{prop:L1bounds} are sharp, it is easy to construct examples where $L_1$ attains any intermediate length scale.
\end{rem}

We now want to consider the parameter $L_2$ introduced in Definition \ref{def:L2}. Before describing the bounds that $L_2$ must satisfy, we first give a simple bound on the eigenvalue $\la$.

\begin{prop} \label{prop:simpleeigenvalue}
The first eigenvalue $\la$ satisfies
\begin{align*}
1 \leq \la \leq 1 + C_1L_1^{-2},
\end{align*}
for an absolute constant $C_1>0$.
\end{prop}
\begin{proof}{Proposition \ref{prop:simpleeigenvalue}}
We will establish these bounds by using the variational formulation of the first eigenvalue, $\la$. That is,
\begin{align} \label{eqn:variation}
\la = \inf \left\{ \frac{\int_{\Om} \left|\nabla\psi(x,y)\right|^2 \ud x \ud y + \int_{\Om} V(x,y)\psi(x,y)^2 \ud x \ud y}{\int_{\Om} \psi(x,y)^2 \ud x \ud y} \bigg| \psi \in W^{1,2}(\Om), \psi|_{\pa\Om} = 0, \psi \not\equiv 0 \right\}
\end{align}
Since $V(x,y) \geq 1$ for all $(x,y) \in \Om$, the lower bound, $\la \geq1$ follows immediately.

To prove the upper bound, we need to construct a suitable test function $\psi(x,y)$ to use in \eqref{eqn:variation}. By the definition of $L_1$, we know that the sublevel set 
\begin{align*}
\Om_{L_1^{-2}} = \{(x,y):V(x,y) \leq  1 + L_1^{-2}\}
\end{align*}
has inner radius equal to $L_1$. Thus, we can choose a point $(x_0,y_0)$ and a constant $c>0$, such that the set
\begin{align*}
R = \{(x,y):|x-x_0| \leq cL_1, |y-y_0| \leq cL_1\}
\end{align*}
is contained in the interior of $\Om_{L_1^{-2}}$. We then define $\psi(x,y)$ as
\begin{align*}
\psi(x,y) = \cos\left(\frac{\pi (x-x_0)}{2cL_1}\right)\cos\left(\frac{\pi (y-y_0)}{2cL_1}\right)
\end{align*}
inside the square $R$, and set $\psi(x,y) = 0$ for all other $(x,y) \in\Om$. It is then clear that
\begin{align*}
 \frac{\int_{\Om} \left|\nabla\psi(x,y)\right|^2 \ud x \ud y}{\int_{\Om} \psi(x,y)^2 \ud x \ud y} \leq C_2L_1^{-2},
\end{align*}
and since $V(x,y) \leq 1+ L_1^{-2}$ on the support of the test function $\psi(x,y)$, we also have
\begin{align*}
 \frac{ \int_{\Om} V(x,y)\psi(x,y)^2 \ud x \ud y}{\int_{\Om} \psi(x,y)^2 \ud x \ud y} \leq 1 + C_3L_1^{-2}.
\end{align*}
Using these inequalities in \eqref{eqn:variation} gives the desired upper bound on $\la$.

\end{proof}

We now consider the parameter $L_2$ from Definition \ref{def:L2}. We recall that the sublevel set $\Om_{L_1^{-2}}$ has inner radius $L_1$ and diameter $\tilde{L}_1$, and that we set $L_2$ to be equal to $L_1$ unless $\tilde{L}_1 \gg L_1$. The upper and lower bound for $L_2$ from Definition \ref{def:L2} that we want to establish is the following:
\begin{prop} \label{prop:L2bound}
The parameter $L_2$ satisfies
\begin{align*}
c_1 \tilde{L}_1^{1/3} L_1^{2/3} \leq L_2 \leq \frac{1}{c_1}\tilde{L}_1,
\end{align*}
for some absolute constant $c_1>0$.
\end{prop}
\begin{rem}
In particular, the lower bound shows us that if we have $\tilde{L}_1 \gg L_1$, then also $L_2 \gg L_1$.
\end{rem}

\begin{proof}{Proposition \ref{prop:L2bound}}
The value of $L_2$ depends on the function $\mu(x)$, where $\mu(x)$ is the first eigenvalue of the operator
\begin{align} \label{eqn:Lx2}
\mathcal{L}(x) = - \frac{d^2}{dy^2} + V(x,y).
\end{align}
$L_2$ is the largest value such that $\mu(x)$ increases by $L_2^{-2}$ from its minimum value, $\mu^*$, on an interval of length at least $L_2$. Therefore, before proving the bounds on $L_2$, we first want to study the properties of the function $\mu(x)$.

We have rotated $\Om$ so that the projection of the set $\Om_{L_1^{-2}}$ onto the $y$-axis is of the smallest length amongst the projections onto any line. One immediate consequence of this is that if we set  $J$ to be the interval which is the projection of $\Om_{L_1^{-2}}$ onto the $x$-axis, then the length of $J$ is comparable to $\tilde{L}_1$, the diameter of $\Om_{L_1^{-2}}$.

We now give a bound on the eigenvalues $\mu(x)$ for $x \in J$.
\begin{lem} \label{lem:mubound}
For $x$ in the middle half of the interval $J$, there exists an absolute constant $C_1>0$ such that
\begin{align*}
1 + \frac{1}{C_1L_1^{2}} \leq \mu(x) \leq 1 + \frac{C_1}{L_1^2}.
\end{align*}
\end{lem}
\begin{proof}{Lemma \ref{lem:mubound}}
Since $\mu(x)$ is the first eigenvalue in the ordinary differential operator in \eqref{eqn:Lx2}, we want to apply Lemma 2.4 (a) in \cite{J1}. This lemma implies that
\begin{align} \label{eqn:Lem2.4}
1 + \frac{1}{C_1L(x)^2} \leq \mu(x) \leq 1 + \frac{C_1}{L(x)^2},
\end{align}
where $L(x)$ in the length scale associated to $V(x,y)$. In other words, for each fixed $x$, $L(x)$ is the largest value such that $V(x,y)$ varies from its minimum by $L(x)^{-2}$ on an interval of length at least $L(x)$. Thus, to prove the lemma it is enough to show that $L(x)$ is comparable to $L_1$ whenever $x$ is in the middle half of the interval $J$.

The projections of $\Om_{L_1^{-2}}$ onto the $x$ and $y$-axes have lengths comparable to $\tilde{L}_1$ and  $L_1$ respectively. It follows from Theorem \ref{thm:john} that, for those $x$ in the middle half of $J$, the height of $\Om_{L_1^{-2}}$ in the $y$-direction is comparable to $L_1$. Since the potential $V(x,y)$ is convex, attains its minimum of $1$, and is equal to $1+L_1^{-2}$ on the boundary of $\Om_{L_1^{-2}}$, we know that for all $x$ in the middle half of  $J$, we must have $V(x,y) \leq 1 + \tfrac{1}{2}L_1^{-2}$ for some $y$.

As a result, for all $x$ fixed in the middle half of $J$, the potential $V(x,y)$ varies by an amount comparable to $L_1^{-2}$, for $y$ in an interval of length comparable to $L_1$.   Therefore, for each $x$ fixed the length scale $L(x)$ is comparable to $L_1$, and hence using \eqref{eqn:Lem2.4} we have the required bound.
\end{proof}
\begin{rem}
Since Lemma 2.4 (a) in \cite{J1} played a key role in the above, let us say a few words about its proof. The upper bound in \eqref{eqn:Lem2.4} follows easily by choosing the appropriate test function, just as in the proof of Proposition \ref{prop:simpleeigenvalue}. The proof of the lower bound is slightly more complicated and makes use of the convexity of the potential to ensure that it grows at a sufficiently fast rate once we move away from its minimum.
\end{rem}

Before completing the proof of Proposition \ref{prop:L2bound}, we need one more property of the function $\mu(x)$.
\begin{lem} \label{lem:mu(x)convex}
The first eigenvalue $\mu(x)$ is a convex function of $x$.
\end{lem}
\begin{proof}{Lemma \ref{lem:mu(x)convex}}
This convexity property follows from Corollary 1.15 in \cite{BL1}. The convexity of the eigenvalue is deduced from the log concavity of the fundamental solution of the associated diffusion operator.
\end{proof}
\begin{rem}
Although in the assumptions of Corollary 1.15 in \cite{BL1}, the potential does not depend on the $x$-variable, the proof of the log concavity of the fundamental solution (and hence the convexity of the first eigenvalue) follows in the same way if $V(x,y)$ is allowed to depend on $x$, provided it remains a convex function.
\end{rem}

We can now combine Lemmas \ref{lem:mubound} and \ref{lem:mu(x)convex} to complete the proof of Proposition \ref{prop:L2bound}: Since the interval $J$ is of length comparable to $\tilde{L}_1$, Lemma \ref{lem:mubound} tells us that $\mu(x)$ varies by an amount at most comparable to $L_1^{-2}$ for $x$ in an interval of length comparable to $\tilde{L}_1$. Thus, since $\mu(x)$ is a convex function, applying the same logic as in Lemma \ref{lem:concave1}, we immediately obtain the lower bound
\begin{align} \label{eqn:L2bound1}
L_2 \geq c_1\tilde{L}_1^{1/3}L_1^{2/3} .
\end{align}
By the convexity of $V(x,y)$, given $C_2>0$, we can find $C_3>0$ to ensure that
\begin{align*}
V(x,y) \geq 1 + C_2L_1^{-2} ,
\end{align*}
whenever the point $(x,y)$ is at least $C_3\tilde{L}_1$ from $\Om_{L_1^{-2}}$. This means that $\mu(x)$ certainly must increase by an amount comparable to $L_1^{-2}$ when $x$ is a distance comparable to $\tilde{L}_1$ from $J$, and this gives us the upper bound
\begin{align} \label{eqn:L2bound2}
L_2 \leq \frac{1}{c_1} \tilde{L}_1.
\end{align}
Combining the inequalities in \eqref{eqn:L2bound1} and \eqref{eqn:L2bound2} completes the proof of the proposition.
\end{proof}

\section{The Bound On The First Eigenvalue $\la$} \label{sec:la}

We recall from Proposition \ref{prop:simpleeigenvalue} that the first eigenvalue $\la$ satisfies
\begin{align*}
1 \leq \la \leq 1 + C_1L_1^{-2}.
\end{align*}

In this section we will assume that we have $\tilde{L}_1 \gg L_1$ (and hence $L_2\gg L_1$ also), and then prove the improved upper and lower bound on the eigenvalue $\la$ from Theorem \ref{thm:eigenvalue}. That is, we will show that $\la$ satisfies 
\begin{align} \label{eqn:lamu}
\mu \leq \la \leq \mu + CL_2^{-2},
\end{align}
where $\mu$ is the first eigenvalue of the ordinary differential operator
\begin{align} \label{eqn:Adefn}
\mathcal{A} = -\frac{d^2}{dx^2}  + \mu(x) .
\end{align}

The lower bound in \eqref{eqn:lamu} is more straightforward, and so we establish this bound first.

\begin{prop}[Lower bound on $\la$] \label{prop:lalower}
The first eigenvalue $\la$ satisfies
\begin{align*}
 \la \geq \mu. 
 \end{align*}
\end{prop}
\begin{proof}{Proposition \ref{prop:lalower}}
As before, for each $x$ fixed, let $\Om(x)$ be the cross-section of $\Om$ at $x$. Then, the first Dirichlet eigenfunction $u(x,y)$ satisfies $u(x,y) = 0$ whenever $y$ is at the endpoints of the interval $\Om(x)$. In particular, for each fixed $x$, the function $u(x,\cdot)$ is an admissible test function for the variational formulation of the first eigenvalue of the operator $\mathcal{L}(x)$. Thus,
\begin{align*}
 \int_{\Om(x)} (\pa_yu(x,y))^{2} + V(x,y)u(x,y)^2 \ud y \geq \mu(x) \int_{\Om(x)}u(x,y)^2 \ud y. 
 \end{align*}
Integrating this over $x$, and using
\begin{eqnarray*}
    \left\{ \begin{array}{rlcc}
    (-\Delta_{x,y} + V(x,y))u(x,y) & = \la u(x,y) &&   \text{in } \Om \\
   u(x,y) & = 0  && \text{on } \pa \Om,
    \end{array} \right.
\end{eqnarray*}
we see that
\begin{align*}
 \la\int_{\Om} u(x,y)^2 \ud x \ud y &= \int_{\Om} (\pa_xu(x,y))^{2} + (\pa_yu(x,y))^{2} + V(x,y)u(x,y)^2 \ud x \ud y \\
 & \geq \int_{\Om}  (\pa_xu(x,y))^{2} + \mu(x) u(x,y)^2 \ud x\ud y  \\
 & \geq \mu \int_{\Om} u(x,y)^2 \ud x \ud y.
 \end{align*}
 To get the final inequality, we have defined $u(x,y)=0$ outside $\Om$, used Fubini to calculate 
the interval in $x$ first, and then used the variational formulation for the first eigenvalue $\mu$ of the operator $\mathcal{A}$ in \eqref{eqn:Adefn}. This gives us the bound $\la \geq \mu$, as required.
\end{proof}

We now turn to the upper bound and prove:
\begin{prop}[Upper bound on $\la$] \label{prop:eigbound}
We have an upper bound on the first eigenvalue $\la$ of the form,
\[ \la \leq \mu + CL_2^{-2}, \]
for an absolute constant $C>0$.
\end{prop}

\begin{rem}
 From Lemma 4.2 (e) in \cite{J1}, the operator $\mathcal{A}$ defined in \eqref{eqn:Adefn} has spectral gap bounded from below by a multiple of $L_2^{-2}$. Therefore, obtaining bounds on $\la$ up to a precision of $CL_2^{-2}$ is important if we want this separation of variables in the $x$ and $y$ variables to be of use to us.
\end{rem}

\begin{proof}{Proposition \ref{prop:eigbound}}

As in the proof of the simple bound on $\la$ in Proposition \ref{prop:simpleeigenvalue}, we will again make use of the variational formulation for $\la$ given in \eqref{eqn:variation}. To do this we need to construct an appropriate test function, and our motivation will come from performing an approximate change of variables in the $x$ and $y$-directions. Before stating our test function, we need some definitions. 

\begin{defn} \label{def:psi1}
For each fixed $x$, we define $\psi_1^{(x)}(y)$ to be the $L^2$-normalised first eigenfunction of the ordinary differential operator $\mathcal{L}(x)$. That is, $\psi_1^{(x)}(y)$ is $L^2$-normalised on the cross-section $\Om(x)$, and satisfies
\begin{eqnarray*}
    \left\{ \begin{array}{rlcc}
    \left(-\frac{d^2}{dy^2} + V(x,y)\right)\psi_1^{(x)}(y) & = \mu(x) \psi_1^{(x)}(y) &&   \text{in } \Om(x) \\
  \psi_1^{(x)}(y) & = 0  && \text{on } \pa \Om(x).
    \end{array} \right.
\end{eqnarray*}
\end{defn}
\begin{defn} \label{def:chi}
Let $I$ be the interval of length $L_2$ from Definition \ref{def:L2}. We define the cut-off function $\chi(x)$ to be a positive function which is comparable to its maximum in the middle half of the interval $I$, and supported in the middle three quarters of $I$, such that it decays smoothly to zero from its maximum. We also require that $\chi(x)$ is $L^2$-normalised on the interval $I$. In particular, this allows us to ensure that
\begin{align*}
|\chi'(x)| \leq C_1L_2^{-3/2},
\end{align*}
for some absolute constant $C_1$.
\end{defn}

We can now define the test function $f(x,y)$ which we will use in \eqref{eqn:variation}.
\begin{defn} \label{def:f(x,y)}
We define the test function $f(x,y)$ by
\begin{align*}
f(x,y) \coloneqq \chi(x)\psi_1^{(x)}(y).
\end{align*}
\end{defn}

As a first step towards proving Proposition \ref{prop:eigbound}, we prove the following intermediate step.
\begin{prop} \label{prop:eigboundinter}
We have an upper bound for $\la$ of the form
\begin{align*}
 \la \leq  \mu + \int_{\Om} \chi(x)^2(\paeigx)^2 \ud x \ud y + C_1L_2^{-2},
 \end{align*}
for a constant $C_1$.
\end{prop}
\begin{proof}{Proposition \ref{prop:eigboundinter}}
To obtain an upper bound on the first eigenvalue $\la$, we will calculate the quotient from \eqref{eqn:variation}
\begin{align} \label{eqn:eigquot}
\frac{\int_\Om |\nabla f(x,y)|^2 \ud x\ud y + \int_\Om V(x,y)|f(x,y)|^2 \ud x\ud y}{\int_\Om |f(x,y)|^2 \ud x \ud y},
\end{align}
with $f(x,y)$ as in Definition \ref{def:f(x,y)}. Since $\psi^{(x)}(y)$ is $L^2(\Om(x))$-normalised in $y$ for any fixed $x$, and $\chi(x)$ is $L^2(I)$-normalised in $x$, first computing the integral in $y$, and then the integral in $x$, we see that the denominator in \eqref{eqn:eigquot} is equal to $1$. Thus, we have the bound
\begin{align} \label{eqn:eigquot1}
\la \leq \int_\Om  |\nabla_{x,y}\left(\chi(x)\psi_1^{(x)}(y)\right)|^2 \ud x \ud y + \int_\Om V(x,y) \chi(x)^2 \eigx^2 \ud x \ud y.
\end{align}
For each $x$, the function $\psi^{(x)}(y)$ satisfies
\begin{align} \label{eqn:L2norm}
\int_{\Om(x)} \eigx^2 \ud y = 1,
\end{align}
and it is equal to $0$ at the endpoints of the interval $\Om(x)$. Therefore, differentiating \eqref{eqn:L2norm} with respect to $x$, we obtain the orthogonality relation
\begin{align*}
 \int_{\Om(x)} \paeigx \eigx \ud y = 0.
 \end{align*}

 Thus, calculating the derivatives in the first integral in \eqref{eqn:eigquot1}, and using this orthogonality relation, we see that \eqref{eqn:eigquot1} becomes
 \begin{align*}
  \la \leq &\int_\Om  \chi'(x)^2\eigx^2 \ud x \ud y + \int_\Om \chi(x)^2(\paeigx)^2 \ud x \ud y \\
 &+ \int_{\Om}\chi(x)^2(\pa_y\eigx)^2 \ud x \ud y + \int_\Om V(x,y) \chi(x)^2 \eigx^2 \ud x \ud y.
 \end{align*}
The eigenfunction $\eigx$ of $\mathcal{L}(x)$ has eigenvalue $\mu(x)$, and so we have the inequality
 \begin{align*}
   \la \leq \int_I  \chi'(x)^2 \ud x + \int_\Om \chi(x)^2(\paeigx)^2 \ud x \ud y +\int_{I} \chi(x)^2\mu(x) \ud x.
 \end{align*}
 From Definition \ref{def:L2} we know that
 \begin{align*}
 |\mu(x) - \mu| \leq L_2^{-2}.
 \end{align*}
 Therefore, combining this with the bound on $\chi'(x)$ given in Definition \ref{def:chi}, we obtain the desired upper bound on $\la$ of 
 \begin{align*}
 \la \leq  \mu + \int_\Om \chi(x)^2(\paeigx)^2 \ud x \ud y +C_1L_2^{-2}.
 \end{align*}
 
\end{proof}

As a result of Proposition \ref{prop:eigboundinter}, to obtain an upper bound on $\la$, we need to consider the derivative with respect to $x$ of the eigenfunction $\eigx$. In particular, we want to bound
\begin{align*}
 \int_{\Om(x)} (\pa_x\eigx)^2 \ud y.
 \end{align*}
We will prove the following proposition:
\begin{prop} \label{prop:paxeigupperbound}
Let $x$ be fixed in the support of the cut-off function $\chi(x)$. Then,
\begin{align*}
 \int_{\Om(x)} (\paeigx)^2 \ud y \leq C_1L_2^{-2},
\end{align*}
with the constant $C_1$ independent of $x$.
\end{prop}
\begin{rem}
Combining Proposition \ref{prop:eigboundinter} with Proposition \ref{prop:paxeigupperbound} establishes
\begin{align*}
 \la \leq \mu + C_1L_2^{-2},
\end{align*}
and finishes the proof of Proposition \ref{prop:eigbound}.
\end{rem}

\begin{proof}{Proposition \ref{prop:paxeigupperbound}}

Throughout the proof of this proposition, $x \in I$ will be fixed in the support of the cut-off function $x$, and all bounds that appear will be uniform in $x$. We will also suppress the dependence of certain functions on $x$ where this simplifies the notation.

Since,
\begin{align*}
 \left( -\frac{d^2}{dy^2} + V(x,y)\right)\eigx = \mu(x) \eigx,
 \end{align*}
differentiating with respect to $x$ we find that for $y\in \Om(x)$, we have
\begin{align} \label{eqn:paeigx}
 \left( -\frac{d^2}{dy^2} + V(x,y) - \mu(x)\right)\pa_x\eigx = \mu'(x) \eigx - \pa_xV(x,y) \eigx,
 \end{align}
where the notation $'$ denotes differentiation with respect to $x$. Although, for each fixed $x$, $\eigx$ is equal to zero at the endpoints on $\Om(x)$, the function $\paeigx$ will not in general be zero here. 

Therefore, we will also need to take into account its boundary values. For those $x$ in the support of the cut-off function $\chi(x)$, we can write the two parts of $\pa \Om$ below and above in the $y$-direction as $\{ y= g_1(x)\}$ and $\{ y= g_2(x)\}$, where $g_1(x)$ and $g_2(x)$ are convex and concave functions respectively. We set $\alpha = \pa_x\psi_1^{(x)}(g_2(x))$, and define
\begin{align} \label{eqn:gdef}
g(y) \coloneqq \paeigx - \alpha.
\end{align} 
Our aim is to find an expression for the function $g(y)$ using \eqref{eqn:paeigx}. To do this we need to make the following definitions (again suppressing the dependence on $x$ throughout).

\begin{defn} \label{def:F}
We define the function $F(y)$ by,
\begin{align*}
 F(y) \coloneqq V(x,y) -\mu(x). 
 \end{align*}
\end{defn}
We know that $\mu(x) \leq 1+C_1L_1^{-2}$, and that $\min_yV(x,y) \leq \mu(x)$ for all $x$ in the support of $\chi(x)$. This allows us to define the three points $y_1$, $y_2$ and $y_3$.

\begin{figure}[h!] 
\begin{center}
\includegraphics[width=0.45\textwidth]{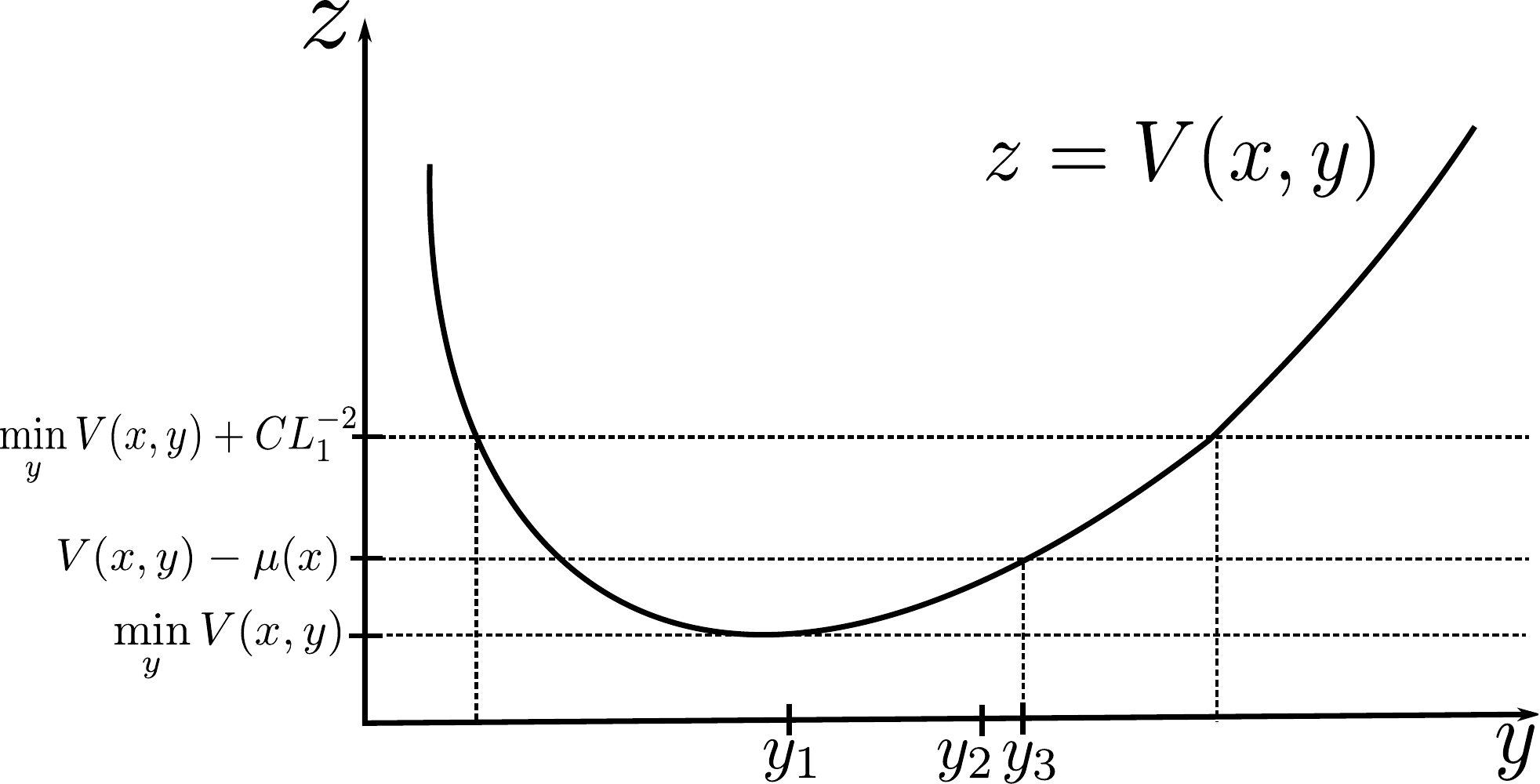}
\caption{The Points $y_1$, $y_2$ and $y_3$ from Definition \ref{def:y}}
\label{fig:y}
\end{center}
\end{figure}

\begin{defn} \label{def:y}
We fix an absolute constant $C$. We define $y_1$  to be the middle point of the \lqu centre', where the centre is  the interval on which $V(x,y) \leq \min_yV(x,y) + CL_1^{-2}$. We then choose  $y_2\geq y_1$ to be the largest value such that $[y_1,y_2]$ is contained in the middle half of the centre. Finally, we define $y_3\geq y_2$ to be the value of $y$ for which $F(y_3) = V(x,y_3) - \mu(x) = 0.$ (See Figure \ref{fig:y})
\end{defn}

\begin{defn} \label{def:phitilde}
We set $\phi(y)$ to be the first eigenfunction of $\mathcal{L}(x)$, but this time normalised to be positive with a maximum of $1$. Note that this function is equal to a multiple of $\eigx$ (where the multiple depends on the fixed value of $x$).

For $y \geq y_1$, we define the function $\tphi(y)$ by
\begin{align*}
\tphi(y) \coloneqq \phi(y) \int_{y_1}^{y} \phi(t)^{-2} \ud t.
\end{align*}
\end{defn}
We can now write down an expression for the function $g(y)$.

\begin{lem} \label{lem:gexpression}
 Let $c_0(x)$ be the value such that
\begin{align*}
 g(y) - c_0(x)\eigx = 0
 \end{align*}
at $y=y_1$. Then, for $y \geq y_1$, the function $g(y)$ satisfies
\begin{align} \label{eqn:gexpression} 
g(y) - c_0(x)\eigx = \phi(y)\int_{y_1}^y \tphi(t)G(x,t) \ud t + \tphi(y)\int_y^{g_2(x)}\phi(t)G(x,t) \ud t, 
\end{align}
where $G(x,y)$ is equal to
\begin{align*}
G(x,y)  =  \mu'(x) \eigx - \pa_xV(x,y) \eigx + (V(x,y) - \mu(x))\alpha . 
\end{align*}
\end{lem}
\begin{proof}{Lemma \ref{lem:gexpression}} 
We see from the definition of $g(y)$ from \eqref{eqn:gdef} and the equation that $\paeigx$ satisfies in \eqref{eqn:paeigx}, that we have
\begin{align*}
 \left(  -\frac{d^2}{dy^2} + V(x,y) -\mu(x)\right)g(y) = \mu'(x) \eigx - \pa_xV(x,y) \eigx + (V(x,y) - \mu(x))\alpha. 
\end{align*}
The right hand side of the above equation is equal to $G(x,y)$, so that
\begin{align} \label{eqn:geqn}
( \mathcal{L}(x) -\mu(x))(g(y)-c_0(x)\phi(y)) = G(x,y). 
\end{align}
Since $\mathcal{L}(x)$ is a second order ordinary differential operator, to find an expression for $g(y)$ we will apply the method of variation of parameters to \eqref{eqn:geqn}. From Definition \ref{def:phitilde}, we know that 
\begin{align*}
( \mathcal{L}(x) -\mu(x))\phi(y) = 0,
\end{align*}
with $\phi(g_2(x)) = 0$. It is straightforward to check that the function $\tphi(y)$ from Definition \ref{def:phitilde} also satisfies
\begin{align*}
( \mathcal{L}(x) -\mu(x))\tphi(y) = 0,
\end{align*}
for $y\geq y_1$, and is equal to $0$ at $y=y_1$. Thus, since the function $g(y) - c_0(x)\phi(y)$ is equal to $0$ at $y=y_1$ and $y=g_2(x)$, using \eqref{eqn:geqn} and variation of parameters, we can write
\begin{align*}
g(y) - c_0(x)\eigx = \phi(y)\int_{y_1}^y \tphi(t)G(x,t) \ud t + \tphi(y)\int_y^{g_2(x)}\phi(t)G(x,t) \ud t.
\end{align*}
\end{proof}

Looking at this expression for $g(y)$, we see that we will need to study how the magnitude of the functions $\phi(y)$ and $\tphi(y)$ depends on the size of the potential $V(x,y)$, and its derivative with respect to $x$, $\pa_xV(x,y)$. Also, since $g(y) = \paeigx - \alpha$, where $\alpha = \pa_x\psi_1^{(x)}(g_2(x))$, we will also need to estimate the size of $\paeigx$ at the endpoints of the interval $\Om(x)$.

\subsection{Properties of $\phi(y)$}

We first study the function $\phi(y)$, where we recall that it satisfies
\begin{align*}
\left(-\frac{d^2}{dy^2} + V(x,y) - \mu(x)\right) \phi(y) = 0.
\end{align*}
For $x$ fixed in the support of $I$, let us set $L(x)$ to be the largest value such that $V(x,y)$ varies from its minimum value by $L(x)^{-2}$ on an interval in $y$ of length at least $L(x)$. Then, as we remarked in the proof of Lemma \ref {lem:mubound}, $L(x)$ is comparable to $L_1$. Thus, from Lemma 2.4 (b), (d) in \cite{J1}, we immediately get the following estimates on $\phi(y)$ (uniformly in $x$).

\begin{lem} \label{lem:phibasic}
There exists an absolute constant $C_1$ such that the eigenfunction $\phi(y)$ (which we recall will depend on $x$) satisfies
\begin{align*}
|\phi'(y)| \leq C_1/L_1 \text{ for all } y \in \Om(x),
\end{align*}
and
\begin{align*}
\phi(y) \leq C_1e^{-c|y-y_1|/L_1} ,
\end{align*}
where $y_1$ is the point in the \lqu centre' given in Definition \ref{def:y}.
\end{lem}
This second inequality gives an $L^{\infty}$ exponential decay estimate for $\phi(y)$ as we move away from the minimum of $V(x,y)$ on a length scale comparable to $L_1$. In particular, this means that the $L^2(\Om(x))$ norm of $\phi(y)$ is bounded above by a multiple of $L_1^{1/2}$. (In fact, it follows from Lemma 2.4 in \cite{J1} that the $L^2(\Om(x))$-norm also has a lower bound that is comparable to $L_1^{1/2}$.)

We now want to sharpen this $L^{\infty}$ exponential decay estimate for $\phi(y)$ as $V(x,y)$ increases from its minimum.
\begin{prop} \label{prop:phi}
Define the interval $J_k$ by,
\begin{align} \label{eqn:Jk}
J_k= [t_k,t_{k+1}] \coloneqq \{ t\geq y_3: \pa_tV(x,t) \in [2^{-k}, 2^{-k+1}] \}.
\end{align}
Then, for all $t_k\leq t \leq g_2(x)$,
\begin{align*}
  \phi(t) \leq \phi(t_k) \exp(-(t-t_k)2^{-k/3}/10), 
  \end{align*}
for all $y_3 \leq t\leq t_{k+1}$,
\begin{align*}
 \phi(t_{k+1}) \leq \phi(t) \exp(-(t_{k+1}-t)2^{-k/3}/10) 
 \end{align*}
 and for all $t \in J_k$,
\begin{align*}
 \phi(t) \leq |\phi'(t)| 2^{k/3}. 
 \end{align*}
For the interval $\tilde{J}_k$ defined by,
\begin{align}  \label{eqn:Jktilde}
\tilde{J}_k = [\tilde{t}_k,\tilde{t}_{k+1}] \coloneqq \{ t\geq y_3: V(x,t) - \min_{t}V(x,t) \in [2^{-2k/3}, 2^{-2(k-1)/3}]\}, 
\end{align}
we have the analogous bounds on $\phi(t)$.
\end{prop}
\begin{rem}
We have the analogous decay estimates for $\phi(y)$ as we move away from the region where $V(x,y) \leq \min_yV(x,y) + L_1^{-2}$ in the other direction.
\end{rem}
\begin{rem}
We recall that $y=y_3$ is the point where $V(x,y) - \mu(x) = 0$. Since $\min_{y}V(x,y) - \mu(x) \leq -cL_1^{-2}$, by convexity, $J_k$ and $\tilde{J}_k$ are only non-empty for those $k$ satisfying $2^k \leq CL_1^3$, for some absolute constant $C>0$.
\end{rem}

\begin{proof}{Proposition \ref{prop:phi}}
The proposition follows from the key inequality given in the proof of Theorem A in \cite{J1},
\begin{align*}
\left|\left(\log \phi(t)\right)' \right|= |\phi'(t)|/\phi(t) \geq 2^{-k/3}/10\qquad  \text{ for all } t \in J_k. 
\end{align*} 
Integrating this inequality from both $t=t_k$ and $t=t_{k+1}$ gives all of the desired estimates involving the intervals $J_k$.

By the definition of the intervals $\tilde{J}_k$, we have $V(x,t)-\mu(x) \geq 2^{-2k/3}$ for $t \in \tilde{J}_k$. Therefore, it is straightforward to obtain the same bounds for $\left(\log \phi(t)\right)' $, and hence $\phi(t)$ itself on $\tilde{J}_k$ as for the intervals $J_k$.
\end{proof}

We now show to what extent $\phi'(y)$ inherits this exponential decay as we move away from the centre.
\begin{prop} \label{prop:phiy}
Let the intervals $J_k$ be defined as in Proposition \ref{prop:phi}. Then, for all $t\geq t_k$,
\begin{align*}
 |\phi'(t)| \leq C|\phi'(t_k)|\exp(-c|t-t_k|2^{-k/3}), 
 \end{align*}
for some absolute constants $c$ and $C>0$.
\end{prop}

\begin{proof}{Proposition \ref{prop:phiy}}
The function $\phi(t)$ satisfies the equation
\begin{align*}
 \phi''(t) = F(t)\phi(t), 
 \end{align*}
with the function $F(t) = V(x,t) - \mu(x)$ as before. On the intervals $J_k$, we know that $t \geq y_3$, and so certainly $F(t) \geq 0$. Also, $\phi'(t) \leq 0$, and so this mean that $|\phi'(t)|$ is decreasing. Thus, for $t \geq t_k$, we have
\begin{align*}
|\phi'(t)| \leq |\phi'(t_k)|.
\end{align*}
If $|t-t_k| \leq 2^{k/3}$, then this is enough to establish the required bound.

Now suppose that $|t-t_k| \in [N2^{k/3},(N+1)2^{k/3}]$ for some $N\geq1$. Then, by Proposition \ref{prop:phi}, we know that $\phi(t)$  satisfies
\begin{align*}
 \phi(t) \leq C2^{k/3}|\phi'(t_k)| \exp(-cN2^{-k/3}). 
 \end{align*}
 In particular, $\phi(t)$ changes by at most $C2^{k/3}|\phi'(t_k)|\exp(-cN2^{-k/3})$, as $t$ ranges over this interval of length $2^{k/3}$. Since $\phi'(t)$ is negative here, this gives us a bound on the integral of $|\phi'(t)|$ over this interval. 
 
 Moreover, as we noted above, by convexity, $|\phi'(t)|$ decreases as $t$ increases.  In particular, since the interval $[N2^{k/3},(N+1)2^{k/3}]$ has length $2^{k/3}$, this means that
\begin{align*}
 |\phi'(t)| \leq C2^{k/3}|\phi'(t_k)|\exp(-cN2^{-k/3}). 2^{-k/3} = C|\phi'(t_k)|\exp(-cN2^{-k/3}), 
 \end{align*}
for $t$ at the right endpoint of the interval. This concludes the proof of the proposition.
\end{proof}

It will often be important to measure the distance of a point $(x,y)$ from the level sets $\{(x,y)\in\Om: V(x,y) = 1+L_1^{-2}\}$.
\begin{defn} \label{def:ystar}
Fix a large absolute constant $C^*$. Then, suppressing the dependence on $x$, let $y^*\geq y_1$ be the first point where $V(x,y) \geq 1+C^*L_1^{-2}$.
\end{defn}

We can now write down an immediate corollary of Proposition \ref{prop:phiy}.
\begin{cor} \label{cor:phiy}
For any $t\geq t_k$, we have the first derivative estimate
\begin{align*}
 |\phi'(t)| \leq CL_1^{-1}\exp(-c|t-t_k|2^{-k/3})\exp(-c|t_k-y^*|/L_1). 
 \end{align*}
\end{cor}
\begin{proof}{Corollary \ref{prop:phiy}}
We can apply Proposition \ref{prop:phiy} with $t$ replaced by $t_k$ and $t_k$ replaced by $y^*$ to obtain a bound on $|\phi'(t_k)|$ of the form
\begin{align*}
|\phi'(t_k)| \leq CL_1^{-1} \exp(-c|t_k-y^*|/L_1).
\end{align*}
We then use this bound in the right hand side of the estimate for $|\phi'(t)|$ in Proposition \ref{prop:phiy} to get the desired result.
\end{proof}

\subsection{Properties of $\tphi(y)$}

From Lemma \ref{lem:gexpression}, we see that as well as $\phi(y)$, it will also be important to study the properties of $\tphi(y)$, where we recall that for $y \geq y_1$, we have
\begin{align*}
\tphi(y) = \phi(y) \int_{y_1}^{y} \phi(t)^{-2} \ud t.
\end{align*}
We recall from Definition \ref{def:y} that $y_2\geq y_1$ is the largest value of $y_2$ such that $[y_1,y_2]$ is contained in the middle half of the \lqu centre', where $V(x,y) \leq \min_{t}V(x,t) + CL_1^{-2}$  and that $y_3\geq y_2$ is the value of $y$ for which $F(y_3) = V(x,y_3) - \mu(x) = 0.$ We now prove:

\begin{lem} \label{lem:tphi}
The function $\tphi(y)$ satisfies
\begin{align*}
 \tphi(y) \leq C_1L_1,
 \end{align*}
for $y_1\leq y \leq y_3$ and
\begin{align*}
 \tphi(y) \leq C_1L_1 + C_1|\phi'(y)|^{-1},
 \end{align*}
for $y_3\leq y \leq g_2(x)$.
\end{lem}
\begin{proof}{Lemma \ref{lem:tphi}}
We first consider the interval $[y_1,y_2]$. By the definition of the point $y_2$, Lemma 2.4 in \cite{J1} implies that we have an absolute lower bound on $\phi(t)$ for $t\in[y_1,y_2]$, and we know that this interval is of length comparable to $L_1$. Thus, for $y\in[y_1,y_2]$, we have
\begin{align*}
 \tphi(y) \leq C_1L_1 \phi(y).
 \end{align*}
Before considering $y\in[y_2,y_3]$, we first assume that $y\geq y_3$. Here $F(y) \geq 0$, and so $|\phi'(y)|$ is decreasing ($\phi'(y)$ is becoming less negative). Therefore, for $t\in[y_3,y]$, we have the lower bound
\begin{align*}
 \phi(t) \geq \phi(y) + |\phi'(y)|(y-t).
 \end{align*}
 This gives us the bound
\begin{align}  \label{eqn:tphilem1}
 \int_{y_3}^y \phi(t)^{-2} \ud t \leq C_1\phi(y)^{-1}|\phi'(y)|^{-1}.
 \end{align}
 We now want to bound
 \begin{align*}
  \int_{y_2}^{y_3} \phi(t)^{-2} \ud t.
 \end{align*}
 Since $\phi''(y) = F(y)\phi(y)$, we have
 \begin{align*}
 \phi'(y) = \int_{\tilde{y}}^yF(t)\phi(t) \ud t,
 \end{align*}
 where $\phi(y)$ attains its maximum of $1$ at $y=\tilde{y}$. For $t\in[y_2,y_3]$, $F(t) \leq 0$, and so $|\phi'(t)|$ is increasing from $0$, $\phi(t)$ is decreasing from $1$, and $|y_3-y_2| \leq C_1L_1$. Therefore, either $\phi(t)$ is bounded below by an absolute constant or else $|\phi'(y)| \geq C_1L_1^{-1}$. This gives us the bound
 \begin{align} \label{eqn:tphilem2}
   \int_{y_2}^{y_3} \phi(t)^{-2} \ud t \leq C_1L_1. 
   \end{align}
   Combining the bounds in \eqref{eqn:tphilem1} and \eqref{eqn:tphilem2} shows that
   \begin{align*}
   \tphi(y) \leq C_1L_1, 
   \end{align*}
for $y \in[y_2,y_3]$, and
\begin{align*}
\tphi(y) \leq C_1L_1 + C|\phi'(y)|^{-1} 
\end{align*}
for $y\geq y_3$, as required. 

\end{proof}

\subsection{An Estimate For $\paeigx$ At The Boundary}

We can now bound $\paeigx$ at the endpoints of the interval $\Om(x)$. For each fixed $x$, $\eigx$ has zero boundary conditions on $\Om(x)$. However, since the interval $\Om(x)$ will in general depend on $x$, $\pa_x\eigx$ will not necessarily be zero when $y$ is at the end-points of $\Om(x)$. 

We recall from Definition \ref{def:ystar}, that $y^*\geq y_1$ is the first point where $V(x,y) \geq 1+ C^*L_1^{-2}$, for a fixed large constant $C^*$. The upper endpoint of  the interval $\Om(x)$ is equal to $g_2(x)$, and we set
\begin{align} \label{eqn:Mdef}
M \coloneqq g_2(x) - y^*,
\end{align}
which is the distance between the endpoint of $\Om(x)$ and the region where the potential $V(x,y)$ is less than $1+C^*L_1^{-2}$. We can prove a bound on $\pa_x\psi_1^{(x)}(g_2(x))$ in terms of $M$.

\begin{prop} \label{prop:boundary}
For $y = g_2(x)$ equal to the upper endpoint of the interval $\Om(x)$, we have the bound
\begin{align*}
|\alpha| = \left|\pa_x\psi_1^{(x)}(g_2(x))\right| \leq CL_2^{-1}L_1^{-3/2}(L_1+M)\exp(-cML_1^{-1}). 
\end{align*}
We also have an analogous bound for $y$ equal to the lower endpoint of $\Om(x)$.
\end{prop}

\begin{proof}{Proposition \ref{prop:boundary}}
We can view $\eigx$ as a function of two variables on the domain $\Om$, with $\eigx$ identically equal to $0$ on $\pa \Om$. In particular, for those $x$ in the support of the cut-off function $\chi(x)$, we have written the upper boundary of $\Om$ as the graph of the function $y=g_2(x)$, and so $\psi_1^{(x)}(g_2(x))$ is identically zero as a function of $x$. Differentiating this with respect to $x$ gives
\begin{align} \label{eqn:gradbound}
 \pa_x\psi_1^{(x)}(g_2(x)) = -g_2'(x)\pa_y\psi_1^{(x)}(g_2(x)).
\end{align}
Thus, to obtain a bound on $\pa_x\psi_1^{(x)}(g_2(x))$, it is enough to consider $\pa_y\psi_1^{(x)}(y)$, and the slope of $\pa \Om$ at $(x,g_2(x))$.

We remarked in the definition of $\phi(y)$ in Definition \ref{def:phitilde} that the eigenfunction $\eigx$ is equal to a multiple of $\phi(y)$. Since $\phi(y)$ has $L^2(\Om(x))$-norm comparable to $L_1^{1/2}$, whereas $\eigx$ is $L^2(\Om(x))$-normalised, this multiple is comparable to $L_1^{-1/2}$. Thus, by the bound on $\phi'(y)$ from Proposition \ref{prop:phiy}, with $2^k$ comparable to $L_1^3$, we have the bound
\begin{align*}
\left|\pa_y\psi_1^{(x)}(g_2(x))\right| \leq CL_1^{-3/2}\exp\left(-cML_1^{-1}\right).
\end{align*}
Therefore, by \eqref{eqn:gradbound},  to conclude the proof of the proposition it is enough to show that
\begin{align} \label{eqn:boundary1}
|g_2'(x)| \leq C(L_1+M)L_2^{-1},
\end{align}
for an absolute constant $C>0$. Recall the set $\Om_{L_1^{-2}} = \{(x,y)\in\Om: V(x,y) \leq 1 + L_1^{-2}\}$. This is a convex subset of $\Om$ with height comparable to $L_1$ in the $y$-direction, and length comparable to $\tilde{L}_1$ in the $x$-direction. Moreover, for $x$ fixed in the support of $\chi(x)$, we are at a distance at least comparable to $L_2$ from the left and right ends of $\Om_{L_1^{-2}}$. Therefore, if we write the upper boundary of $\Om_{L_1^{-2}}$ of this set as the graph of a function $y=v(x)$, then certainly we have the derivative bound
\begin{align*}
|v'(x)| \leq CL_1L_2^{-1}.
\end{align*}
In particular, if the distance $M$ is bounded above by a multiple of $L_1$, then by convexity, the part of $\pa\Om$ for $x$ contained in the support of $\chi(x)$ has slope bounded by a multiple of $L_1L_2^{-1}$. This gives the desired bound for $g_2'(x)$ in \eqref{eqn:boundary1},
\begin{align*}
|g_2'(x)| \leq CL_1L_2^{-1}
\end{align*}

If the distance $M$ is large compared to $L_1$, then the domain $\Om$ is convex, and contains an ellipse of height comparable to $M$ in the $y$-direction, and length comparable to $L_2$ in the $x$-direction. Thus, the part of $\pa\Om$ with $x$ in the support of $\chi(x)$ has slope bounded by a multiple of $ML_2^{-1}$.  Again we get a bound for $g_2'(x)$, 
\begin{align*}
|g_2'(x)| \leq CML_2^{-1}
\end{align*}
which implies the bound in \eqref{eqn:boundary1}.

This establishes the estimate in \eqref{eqn:boundary1} in all cases, and completes the proof of the proposition.
\end{proof}

We have now established the properties of the functions $\phi(y)$ and $\tphi(y)$ together with the bound required on $\alpha = \pa_x\psi_1^{(x)}(g_2(x))$. Thus, we return to the expression for
\begin{align*}
g(y) = \paeigx - \alpha
\end{align*}
that we derived in Lemma \ref{lem:gexpression}:
\begin{align} \label{eqn:gexpression2}
g(y) - c_0(x)\eigx = \phi(y)\int_{y_1}^y \tphi(t)G(x,t) \ud t + \tphi(y)\int_y^{g_2(x)}\phi(t)G(x,t) \ud t, 
\end{align}
where $G(x,y)$ equals
\begin{align} \label{eqn:Gexpression}
G(x,y)  =  \mu'(x) \eigx - \pa_xV(x,y) \eigx + (V(x,y) - \mu(x))\alpha . 
\end{align}

We will use \eqref{eqn:gexpression2} to obtain the desired bound on $\paeigx$:

\begin{prop} \label{prop:paxeigpointwise}
As usual, for $2^k \leq CL_1^3$, let the intervals $J_k$ be given by
\begin{align*}
J_k = [t_k,t_{k+1}] = \{t\geq y_3: \pa_tV(x,t) \in [2^{-k},2^{-k+1}] \}. 
\end{align*}
We have the pointwise bound 
\begin{align*}
\left| \paeigx - c_0(x)\psi_1^{(x)}(y) \right| \leq F_1(y) + F_2(y)
%\left| \paeigx - c_0(x)\psi_1^{(x)}(y) \right| \leq C_1L_2^{-1}L_1^{-1/2} ,
\end{align*}
for all $y\in \Om(x)$ with $y\geq y_1$. Here $F_1(y)$ is a positive function on $\Om(x)$, with a maximum comparable to $L_2^{-1}L_1^{-1/2}$ and decaying exponentially from this maximum on a length scale comparable to $L_1$ as $y$ moves away from the interval where $V(x,y) \leq 1+ L_1^{-2}$. The function $F_2(y)$  is also a positive function on $\Om(x)$, with a maximum comparable to $L_2^{-1}L_1^{-1/2}$  but it decays exponentially from this  maximum within each interval $J_k$ on a length scale comparable to $2^{k/3}$. We also have the analogous exponential decay estimate on the corresponding intervals as we move away from the \lqu centre' region where $V(x,y) \leq 1 + L_1^{-2}$ in the opposite direction with $y \leq y_1$.
\end{prop}

Before we prove this proposition, let us show how it implies the $L^2(\Om(x))$-bound on $\paeigx$ given in Proposition \ref{prop:paxeigupperbound}: We saw in the proof of Proposition \ref{prop:eigboundinter} that $\paeigx$ and $\eigx$ satisfy the orthogonality relation
\begin{align*}
\int_{\Om(x)} \paeigx \eigx \ud y = 0.
\end{align*}
Thus, since $\eigx$ is $L^2(\Om(x))$-normalised, we have the expression
\begin{align*}
c_0(x) = \int_{\Om(x)} \left(c_0(x)\eigx - \paeigx\right)\eigx \ud y .
\end{align*}
Using the bound on $c_0(x)\eigx - \paeigx$ in Proposition \ref{prop:paxeigpointwise} we obtain
\begin{align} \label{eqn:c0bound}
|c_0(x)| \leq  C_1L_2^{-1}L_1^{-1/2}\int_{\Om(x)} \eigx \ud y  \leq C_1L_2^{-1},
\end{align}
where the final inequality holds since $\eigx$ has $L^2(\Om(x))$-norm equal to $1$, and decays exponentially away from its maximum on a length scale comparable to $L_1$.

Combining this bound on $c_0(x)$ in \eqref{eqn:c0bound} with Proposition \ref{prop:paxeigpointwise}, we see that $\paeigx$ can be bounded by functions $F_1(y) + F_2(y)$ with the same properties as in the statement of Proposition \ref{prop:paxeigpointwise}. This gives us an $L^2(\Om(x))$-bound on $\paeigx$ of the form
\begin{align*}
\int_{\Om(x)} \left(\paeigx \right)^2 \ud y \leq C_1L_2^{-2} + \sum_{2^k \leq CL_1^3} 2^{k/3}L_2^{-2} L_1^{-1} \leq C_1L_2^{-2}.
\end{align*}
This completes the proof of Proposition \ref{prop:paxeigupperbound}.

\end{proof}

Since Proposition \ref{prop:paxeigupperbound} implies the desired upper bound on the eigenvalue $\la$ in Proposition \ref{prop:eigbound}, we just need to prove Proposition \ref{prop:paxeigpointwise}.

\begin{proof}{Proposition \ref{prop:paxeigpointwise}}

From Proposition \ref{prop:boundary} we know that
\begin{align*}
|\paeigx - g(y)| = | \pa_x\psi_1^{(x)}(g_2(x))| = |\alpha| \leq CL_2^{-1}L_1^{-3/2}(L_1+ M)\exp\left(-cML_1^{-1}\right),
\end{align*}
and this bound has the same properties as the function $F_1(y)$ in the statement of the proposition. Therefore, to prove Proposition \ref{prop:paxeigpointwise}, it is enough to show that 
\begin{align*}
g(y) - c_0(x)\psi_1^{(x)}(y)
\end{align*}
has the desired bounds. 

To do this, we want to bound the right hand side of \eqref{eqn:gexpression2}, which contains the functions $\phi(y)$, $\tphi(y)$ together with $G(x,y)$. The two remaining functions which we have not discussed above are the functions $\mu'(x)$ and $\pa_xV(x,y)$ appearing in $G(x,y)$. Therefore, let us prove two simple lemmas concerning these functions, and then we will be in a position to bound \eqref{eqn:gexpression2}.

\begin{lem} \label{lem:muprime}
Let $x$ be in the support of the cut-off function $\chi(x)$. Then, we have the bound
\begin{align*}
|\mu'(x)| \leq C_1L_2^{-3},
\end{align*}
for an absolute constant $C_1>0$.
\end{lem}
\begin{proof}{Lemma \ref{lem:muprime}}
We recall from Lemma \ref{lem:mu(x)convex} that the function $\mu(x)$ is a convex function of $x$. Moreover, by the definition of the parameter $L_2$, we know that $\mu(x)$ varies by $L_2^{-2}$ for $x$ in an interval of length at least $L_2$. Since the support of $\chi(x)$ is contained within the middle half of this interval, we immediately obtain the bound
\begin{align*}
|\mu'(x)| \leq C_1L_2^{-3},
\end{align*}
by convexity.
\end{proof}

\begin{lem} \label{lem:paxV}
Let $x$ be in the support of $\chi(x)$, and as in Definition \ref{def:ystar} let $y=y^*$ be the first point where $V(x,y) \geq 1+C^*L_1^{-2}$. Then, for $y\geq y^*$, 
\begin{align*}
 |\pa_xV(x,y)| \leq C_1(|y-y^*|+L_1)L_2^{-1}|\pa_yV(x,y)|,
 \end{align*}
and for $y_1\leq y\leq y^*$,
\begin{align*}
 |\pa_xV(x,y)| \leq C_1L_1^{-2}L_2^{-1} + C_1L_1L_2^{-1}|\pa_yV(x,y)|.
 \end{align*}
\end{lem}
\begin{proof}{Lemma \ref{lem:paxV}}
Given, $c$, let $y=f(x)$ be  a parameterisation of the upper part of the level set $\{(x,y) \in \Om : V(x,y) = c\}$, so that
\begin{align*}
 V(x,f(x)) = c = \text{constant}.
 \end{align*}
 Differentiating this with respect to $x$, we see that
\begin{align} \label{eqn:paxV1}
 \pa_xV(x,f(x)) = -f'(x)\pa_yV(x,f(x)).
 \end{align}
Assume first that $y=f(x) \geq y^*$. The sublevel set $\{(x,y)\in\Om :V(x,y) \leq 1+ C^*L_1^{-2}\}$ is convex with height comparable to $L_1$ in the $y$-direction and length comparable to $\tilde{L}_1$ in the $x$-direction, and $x$ is at distance comparable to $L_2$ from the ends of this set. Thus, by the convexity of the sublevel sets, we certainly have a bound on the slope of 
\begin{align*}
|f'(x)| \leq C_1(|y-y^*|+L_1)L_2^{-1}.
\end{align*}
Using this bound in the right hand side of \eqref{eqn:paxV1} gives the desired bound for $y\geq y^*$.

We now suppose that $y_1 \leq y=f(x) \leq y^*$.  If $y$ is in the middle half of the interval $\{t: V(x,t) \leq 1+L_1^{-2}\}$, then we certainly have the bound
\begin{align*}
 |\pa_xV(x,f(x))| \leq C_1L_1^{-2}L_2^{-1},
 \end{align*}
by the convexity of the potential $V(x,y)$. For the remaining points $(x,f(x))$ of interest, we can again use the shape of the level set to obtain the desired bound
\begin{align*}
 |\pa_xV(x,f(x))| \leq C_1L_1^{-2}L_2^{-1} + C_1L_1L_2^{-1}|\pa_yV(x,f(x))|.
 \end{align*}
 This is because for these points we can find a direction $\mathbf{e}$, such that the directional derivative of $V$ at $(x,f(x))$ is bounded by $L_1^{-2}L_2^{-1}$, and this direction makes an angle comparable to $L_1L_2^{-1}$ with the $x$-axis.

\end{proof}

Combining Lemmas \ref{lem:muprime} and \ref{lem:paxV}, we see from \eqref{eqn:Gexpression} that
\begin{align} \label{eqn:Gexpression2} \nonumber
|G(x,t)| & \leq C_1L_1^{-2}L_2^{-1}\psi_1^{(x)}(t) + C_1(|t-y^*|+L_1)L_2^{-1}|\pa_tV(x,t)|\psi_1^{(x)}(t) + |V(x,t) - \mu(x)|\alpha \\
& \leq C_1L_1^{-5/2}L_2^{-1}\phi(t) + C_1(|t-y^*|+L_1)L_1^{-1/2}L_2^{-1}|\pa_tV(x,t)|\phi(t) + |V(x,t) - \mu(x)|\alpha.
\end{align}
The final inequality comes from
\begin{align*}
 \psi_1^{(x)}(t) \leq C_1L_1^{-1/2}\phi(t),
\end{align*}
which holds since $\psi_1^{(x)}(t)$ is $L^2(\Om(x))$-normalised, whereas $\phi(t)$ has $L^{2}(\Om(x))$-norm comparable to $L_1^{1/2}$.

Everything is now set up to show that the two integrals in \eqref{eqn:gexpression2} have the bounds required in the statement of Proposition \ref{prop:paxeigpointwise}.

\subsection{A Bound on $\phi(y)\int_{y_1}^y \tphi(t)G(x,t) \ud t$}

We start by considering the first integral in \eqref{eqn:gexpression2},
\begin{align} \label{eqn:key1}
 \left| \phi(y)\int_{y_1}^y\tphi(t)G(x,t) \ud t \right| \leq \phi(y)\int_{y_1}^y \tilde{\phi}(t)|G(x,t)| \ud t .
 \end{align}
 Using \eqref{eqn:Gexpression2}, it is enough to bound
 \begin{align} \label{eqn:key2}
\phi(y) \int_{y_1}^y \tphi(t)  \left(C_1L_1^{-5/2}L_2^{-1}\phi(t) + C_1(|t-y^*|+L_1)L_1^{-1/2}L_2^{-1}|\pa_tV(x,t)|\phi(t) + |V(x,t) - \mu(x)|\alpha \right) \ud t.
\end{align}
We now bound the three terms in equation \eqref{eqn:key2}.

\begin{lem} \label{lem:G1}
We have a bound on the first term in \eqref{eqn:key2},
\begin{align*}
\phi(y) \int_{y_1}^y \tphi(t)L_1^{-5/2}L_2^{-1} \phi(t) \ud t \leq C_1L_1^{-1/2}L_2^{-1}.
\end{align*}
\end{lem}
\begin{rem}
We will see in the proof of the lemma that the function decays exponentially from its maximum  away from the region where $V(x,y) \leq 1+ L_1^{-2}$ on a length scale comparable to $L_1$. Therefore we can include this term in the function $F_1(y)$ in the statement of Proposition \ref{prop:paxeigpointwise}.
\end{rem}

\begin{proof}{Lemma \ref{lem:G1}}
By Lemma \ref{lem:tphi}, we can bound the left hand side by
\begin{align*}
  \phi(y)\int_{y_1}^{y_3} C_1L_1L_1^{-5/2}L_2^{-1} \phi(t)\ud t +  \phi(y)\int_{y_3}^{y} (C_1L_1+C_1|\phi'(t)|^{-1})L_1^{-5/2}L_2^{-1} \phi(t)\ud t.
  \end{align*}
  Using Proposition \ref{prop:phi} we have the bound, $\phi(t) \leq 2^{k/3}|\phi'(t)| \leq C_1L_1|\phi'(t)|$ for $t\in J_k$, and so these integrals can be bounded by
\begin{align} \label{eqn:G1a}
C_1 \phi(y)\int_{y_1}^y L_2^{-1}L_1^{-3/2} \ud t.
\end{align}
The eigenfunction $\phi(y)$ has a maximum of $1$, and decays exponentially away from this maximum on a length scale comparable to $L_1$. Thus, we can bound \eqref{eqn:G1a} by $C_1L_1^{-1/2}L_2^{-1}$ as required, and it also has the decay properties of the function $F_1(y)$.
\end{proof}

\begin{lem} \label{lem:G2}
We have a bound on the second term in \eqref{eqn:key2},
\begin{align*}
 \phi(y) \int_{y_1}^y \tphi(t)(|t-y^*|+L_1)L_1^{-1/2}L_2^{-1}|\pa_tV(x,t)|\phi(t) \ud t \leq C_1L_1^{-1/2}L_2^{-1}.
 \end{align*}
\end{lem}
\begin{rem}
We will again see in the proof of the lemma that the function decays exponentially from its maximum on a length scale comparable to $2^{k/3}$ within each interval $J_k$. Therefore we can include this term in the function $F_2(y)$ in the statement of Proposition \ref{prop:paxeigpointwise}.
\end{rem}

\begin{proof}{Lemma \ref{lem:G2}}
We first consider the part of this integral over $[y_1,y_3]$. Here, $\tphi(t) \leq C_1L_1$, and by the convexity of the potential
\begin{align*}
 \int_{y_1}^{y_3} |\pa_tV(x,t)| \ud t \leq 2C_1L_1^{-2}.
 \end{align*}
Therefore, we immediately obtain a bound of $C_1L_1^{-1/2}L_2^{-1}\phi(y)$. This is certainly at most $C_1L_1^{-1/2}L_2^{-1}$, and by the properties of $\phi(y)$ it also has the decay properties of the function $F_2(y)$.

We now consider the part of the integral over $[y_3,y]$. Let $y$ be in the interval $J_{k^*}$ for some $k^*$, where as usual the intervals $J_k$ are as in \eqref{eqn:Jk}. We decompose the integral between $y_3$ and $y$ as an integral over the relevant intervals $J_k$ where $k\geq k^*$. 

By Proposition \ref{prop:phi}, 
\begin{align*}
 \phi(t) \leq |\phi'(t)|2^{k/3}, 
 \end{align*}
and so using the bound on $\tphi(t)$ from Lemma \ref{lem:tphi}, to estimate the contribution to the integral from $J_k$, we have to bound
\begin{multline}  \label{eqn:G2}
 \phi(y) \int_{J_k} \phi(t)(L_1+|\phi'(t)|^{-1})(|t-y^*|+L_1) L_1^{-1/2}L_2^{-1}|\pa_tV(x,t)| \ud t \\
 \leq C_1\phi(y)\int_{J_k} 2^{k/3}|\phi'(t)| (L_1+|\phi'(t)|^{-1})(|t-y^*|+L_1) L_1^{-1/2} L_2^{-1}2^{-k} \ud t \\
 \leq  C_12^{-2k/3}\phi(y)\int_{J_k} (|t-y^*|+L_1)L_1^{-1/2}L_2^{-1} \ud t .
 \end{multline}
 
 Using Proposition \ref{prop:phi} again, we find that for any $k\geq k^*$,
\begin{align*}
 \phi(y) \leq \phi(t_{k^*})\exp(-(y-t_{k^*})2^{-k^*/3}/10)\leq 2^{k^*/3}|\phi'(t_{k^*})|\exp(-(y-t_{k^*})2^{-k^*/3}/10).
 \end{align*}
By Corollary \ref{cor:phiy}, we can bound the factor of $|\phi'(t_{k^*})|$ as
\begin{align*}
 |\phi'(t_{k^*})| \leq CL_1^{-1} \exp(-c|t-t_k|/2^{k/3}) \exp(-c|t_k-y^*|/L_1).
 \end{align*}
Inserting these estimates into the integral in \eqref{eqn:G2} and integrating over the interval $J_k$, we have the bound
\begin{align*}
 C\exp(-(y-t_{k^*})2^{-k^*/3}/10)2^{k^*/3}2^{-k/3}L_1^{-1/2}L_2^{-1}.
 \end{align*}
 Summing over $k\geq k^*$ gives a bound for the integral over $y_3\leq t\leq y$ of the form
\begin{align*}
C_1 L_1^{-1/2}L_2^{-1}\exp(-(y-t_{k^*})2^{-k^*/3}/10).
 \end{align*}
Note that this quantity is bounded by a multiple of $L_1^{-1/2}L_2^{-1}$, and has the required decay properties that we can include it in the function $F_2(y)$.

\end{proof}

\begin{lem} \label{lem:G3}
We have a bound on the final term in \eqref{eqn:key2},
\begin{align} \label{eqn:G3a}
 \phi(y) \int_{y_1}^y \tphi(t)|V(x,t) - \mu(x)||\alpha| \ud t \leq C_1L_1^{-1/2}L_2^{-1}. 
 \end{align}
\end{lem}

\begin{rem}
We will see  that the function decays exponentially from its maximum  away from the region where $V(x,y) \leq 1+ L_1^{-2}$ on a length scale comparable to $L_1$. Therefore we can include this term in the function $F_1(y)$ in the statement of Proposition \ref{prop:paxeigpointwise}.
\end{rem}

\begin{proof}{Lemma \ref{lem:G3}}
We recall from Proposition \ref{prop:boundary} that we have an estimate on the boundary value of $\paeigx$ of the form
\begin{align} \label{eqn:G3b}
 |\alpha|  = |\pa_x\psi_1^{(x)}(g_2(x))| \leq CL_2^{-1}L_1^{-3/2}(L_1+M)\exp(-cML_1^{-1}).
 \end{align}
Here $M$ is the distance from $g_2(x)$ to the point $y^*$  where $V(x,y^*) =  1+C^*L_1^{-2}$.

For the part of the integral in \eqref{eqn:G3a} over $[y_1,y_3]$, we know that $|V(x,t) - \mu(x)| \leq C_1L_1^{-2}$, $|y_3-y_1| \leq C_1L_1$ and $\tphi(t) \leq C_1L_1$. Combining this with the bound on $\alpha$ from \eqref{eqn:G3b},  immediately gives us the desired bound of $L_2^{-1}L_1^{-1/2}\exp(-cML_1^{-1})$ for this part of the integral in \eqref{eqn:G3a}.

For $t\geq y_3$, we decompose $[y_3,y]$ into the intervals $\tilde{J}_k$ as in \eqref{eqn:Jktilde}:
\begin{align*}
\tilde{J}_k = [\tilde{t}_k,\tilde{t}_{k+1}]  = \{ t\geq y_3: V(x,t) - \min_{t}V(x,t) \in [2^{-2k/3}, 2^{-2(k-1)/3}]\} . 
\end{align*}
Since $\mu(x) \geq \min_{t}V(x,t)$, on $\tilde{J}_k$ we know that
\begin{align*}
|V(x,t) - \mu(x)| \leq 2^{-2k/3}.
\end{align*}
So, for the part of the integral in \eqref{eqn:G3a} over $\tilde{J}_k$, combining this with the bound on $\alpha$ in \eqref{eqn:G3b} and the usual bound on $\tphi(t)$ from Lemma \ref{lem:tphi}, we have 
\begin{align} \label{eqn:G3c} \nonumber
\phi(y) & \int_{\tilde{J}_k} \tphi(t)|V(x,t) - \mu(x)||\alpha| \ud t  \\
& \leq C_1L_2^{-1}L_1^{-3/2} \phi(y)\int_{\tilde{J}_k}(L_1+|\phi'(t)|^{-1})2^{-2k/3}(L_1+M)\exp(-cML_1^{-1}) \ud t.
\end{align}
Similarly to  the proof of Lemma \ref{lem:G2}, let us assume that $y \in \tilde{J}_{k^*}$ for some $k^*$. Then,
using Proposition \ref{prop:phi} and then Proposition  \ref{prop:phiy} twice, we obtain
\begin{align*}
\phi(y) \leq C_12^{k^*/3}|\phi'(y)| & \leq C_12^{k^*/3}|\phi'(t_{k^*})| \exp\left(-c|y-t_{k^*}|2^{-k^*/3}\right) \\
& \leq C_12^{k^*/3}|\phi'(t)|  \exp\left(-c|y-t_{k^*}|2^{-k^*/3}\right)  \exp\left(-c|t_{k^*}-t|2^{-k/3}\right)
\end{align*}
Inserting this bound for $\phi(y)$ into the right hand side of \eqref{eqn:G3c} and integrating over $\tilde{J}_k$ gives us the bound for the part of the integral over $\tilde{J}_{k}$ of
\begin{align*}
C_12^{k^*/3}2^{-k/3}L_2^{-1}L_1^{-1/2}\exp(-cML_1^{-1}/2). 
\end{align*}
We finally sum over those $k$ with $k \geq k^*$ to get the desired bound on the part of the integral \eqref{eqn:G3a} with $y_3 \leq t \leq y$.
\end{proof}

Combining Lemmas \ref{lem:G1}, \ref{lem:G2} and \ref{lem:G3}, we see that the part of $g(y) - c_0(x)\eigx$ coming from
\begin{align*}
\phi(y)\int_{y_1}^{y}\tphi(t)G(x,t) \ud t
\end{align*}
has the bounds required in Proposition \ref{prop:paxeigpointwise}.

 Therefore to finish the proof of Proposition \ref{prop:paxeigpointwise} we need to establish the analogous estimates for the other part of $g(y) - c_0(x)\eigx$ in \eqref{eqn:gexpression2},
\begin{align*}
 \tphi(y)\int_y^{g_2(x)}\phi(t)G(x,t) \ud t .
 \end{align*}

\subsection{A Bound on $\tphi(y)\int_y^{g_2(x)}\phi(t)G(x,t) \ud t$}

The estimates for the various parts of this integral will be similar to the estimates we used above. However, there will be places where we have to use different methods to obtain the desired bounds. 

We want to bound
\begin{align} \label{eqn:key1b}
 \left| \tphi(y)\int_{y}^{g_2(x)}\phi(t)G(x,t) \ud t \right| \leq \tphi(y)\int_{y}^{g_2(x)} \phi(t)|G(x,t)| \ud t .
 \end{align}
 We again use \eqref{eqn:Gexpression2} to bound this by
 \begin{align} \label{eqn:key2b}
\tphi(y) \int_{y}^{g_2(x)} \phi(t)  \left(C_1L_1^{-5/2}L_2^{-1}\phi(t) + C_1(|t-y^*|+L_1)L_1^{-1/2}L_2^{-1}|\pa_tV(x,t)|\phi(t) + |V(x,t) - \mu(x)|\alpha \right) \ud t,
\end{align}
and we split this into three terms that we need to estimate.

\begin{lem} \label{lem:G1a}
We have a bound on the first term in \eqref{eqn:key2b}
\begin{align*}
\tphi(y) \int_{y}^{g_2(x)} \phi(t)  L_1^{-5/2}L_2^{-1}\phi(t) \ud t \leq C_1L_1^{-1/2}L_2^{-1} 
\end{align*}
\end{lem}
\begin{rem}
The function also decays exponentially from its maximum away from the region where $V(x,y) \leq 1+ L_1^{-2}$ on a length scale comparable to $L_1$. Therefore we can include this term in the function $F_1(y)$ in the statement of Proposition \ref{prop:paxeigpointwise}.
\end{rem}
\begin{proof}{Lemma \ref{lem:G1a}}
We know that $\tphi(y) \leq \tphi(t)$, and $\phi(t)$ decays exponentially on a length scale comparable to $L_1$ as we move away from $y^*$. Therefore, this bound follows in a very straightforward manner.
\end{proof}

Before bounding the second term in \eqref{eqn:key2b}, we first want to establish the following lemma.
\begin{lem} \label{lem:energy}
 For any $ \ty\geq y_3$, we have the bound
 \begin{align*}
  \int_{\ty}^{g_2(x)}\phi(t)^2 \pa_tV(x,t) \ud t \leq (\phi'(\ty))^{1/2}.
  \end{align*}
\end{lem}
\begin{proof}{Lemma \ref{lem:energy}}
To prove this lemma, we will consider the \lqu energy'
\begin{align} \label{eqn:energy}
\mathcal{E}(t) \coloneqq (\phi'(t))^2 - F(x,t)\phi(t)^2 .
\end{align}
Differentiating $\mathcal{E}(t)$ we find that
\begin{align*}
\mathcal{E}'(t) = 2\phi'(t)(\phi''(t) - F(x,t)\phi(t)) - \pa_tF(x,t)\phi(t)^2 = -  \pa_tF(x,t)\phi(t)^2,
\end{align*}
where the final equality holds because $\phi''(t) = F(x,t)\phi(t)$. Since $F(x,t) = V(x,t)-\mu(x)$, we have
\begin{align*}
\pa_tF(x,t) =\pa_tV(x,t),
\end{align*}
and so
\begin{align*}
\int_{\ty}^{g_2(x)} \pa_tV(x,t) \phi(t)^2 \ud t = -\int_{\ty}^{g_2(x)} \mathcal{E}'(t) \ud t = \mathcal{E}(\ty) - \mathcal{E}(g_2(x)). 
\end{align*}
Since $F(x,t) \geq 0$ for $t\geq y_3$, we know that
\begin{align*}
\mathcal{E}(\ty) = (\phi'(\ty))^2 - F(x,\ty)\phi(\ty)^2 \leq (\phi'(\ty))^2.
\end{align*}
Thus, to finish the proof of the lemma we need to show that $\mathcal{E}(g_2(x)) \geq 0$. We know that
\begin{align*}
\mathcal{E}(g_2(x)) \geq -F(x,g_2(x))\phi(g_2(x))^2,
\end{align*}
and that $\phi(g_2(x)) = 0$. However, we are not assuming that the potential $V(x,y)$ remains bounded as $y$ approaches $g_2(x)$, and so we cannot immediately deduce that $F(x,g_2(x))\phi(g_2(x))^2 = 0$. Instead we argue as follows. The function $\phi'(t)$ is in $L^{\infty}(\Om(x))$, and this has two consequences. First, the eigenfunction $\phi(y)$ decays at least linearly to $0$ at $y=g_2(x)$. It also means that $\phi''(y)$ is in $L^1(\Om(x))$, and hence $F(x,y)\phi(y)$ is in $L^1(\Om(x))$. This means that $F(x,y)\phi(y)$ cannot grow as fast as $(y-g_2(x))^{-1}$ as we approach the boundary and so
\begin{align*}
\liminf_{y\to g_2(x)} \phi(y)F(x,y)\phi(y) = 0.
\end{align*}
This implies that $\mathcal{E}(g_2(x)) \geq0$ and concludes the proof of the lemma.
\end{proof}
\begin{rem}
This energy $\mathcal{E}(t)$ has also been used in \cite{GJ2} in their proof of Theorem 2.1 (B). There they obtain a pointwise estimate comparing  the first eigenfunction of the two dimensional domain with the first eigenfunction of the associated ordinary differential operator.
\end{rem}

We can now bound the contribution from the second term in $G(x,t)$.
\begin{lem} \label{lem:G2b}
We have a bound on the second term in \eqref{eqn:key2b},
\begin{align*}
 \tphi(y) \int_{y}^{g_2(x)} \phi(t)^2(|t-y^*|+L_1)L_1^{-1/2}L_2^{-1}|\pa_tV(x,t)| \ud t \leq C_1L_1^{-1/2}L_2^{-1}.
 \end{align*}
\end{lem}
\begin{rem}
We will see in the proof that the function also decays exponentially from its maximum away from the region where $V(x,y) \leq 1+ L_1^{-2}$  on a length scale comparable to $L_1$. Therefore we can include this term in the function $F_1(y)$ in the statement of Proposition \ref{prop:paxeigpointwise}.
\end{rem}
\begin{proof}{Lemma \ref{lem:G2b}}
If $y \leq y_3$, we first consider the part of the integral where $t$ lies in the interval $[y,y_3]$ of length at most $C_1L_1$. In this case, we know that $\tphi(y) \leq C_1L_1$ and the estimates follow easily.

For $t\geq y_3$, we first consider the integral between $\ty$ and $\ty+L_1$, where $\ty$ is some point with $\ty\geq y_3$ and $\ty\geq y$. Since $\pa_tV(x,t)\geq 0$ here, we have
\begin{align} \label{eqn:G2b1} \nonumber
 \tphi(y) \int_{\ty}^{\ty+L_1} & \phi(t)^2(|t-y^*|+L_1)L_2^{-1}|\pa_tV(x,t)| L_1^{-1/2}\ud t \\
 & \leq C_1\tphi(y)(|\ty-y^*|+L_1)L_2^{-1}L_1^{-1/2}\int_{\ty}^{g_2(x)}\phi(t)^2 \pa_tV(x,t)  \ud t.
\end{align}
 Applying Lemma \ref{lem:energy}, we can bound the right hand side of \eqref{eqn:G2b1} by
 \begin{align*}
 C_1\tphi(y) (|\ty-y^*|+L_1)L_2^{-1}L_1^{-1/2}(\phi'(\ty))^2 .
 \end{align*}
 Lemma \ref{lem:tphi} shows that
 \begin{align*}
 \tphi(\ty)|\phi'(\ty)| \leq C_1,
 \end{align*}
 and using Proposition \ref{prop:phiy} with $2^k$ comparable to $L_1^{3}$, we have the derivative bound
 \begin{align*}
 |\phi'(\ty)| \leq C_1L_1^{-1}\exp\left(-c|\ty-y^*|/L_1\right).
 \end{align*}
 Thus the right hand side of \eqref{eqn:G2b1} has the bound
 \begin{align*}
 C_1L_2^{-1}L_1^{-1/2} \exp\left(-c|\ty-y^*|/L_1\right).
 \end{align*}
 Summing over $\ty$ between $y$ and $g_2(x)$ at intervals of length comparable to $L_1$ then gives the desired bound.

\end{proof}

We finally have to bound the contribution from the third term in $G(x,t)$.

\begin{lem} \label{lem:G3b}
We have a bound on the third term in \eqref{eqn:key2b}
\begin{align} \label{eqn:G3b1}
 \tphi(y) \int_{y}^{g_2(x)} \phi(t)|V(x,t) - \mu(x)||\alpha| \ud t \leq C_1L_1^{-1/2}L_2^{-1}.
 \end{align}
\end{lem}
\begin{rem}
As for the previous two lemmas, we will see in the proof that the function also decays exponentially from its maximum away from the region where $V(x,y) \leq 1+ L_1^{-2}$ on a length scale comparable to $L_1$. Therefore we can include this term in the function $F_1(y)$ in the statement of Proposition \ref{prop:paxeigpointwise}.
\end{rem}
\begin{proof}{Lemma \ref{lem:G3b}}
From Proposition \ref{prop:boundary} we have the estimate on $\alpha$ of the form
\begin{align} \label{eqn:G3b1}
 |\alpha|  = |\pa_x\psi_1^{(x)}(g_2(x))| \leq CL_2^{-1}L_1^{-3/2}(L_1+M)\exp(-cML_1^{-1}).
 \end{align}
 For the part of the integral in \eqref{eqn:key2} for $t$ between $y$ and $y_3$, we know that $|V(x,t) - \mu(x)|$ is at most $C_1L_1^{-2}$, and $\tphi(y) \leq C_1L_1$. Thus, we immediately get a bound of
\begin{align*}
C_1L_1^{-1/2}L_2^{-1}\exp(-cML_1^{-1})
\end{align*}
for this part.

For $t \geq y_3$, we know that $F(x,t) = V(x,t) - \mu(x) \geq 0$. Thus, we can bound this part of the integral in \eqref{eqn:key2} by
\begin{align} \label{eqn:G3b2}
 C_1\tphi(y)L_2^{-1}L_1^{-3/2}(L_1+M)\exp(-cML_1^{-1}) \int_{\max\{y_3,y\}}^{g_2(x)} \phi(t)(V(x,t) - \mu(x)) \ud t.
\end{align} 
Since
\begin{align*}
\phi''(t) = (V(x,t) - \mu(x))\phi(t),
\end{align*}
and $|\phi'(t)|$ is decreasing for $t \geq y_3$, we find that \eqref{eqn:G3b2} can be bounded by
\begin{align*}
 C_1\tphi(y)L_2^{-1}L_1^{-3/2}(L_1+M)\exp(-cML_1^{-1}) |\phi'(y)| \leq C_1L_1^{-1/2}L_2^{-1}\exp(-cML_1^{-1}/2),
 \end{align*}
 where the last inequality comes from Lemma \ref{lem:tphi} as usual. This concludes the proof of the lemma.

\end{proof}

We recall that
\begin{align*}
g(y) - c_0(x)\eigx = \phi(y)\int_{y_1}^y \tphi(t)G(x,t) \ud t + \tphi(y)\int_y^{g_2(x)}\phi(t)G(x,t) \ud t.
\end{align*}
Then by the bounds on the right hand side in Lemmas \ref{lem:G1}, \ref{lem:G2}, \ref{lem:G3} and Lemmas \ref{lem:G1a}, \ref{lem:G2b}, \ref{lem:G3b}, we have shown that
\begin{align} \label{eqn:eigenvaluefinal}
\left| g(y) - c_0(x)\eigx \right| \leq F_1(y) + F_2(y).
\end{align}
Here the functions $F_1(y)$ and $F_2(y)$ have the desired properties from the statement of Proposition \ref{prop:paxeigpointwise}. As we remarked at the beginning of the proof, by the bound on $\alpha$ that we obtained in Proposition \ref{prop:boundary}, the estimate in \eqref{eqn:eigenvaluefinal} is sufficient to conclude the proof of Proposition \ref{prop:paxeigpointwise}.

\end{proof}

After the statement of Proposition \ref{prop:paxeigpointwise}, we showed that this implied Proposition \ref{prop:paxeigupperbound} and the bound
\begin{align*}
\int_{\Om(x)} \left( \paeigx \right)^2 \ud y \leq C_1L_2^{-2}.
\end{align*}
Combining this with the estimate on the first eigenvalue $\la$ from Proposition \ref{prop:eigboundinter}
\begin{align*}
\la \leq \mu + \int_{\Om} \chi(x)^2\left( \paeigx \right)^2 \ud x \ud y + C_1L_2^{-2}
\end{align*}
 gives
\begin{align*}
\la \leq \mu + CL_2^{-2}.
\end{align*}
This completes the proof of the upper bound on $\la$ in Proposition \ref{prop:eigbound}. 
\end{proof}

By Propositions \ref{prop:lalower} and \ref{prop:eigbound}, we see that the first eigenvalue $\la$ satisfies
\begin{align*}
\mu \leq \la \leq \mu + CL_2^{-2},
\end{align*}
and so we have established Theorem \ref{thm:eigenvalue}.

\section{$L^2(\Om)$ Bounds For The First Eigenfunction $u(x,y)$} \label{sec:L2}

Now that we have established the improved eigenvalue bound on $\la$ in Theorem \ref{thm:eigenvalue}, we want to use it to study the corresponding eigenfunction $u(x,y)$. We recall that $u(x,y)$ satisfies 
\begin{eqnarray*}
    \left\{ \begin{array}{rlcc}
    (-\Delta_{x,y} + V(x,y))u(x,y) & = \la u(x,y) && \text{in } \Om \\
   u(x,y) & = 0&& \text{on } \pa \Om,
    \end{array} \right.
\end{eqnarray*}
and is normalised to be positive inside $\Om$ with a maximum of $1$. Our main aim is to prove Theorem \ref{thm:shape} and show that the level sets $\{(x,y)\in\Om: u(x,y) = c\}$ have lengths comparable to $L_2$ and $L_1$ in the $x$ and $y$-directions respectively, whenever $c$ is bounded away from $0$ and $1$.

Before we prove this theorem, in this section we will first establish an $L^2(\Om)$-bound for $u(x,y)$. More precisely, we will prove the following proposition.

\begin{prop} \label{prop:uL2bound}
There exists an absolute constant $C>0$ such that
\begin{align*}
\int_{\Om} u(x,y)^2 \ud x \ud y \leq CL_1L_2.
\end{align*}
\end{prop}
\begin{rem}
Note that this $L^2(\Om)$ bound is consistent with the shape of the level sets described in Theorem \ref{thm:shape}. We will use the eigenvalue bound on $\la$ from Theorem \ref{thm:eigenvalue} in a critical way in the proof.
\end{rem}
\begin{proof}{Proposition \ref{prop:uL2bound}}
Before beginning the proof of this proposition, we make the following definition.
\begin{defn} \label{def:H}
We define the function $H(x)$ by
\begin{align*}
H(x) \coloneqq \int_{\Om(x)}u(x,y)^2 \ud y.
\end{align*}
That is, $H(x)$ is equal to the square of the $L^2(\Om(x))$-norm of the cross-section of the eigenfunction $u(x,\cdot)$.
\end{defn}
To prove Proposition \ref{prop:uL2bound} we will first study the rate at which the function $H(x)$ decays from its maximum, and we will then prove an estimate for the maximum of $H(x)$.

To study the decay of $H(x)$, we prove a Carleman-type inequality. For the convex function $\mu(x)$ let $x^*$ be a point where it achieves its minimum of $\mu^*$. We now prove:
\begin{prop} \label{prop:H(x)decay}
For any $x$ we have the differential inequality,
\begin{align*}
 H''(x) \geq 2(\mu(x) - \la)H(x).
 \end{align*}
In particular, for $|x-x^*| \geq CL_2$,  with $C$ a sufficiently large absolute constant, we have
\begin{align*}
 H''(x) \geq \frac{1}{L_2^2}H(x).
 \end{align*} 
\end{prop}
\begin{rem}
This type of Carleman inequality has been used frequently in the study of the ground state Dirichlet eigenfunction of Schr\"odinger operators. For example, in Lemma 3.9 \cite{GJ2} it has been used to establish the exponential decay of the first Fourier mode of the ground state eigenfunction of the two dimensional convex domain. This first Fourier mode comes from a Fourier decomposition of the cross-section of the domain at each fixed $x$. A similar argument has also been used in Section 3 of \cite{FS1} to study the decay of the $L^2$-norm of the cross-section at $x$ of the eigenfunction for a two dimensional domain which is periodic in the $x$-direction and with height in the $y$-direction depending on a small parameter $\eps>0$.
\end{rem}
\begin{proof}{Proposition \ref{prop:H(x)decay}}
The eigenfunction $u(x,y)$ is equal to $0$ when $y$ is at the endpoints of the interval $\Om(x)$. This allows us to differentiate $H(x)$ twice and pass the derivative inside the integral to obtain
\begin{align*}
 H''(x) &= 2\int_{\Om(x)} u(x,y)\pa_{x}^2u(x,y) + (\pa_xu(x,y))^2 \ud y \\
 &= 2\int_{\Om(x)} (V(x,y) - \la)u(x,y)^2 - u(x,y)\pa_y^2u(x,y) +(\pa_xu(x,y))^2 \ud y. 
\end{align*}
Integrating by parts one time in $y$ in the term containing a factor of $\pa_y^2u(x,y)$, we can rewrite this as
\begin{align} \label{eqn:H(x)decayA} \nonumber
 H''(x) & = 2\int_{\Om(x)} (V(x,y) - \la)u(x,y)^2  + (\pa_yu(x,y))^2 + (\pa_xu(x,y))^2 \ud y \\
 & \geq 2\int_{\Om(x)} (V(x,y) - \la)u(x,y)^2  + (\pa_yu(x,y))^2 \ud y 
 \end{align}
 Since $\mu(x)$ is the first eigenvalue of the operator
 \begin{align*}
 \mathcal{L}(x) = - \frac{d^2}{dy^2} + V(x,y),
 \end{align*}
 and $u(x,\cdot)$ vanishes at the endpoints of $\Om(x)$, \eqref{eqn:H(x)decayA} gives us the lower bound
\begin{align} \label{eqn:H(x)decay1}
 H''(x) \geq 2(\mu(x) - \la)\int_{\Om(x)} u(x,y)^2 \ud y = 2(\mu(x)-\la)H(x) . 
 \end{align}
Since $\mu(x^*) = \mu^*$ is the minimum value of the function $\mu(x)$, by the definition of the length scale $L_2$, we know that 
\begin{align*}
 |\mu(x^*) - \mu| \leq C_1L_2^{-2}.
 \end{align*}
 Thus, applying Theorem \ref{thm:eigenvalue}, we have the bound
\begin{align} \label{eqn:H(x)decay2}
 |\la - \mu(x^*)| \leq C_1L_2^{-2}.
 \end{align}
 The function $\mu(x)$ increases from its minimum by $L_2^{-2}$ as $x$ varies in an interval of length comparable to $L_2$ from $x^*$. Moreover, $\mu(x)$ is a convex function. Therefore, provided we choose $C>0$ sufficiently large, we have
 \begin{align} \label{eqn:H(x)decay3}
 \mu(x) - \mu(x^*) \geq (C_1+1)L_2^{-2}
 \end{align}
 whenever $x$ satisfies $|x-x^*| \geq CL_2$. Combining the inequalities in \eqref{eqn:H(x)decay2} and \eqref{eqn:H(x)decay3} shows that
 \begin{align*}
 \mu(x) - \la \geq L_2^{-2},
 \end{align*}
 and using this bound in \eqref{eqn:H(x)decay1} gives
 \begin{align*}
  H''(x) \geq 2L_2^{-2}H(x) 
 \end{align*}
as required.
\end{proof}

Before giving a corollary of this proposition, we recall the generalised maximum principle.
\begin{prop} \label{prop:GMP}
Suppose that the functions $v_1$ and $v_2$ satisfy
\begin{align*}
\Delta v_1 + c(x)v_1 = 0,  \qquad \Delta v_2 + c(x)v_2 \leq 0,
\end{align*}
in a bounded domain $D$, where $c(x)$ is a continuous function. If in addition $v_1$ and $v_2$ are continuous in $\bar{D}$, $v_1>0$ in $D$ and $v_2>0$ in $\bar{D}$, then
\begin{align*}
\max_{\bar{D}} v_1/v_2 \leq \max_{\pa D} v_1/v_2 .
\end{align*}
\end{prop}
This is proven in \cite{PW}, Theorem 10, page 73, and follows from applying the usual maximum principle to the function $v_1/v_2$. We now prove a corollary of Proposition \ref{prop:H(x)decay}.

\begin{cor} \label{cor:H(x)decay}
Let $A \coloneqq \max_{x}H(x)$. Then, the function $H(x)$ satisfies the upper bound
\begin{align*}
 H(x) \leq C_1A\exp{(-c|x-x^*|/L_2)}.
 \end{align*}
\end{cor}
\begin{proof}{Corollary \ref{cor:H(x)decay}}
With $C>0$ as in the statement of Proposition \ref{prop:H(x)decay}, let $x_1 = x^* + CL_2$. We also define the function $R(x)$ for $x>x_1$ by
\begin{align*}
 R(x) \coloneqq  Ae^{-(x-x_1)/L_2}.
 \end{align*}
 Then, $R(x)$ satisfies $R''(x) = L_2^{-2}R(x)$, and $H(x_1) \leq A = R(x_1)$. By Proposition \ref{prop:H(x)decay} we know that
 \begin{align*}
 H''(x) \geq L_2^{-2}H(x)
 \end{align*}
 for all $x \geq x_1$. Therefore, setting $D$ to be the interval $\{x\geq x_1\}$, the conditions of the generalised maximum principle are satisfied and hence
 \begin{align*}
 H(x) \leq R(x)
 \end{align*}
 for all $x \geq x_1$. There is also an analogous bound for $x \leq x^* - CL_2$, and this completes the proof.
\end{proof}
\begin{rem}
In fact, we see from the proof that we can replace $A$ by $H(x_1)$ and conclude that for any $x_1\geq x^*  + CL_2$ we have the bound
\begin{align} \label{rem:H(x)decay}
H(x) \leq H(x_1)e^{-(x-x_1)/L_2}
\end{align}
for all $x>x_1$.
\end{rem}

In particular, as a result of this corollary, we see that $H(x)$ decays exponentially from its value at $x=x^*$ at least at a length scale comparable to $L_2$.

Our next aim is to obtain an upper bound for
\begin{align} \label{eqn:maxH}
A = \max_x H(x) = \max_x \int_{\Om(x)} u(x,y)^2 \ud y.
\end{align}
Suppose that we can show that $A$ satisfies
\begin{align*}
A \leq C_1L_1
\end{align*}
for an absolute constant $C_1$. Then, by Proposition \ref{prop:H(x)decay} we have
\begin{align*}
H(x) \leq C_1L_1 e^{-c|x-x^*|/L_2},
\end{align*}
and so  integrating over $x$ gives
\begin{align*}
\int_{\Om}u(x,y)^2 \ud x \ud y = \int H(x) \ud x \leq CL_1L_2.
\end{align*}
Therefore to complete the proof of Proposition \ref{prop:uL2bound}, it is sufficient to prove this upper bound on $A$. To do this we first define a cut-off function $\chi_1(x)$ as follows.
\begin{defn} \label{def:chi1}
We define $\chi_1(x)$ to be a smooth cut-off function, which satisfies
\begin{align*}
0 \leq \chi_1(x) \leq 1,
\end{align*}
and is equal to $1$ on the interval $[x^*-2CL_2,x^*+2CL_2]$ of length  $4CL_2$, with $C$ as in the statement of Proposition \ref{prop:H(x)decay}. Moreover, the function $\chi_1(x)$ is supported on the interval $[x^*-3CL_2,x^*+3CL_2]$  and has the derivative estimate
\begin{align*}
\left | \pa^k \chi_1(x) \right| \leq (CL_2)^{-k},
\end{align*}
for $k=1,2$.
\end{defn}

We now prove the following.
\begin{prop} \label{prop:ydecayupper}
Let $\chi_1(x)$ be the cut-off function above in Definition \ref{def:chi1}. Then,
\begin{align*}
\int_{\Om} \chi_1(x)u(x,y)^2 \ud x \ud y \leq C_1L_1L_2,
\end{align*}
for an absolute constant $C_1>0$. Note that this is consistent with $u(x,y)$ decaying on a length scale comparable to $L_1$ in the $y$-direction.
\end{prop}

\begin{proof}{Proposition \ref{prop:ydecayupper}}
The first eigenfunction $u(x,y)$ satisfies
\begin{align*}
-\Delta_{x,y}u(x,y) + (V(x,y) - \la)u(x,y) = 0 \text{ in } \Om
\end{align*}
with zero boundary conditions on $\pa\Om$. We integrate this against the function $\chi_1(x)u(x,y)$ to obtain
\begin{align*}
\int_{\Om} -\chi_1(x)u(x,y)\Delta_{x,y}u(x,y) + \chi_1(x)(V(x,y) - \la)u(x,y)^2 \ud x \ud y = 0,
\end{align*}
and integrating by parts one time in $x$ and $y$ gives
\begin{align} \label{eqn:ydecayupper1} \nonumber
 \int_{\Om} \chi_1(x)|\nabla_{x,y}u(x,y)|^2 &  \ud x \ud y  + \int_{\Om}\chi_1'(x) \pa_xu(x,y)u(x,y) \ud x \ud y \\
& + \int_{\Om} \chi_1(x)(V(x,y) - \la)u(x,y)^2  \ud x \ud y = 0.
\end{align}
In the second integral in \eqref{eqn:ydecayupper1} we can write
\begin{align*}
 \chi_1'(x)\pa_xu(x,y)u(x,y) = \tfrac{1}{2}\chi_1'(x)\pa_x(u(x,y)^2)
 \end{align*}
and integrate by parts in $x$ again to rewrite this integral as
\begin{align*}
 -\frac{1}{2} \int_{\Om} \chi_1''(x)u(x,y)^2\ud x \ud y.
 \end{align*}
 Thus, from \eqref{eqn:ydecayupper1} we have
 \begin{align} \label{eqn:ydecayupper2} \nonumber
 & \int_{\Om} \chi_1(x)|\nabla_{x,y}u(x,y)|^2 \ud x \ud y +  \int_{\Om} \chi_1(x)(V(x,y) - \la)_{+}u(x,y)^2  \ud x \ud y   \\ & =   \frac{1}{2}\int_{\Om}\chi_1''(x) u(x,y)^2 \ud x \ud y 
 + \int_{\Om} \chi_1(x)(V(x,y) - \la)_{-}u(x,y)^2  \ud x \ud y,
 \end{align}
 where we have decomposed $V(x,y) - \la$ into its positive and negative parts via
 \begin{align*}
 V(x,y) - \la = (V(x,y) - \la)_{+} - (V(x,y)-\la)_{-}.
 \end{align*}
 By the simple eigenvalue bound for $\la$ from Proposition \ref{prop:simpleeigenvalue}, we know that
 \begin{align*}
 (V(x,y) - \la)_{-} \leq C_1L_1^{-2}.
 \end{align*} 
 This also means that for any fixed $x$, we can only have $V(x,y) - \la \leq 0$ for $y$ in an interval of length at most comparable to $L_1$. Since the eigenfunction is normalised to have a maximum of $1$, and $\chi_1(x)$ is only non-zero in an interval of length comparable to $L_2$, this gives us a bound on the final term in the right hand side of \eqref{eqn:ydecayupper2} of
 \begin{align}  \label{eqn:ydecayupper3} 
  \int_{\Om} \chi_1(x)(V(x,y) - \la)_{-}u(x,y)^2  \ud x \ud y \leq C_1L_1^{-2}L_1L_2 = C_1L_1^{-1}L_2. 
 \end{align}
 We now turn to the second integral on the left hand side of \eqref{eqn:ydecayupper2}
 \begin{align*}
  \int_{\Om} \chi_1(x)(V(x,y) - \la)_{+}u(x,y)^2  \ud x \ud y .
 \end{align*}
 Fix a large constant $C_2>0$. For each fixed $x$, $V(x,y) - \la$ is only bounded above by $C_2L_1^{-2}$ on an interval in $y$ of length comparable to $L_1$. Therefore, again combining this with the bound $u(x,y) \leq 1$, we can write
 \begin{align}  \label{eqn:ydecayupper4} 
 C_2L_1^{-2} \int_{\Om}\chi_1(x)u(x,y)^2 \ud x \ud y - C_1L_1^{-1}L_2 \leq   \int_{\Om} \chi_1(x)(V(x,y) - \la)_{+}u(x,y)^2  \ud x \ud y .
 \end{align}
 Inserting the estimates in \eqref{eqn:ydecayupper3} and \eqref{eqn:ydecayupper4} back into \eqref{eqn:ydecayupper2} we see that
  \begin{align} \label{eqn:ydecayupper5} \nonumber
 & \int_{\Om} \chi_1(x)|\nabla_{x,y}u(x,y)|^2 \ud x \ud y +  C_2L_1^{-2} \int_{\Om} \chi_1(x) u(x,y)^2  \ud x \ud y   \\ & \leq   \frac{1}{2}\int_{\Om}\chi_1''(x) u(x,y)^2 \ud x \ud y  +C_1L_1^{-1}L_2.
 \end{align}
 The first integral in \eqref{eqn:ydecayupper5} is positive, and so we can drop it from the estimate. Therefore, dividing by $C_2L_1^{-2}$ gives us
 \begin{align}  \label{eqn:ydecayupper6} 
 \int_{\Om} \chi_1(x) u(x,y)^2  \ud x \ud y    \leq  \frac{1}{2} C_2^{-1}L_1^2\int_{\Om}\chi_1''(x) u(x,y)^2 \ud x \ud y  + C_1L_1L_2.
 \end{align} 
 To conclude the proof of the proposition, we will use Corollary \ref{cor:H(x)decay} and the remark following it. By \eqref{rem:H(x)decay}, for any $x_1 \geq x^* + CL_2$ and any $x \geq x_1$, we have
 \begin{align*}
 \int_{\Om(x)}u(x,y)^2 \ud y \leq e^{-(x-x_1)/L_2}  \int_{\Om(x_1)}u(x_1,y)^2 \ud y. 
 \end{align*}
 Therefore, we certainly have the estimate
 \begin{align}  \label{eqn:ydecayupper7} 
\int_{x^*+2CL_2}^{x^*+3CL_2}  \int_{\Om(x)}u(x,y)^2\ud x \ud y \leq \int_{x^*+CL_2}^{x^*+2CL_2} \int_{\Om(x)}u(x,y)^2 \ud x \ud y,
 \end{align}
 and an analogous estimate for $x_1 \leq x^*-CL_2$ and  $x \leq x_1$. By the definition of the cut-off function $\chi_1(x)$, the second derivative $\chi_1''(x)$ is supported on the intervals $[x^*-3CL_2,x^*-2CL_2]$ and $[x^*+2CL_2,x^*+3CL_2]$, and is of order $L_2^{-2}$ here. Also, $\chi_1(x)$ is equal to $1$ on the intervals $[x^*-2CL_2,x^*-CL_2]$ and $[x^*+CL_2,x^*+2CL_2]$. Therefore, using the estimate in \eqref{eqn:ydecayupper7} the integral on the right hand side of \eqref{eqn:ydecayupper6} is certainly at most $\tfrac{1}{2}$ the size of the integral on the left hand side. This means that in \eqref{eqn:ydecayupper6} we can bring over the integral to the left hand side and get the bound
 \begin{align*}
 \int_{\Om} \chi_1(x) u(x,y)^2  \ud x \ud y    \leq  C_1L_1L_2,
 \end{align*}
 as required.
\end{proof}

\begin{cor} \label{cor:ydecayupper}
We have the derivative bound
\begin{align*}
\int_{\Om} \chi_1(x) |\nabla_{x,y}u(x,y)|^2 \ud x \ud y \leq C_1L_1^{-1}L_2.
\end{align*}
\end{cor}
\begin{proof}{Corollary \ref{cor:ydecayupper}}
In the proof of Proposition \ref{prop:ydecayupper} in \eqref{eqn:ydecayupper5} we established the estimate
\begin{align} \label{cor:ydecayupper1} \nonumber
 & \int_{\Om} \chi_1(x)|\nabla_{x,y}u(x,y)|^2 \ud x \ud y +  C_2L_1^{-2} \int_{\Om} \chi_1(x) u(x,y)^2  \ud x \ud y   \\ & \leq   \frac{1}{2}\int_{\Om}\chi_1''(x) u(x,y)^2 \ud x \ud y  +C_1L_1^{-1}L_2.
\end{align}
We also showed that 
\begin{align*}
L_1^2\int_{\Om}\chi_1''(x) u(x,y)^2 \ud x \ud y
\end{align*}
is bounded by $\int_{\Om} \chi_1(x) u(x,y)^2  \ud x \ud y $, and hence by Proposition \ref{prop:ydecayupper}  is bounded by $C_1L_1^{-1}L_2$. Using this estimate in \eqref{cor:ydecayupper1} gives the desired result.
\end{proof}

The derivative bound
\begin{align*}
\int_{\Om} \chi_1(x) |\nabla_{x,y}u(x,y)|^2 \ud x \ud y \leq C_1L_1^{-1}L_2 
\end{align*}
 is of order $L_1^{-2}$ smaller than the bound we obtained for the eigenfunction $u(x,y)$ itself in Proposition \ref{prop:ydecayupper}. For the $y$-derivative $\pa_yu(x,y)$, this bound is consistent with our eventual aim to show that $u(x,y)$ decays away from its maximum on a length scale comparable to $L_1$. However, in the $x$-direction, our aim is to show that $u(x,y)$ decays away from its maximum on a length scale comparable to $L_2$. Therefore, we want to improve the bound on $\pa_xu(x,y)$ given in Corollary \ref{cor:ydecayupper}. 

\begin{prop} \label{prop:xdecayupper}
Let $\chi_1(x)$ be as in Definition \ref{def:chi1}. Then, there exists an absolute constant $C_1>0$ such that
\begin{align*}
\int_{\Om} \chi_1(x) (\pa_xu(x,y))^2 \ud x \ud y \leq CL_1L_2^{-1}.
\end{align*}
Note that for $L_2\gg L_1$ this is an improvement on the bound in Corollary \ref{cor:ydecayupper}.
\end{prop}
\begin{proof}{Proposition \ref{prop:xdecayupper}}
We begin by proceeding as in the proof of Proposition \ref{prop:ydecayupper} to obtain the equality in \eqref{eqn:ydecayupper2}:
\begin{align} \label{eqn:xdecayupper1} \nonumber
 \int_{\Om} \chi_1(x)|\nabla_{x,y}u(x,y)|^2 \ud x \ud y&  +  \int_{\Om} \chi_1(x)(V(x,y) - \la)u(x,y)^2  \ud x \ud y \\
 &   -  \frac{1}{2}\int_{\Om}\chi_1''(x) u(x,y)^2 \ud x \ud y  = 0 . 
\end{align}
We know that the integral of $\chi_1(x)u(x,y)^2$ is at most $C_1L_1L_2$. Since $|\chi_1''(x)| \leq C_1L_2^{-2}$ this means that
\begin{align*}
  \frac{1}{2}\int_{\Om}\left|\chi_1''(x)\right| u(x,y)^2 \ud x \ud y   \leq C_1L_1L_2^{-1},
\end{align*}
and so from \eqref{eqn:xdecayupper1} we have
\begin{align} \label{eqn:xdecayupper2} \nonumber
&  \int_{\Om} \chi_1(x) (\pa_xu(x,y))^2 \ud x \ud y +  \int_{\Om} \chi_1(x)(\pa_{y}u(x,y))^2 \ud x \ud y \\
 & +   \int_{\Om} \chi_1(x)(V(x,y) - \la)u(x,y)^2  \ud x \ud y \leq C_1L_1L_2^{-1} . 
 \end{align}
 For each fixed $x$, the eigenfunction $u(x,y)$ is an admissible test function for our usual ordinary differential operator
 \begin{align*}
 \mathcal{L}(x) = - \frac{d^2}{dy^2} + V(x,y).
 \end{align*}
 Since this operator has first eigenvalue equal to $\mu(x)$, we obtain the lower bound
 \begin{align} \label{eqn:xdecayupper3}
  \int_{\Om(x)} (\pa_{y}u(x,y))^2 +  (V(x,y) - \la)u(x,y)^2   \ud y \geq (\mu(x) - \la)\int_{\Om(x)}u(x,y)^2 \ud y.
 \end{align}
 Multiplying the inequality in \eqref{eqn:xdecayupper3} by $\chi_1(x)$ and integrating over $x$, \eqref{eqn:xdecayupper2} becomes
 \begin{align} \label{eqn:xdecayupper4}
  \int_{\Om} \chi_1(x) (\pa_xu(x,y))^2 \ud x \ud y + \int_{\Om} \chi_1(x)(\mu(x) - \la)u(x,y)^2 \ud x \ud y \leq C_1L_1L_2^{-1}.
 \end{align}
 By the definition of $L_2$, we have
 \begin{align*}
 \mu(x) - \mu \geq -C_1L_2^{-2},
 \end{align*}
 and by the eigenvalue bounds in Theorem \ref{thm:eigenvalue}, we know that
 \begin{align*}
 \mu - \la \geq -C_1L_2^{-2}.
 \end{align*}
 Therefore, \eqref{eqn:xdecayupper4} tells us that
 \begin{align*}
  \int_{\Om} \chi_1(x) (\pa_xu(x,y))^2 \ud x \ud y \leq C_1L_2^{-2} \int_{\Om}\chi_1(x)u(x,y)^2 \ud x \ud y + C_1L_1L_2^{-1}.
 \end{align*}
 Applying Proposition \ref{prop:ydecayupper} then gives the desired bound.
\end{proof}

Now that we have established $L^2$-bounds for the first derivative, $\nabla_{x,y}u(x,y)$, in Propositions \ref{prop:ydecayupper} and \ref{prop:xdecayupper}, we can return to establishing the required upper bound for
\begin{align*}
A = \max_{x}H(x)  = \max_{x}\int_{\Om(x)} u(x,y)^2 \ud y. 
\end{align*}
\begin{prop} \label{prop:Aupper}
$A$ is bounded by $L_1$ multiplied by an absolute constant.
\end{prop}
\begin{proof}{Proposition \ref{prop:Aupper}}
Suppose that we have
\begin{align} \label{eqn:Aupper1}
\max_{x}H(x) = H(x^*) = \int_{\Om(x^*)} u(x^*,y)^2 \ud y \geq C^* L_1,
\end{align}
where $C^*>0$ is a large absolute constant that we will specify later. Then, for any $(x,y)$, extending $u(x,y)$ to be $0$ outside of $\Om$, we can write 
\begin{align*}
u(x,y) = u(x^*,y) + \int_{x^*}^{x} \pa_tu(t,y) \ud t,
\end{align*}
and so
\begin{align} \label{eqn:Aupper2}
u(x,y)^2 \geq \tfrac{1}{2}u(x^*,y)^2 - C_1|x-x^*|\int_{x^*}^x (\pa_tu(t,y))^2 \ud t,
\end{align}
for a fixed constant $C_1$. Integrating the inequality in \eqref{eqn:Aupper2} over $y$ we find that
\begin{align*}
 H(x) \geq \tfrac{1}{2}H(x^*) - C_1|x-x^*|\int_{x^*}^x\int_{\Om(t)} (\pa_tu(t,y))^2 \ud y \ud t,
 \end{align*}
 and so by the assumption on $H(x^*)$ in \eqref{eqn:Aupper1}, this gives
 \begin{align} \label{eqn:Aupper3}
 H(x) \geq \tfrac{1}{2}C^*L_1- C_1|x-x^*|\int_{x^*}^x\int_{\Om(t)} (\pa_tu(t,y))^2 \ud y \ud t.
 \end{align}
Let us restrict to those values of  $x$ with $|x-x^*|\leq c_1L_2$ for a small constant $c_1>0$. Then by the derivative bound  on $\pa_tu(t,y)$ in Proposition \ref{prop:xdecayupper}, we can ensure that the second term in \eqref{eqn:Aupper3} is small compared to $\tfrac{1}{4}C^*L_1$. Moreover, this constant $c_1$ can be chosen to be independent of $C^*$. Therefore, this tells us that for all $x$ in an interval of length $2c_1L_2$, we have the lower bound
\begin{align*}
H(x) \geq \tfrac{1}{4}C^*L_1 .
\end{align*}
In particular, this shows that
\begin{align*}
\int_{\Om}\chi_1(x)u(x,y)^2 \ud x \ud y \geq  \int_{x^*-c_1L_2}^{x^*+c_1L_2}H(x) \ud x \geq \tfrac{1}{2}c_1C^*L_1L_2.
\end{align*}
Since $c_1$ is independent of $C^*$, we can contradict the $L^2(\Om)$-bound from Proposition \ref{prop:ydecayupper} by choosing $C^*$ sufficiently large.
\end{proof}

By the discussion after the proof of Corollary \ref{cor:H(x)decay}, this upper bound on $A$ from Proposition \ref{prop:Aupper} implies the $L^2(\Om)$-bound
\begin{align*}
\int_{\Om} u(x,y)^2 \ud x \ud y \leq CL_1 L_2.
\end{align*}
This completes the proof of Proposition \ref{prop:uL2bound}.

\end{proof}

In Proposition \ref{prop:uL2bound} we derived an $L^2(\Om)$-bound for the first eigenfunction $u(x,y)$. For our purposes of studying the shape of the level sets of $u(x,y)$ near to its maximum this will be sufficient. However, another interesting question is to study the rate at which $u(x,y)$ decays from its maximum. Therefore, before continuing with our study of the level sets, let us give some indication about the decay of $u(x,y)$ as we move away from its maximum.

We will do this by using an Agmon-type estimate, but  first  we need some definitions.
\begin{defn} \label{def:Om1}
Fix a large absolute constant $C>0$, and let $\Om_1$ be the subset of $\Om$ given by
\begin{align*}
\Om_1 \coloneqq \{ (x,y) \in \Om: V(x,y) \geq 1 + CL_1^{-2} \} .
\end{align*}
Note that the boundary of $\Om_1$ consists of parts of the two convex curves coming from $\pa\Om$ and the level set $\{ (x,y)\in\Om :V(x,y) = 1+ CL_1^{-2} \}$.

%We also define $\Om_2 \subset \Om_1$ as follows. Projecting $\Om\backslash\Om_1$ onto the $y$-axis we obtain an interval, and let $y_1$ be the upper endpoint of this interval. Then, we define $\Om_2$ to be those points $(x,y)\in\Om_1$ such that $y\geq y_1$.
\end{defn}

\begin{defn} \label{def:h(x,y)}
With $ \Om_1 \subset \Om$ as above, we  also define the distance function 
\begin{align*}
h^* : \Om_1 \to [0,\infty) 
\end{align*}
as follows.  We first define the function $\nu^*(x,y)$ to be equal to $V(x,y)-\la$. For $(x,y)$ in $\Om_1$ we then  define $h^*(x,y)$ by
\begin{align*}
h^*(x,y) = \inf_{\gamma} \frac{1}{2} \int_{0}^{1} \nu^*(\gamma(t))^{1/2}|\gamma'(t)| \ud t,
\end{align*}
where the infimum is taken over all paths $\gamma:[0,1]\to \Om_1$ between the inner boundary of $\Om_1$ and $(x,y)$.
\end{defn}

We are now in a position to state our Agmon-type estimate.
\begin{prop} \label{prop:Agmon}
For $\Om_1$ and $h^*(x,y)$ defined as above, we have
\begin{align*}
\int_{\Om_1} u(x,y)^2 e^{2h^*(x,y)} \ud x \ud y  \leq C_2L_1L_2 ,
\end{align*}
for some absolute constant $C_2>0$.
\end{prop}
\begin{rem}
Since we certainly have the lower bound $V(x,y)-\la \geq C_1L_1^{-2}$ on $\Om_1$, roughly speaking this proposition shows that, in an $L^2(\Om)$-sense, the function $u(x,y)$ decays at least on a length scale comparable to $L_1$ as we move away from the region where $V(x,y) \leq 1+CL_1^{-2}$. However, as $V(x,y)-\la$ grows, this rate of exponential decay also increases.  
%We have an analogous estimate in the downwards $y$-direction. 
\end{rem}
\begin{proof}{Proposition \ref{prop:Agmon}}
This proposition will follow from a classical Agmon estimate in \cite{Ag}. Let us restate Theorem 1.5 from \cite{Ag} (using slightly different notation).
\begin{thm}[Theorem 1.5 in \cite{Ag}] \label{thm:Agmon}
Let $D$ be a bounded connected open set in $\R^2$. Let $q(x,y)$ be a real valued function on $D$, and suppose that $\nu(x,y)$ is a positive continuous function on $D$ such that
\begin{align} \label{eqn:Agmonnu} 
\int_{D} |\nabla_{x,y} \psi(x,y)|^2  + q(x,y)\psi(x,y)^2 \ud x \ud y\geq \int_{D} \nu(x,y)\psi(x,y)^2 \ud x \ud y
\end{align}
for all $\psi\in C_0^{\infty}(D)$.

Fix a point $(x_0,y_0) \in D$, and define the distance $\rho_\nu(x,y)$ by
\begin{align} \label{eqn:Agmonrho}
\rho_\nu(x,y) \coloneqq \inf_{\gamma} \int_{0}^{1} \nu(\gamma(t))^{1/2} |\gamma'(t)| \ud t,
\end{align}
where the infimum is taken over all continuous paths $\gamma:[0,1]\to D$ in $D$ between $(x_0,y_0)$ and $(x,y)$. We also define $\rho_{\nu}((x,y),\{\infty\})$ to be the distance from the point $(x,y)$ to $\pa D$ under the distance function $\rho_{\nu}(x,y)$, and define $D_s$ by
\begin{align*}
D_s \coloneqq  \{(x,y)\in D: \rho_{\nu}((x,y),\{\infty\}) > s\} .
\end{align*} 

Finally, suppose that
\begin{align*}
-\Delta_{x,y} W(x,y) + q(x,y)W(x,y) = 0
\end{align*}
and that  the function $g(x,y)$ satisfies
\begin{align} \label{eqn:Agmong} 
|\nabla_{x,y} g(x,y)|^2 < \nu(x,y)
\end{align}
in $D$. Then, we have the estimate
\begin{align} \label{eqn:AgmonW} \nonumber
\int_{D_s} & W(x,y)^2 (\nu(x,y) - |\nabla_{x,y}g(x,y)|^2)e^{2g(x,y)} \ud x \ud y \\
& \leq \frac{2(1+2s)}{s^2} \int_{D\backslash D_s} W(x,y)^2 \nu(x,y)e^{2g(x,y)} \ud x \ud y.
\end{align}
\end{thm}
We will now apply this theorem with $W(x,y) = u(x,y)$ and $q(x,y) = V(x,y) - \la$. We will choose the set $D$ as follows: We recall that $\Om_1$ consists of those points $(x,y)$ with $V(x,y)-\la \geq 1 + CL_1^{-2}$. We then define $D$ to be all points in $\R^2$ outside of the inner boundary of $\Om_1$.
%  by
%\begin{align*}
%D \coloneqq \{(x,y)\in\R^2: y \geq y_1 \}.
%\end{align*}

Since $u(x,y) = 0$ on $\pa\Om$, we can extend $u(x,y)$ to $D$ by setting it to be $0$ for $D \backslash \Om_1$, and we extend the potential $V(x,y)$ to $D$ arbitrarily. 

We clearly have the estimate 
\begin{align*}
\int_{D} |\nabla_{x,y} \psi(x,y)|^2  + (V(x,y)-\la) \psi(x,y)^2 \ud x \ud y\geq \int_{D} (V(x,y)-\la) \psi(x,y)^2 \ud x \ud y\end{align*}
for all $\psi\in C_0^{\infty}(D)$. Also, $V(x,y)-\la\geq C_1L_1^{-2}$ for $(x,y)\in D$. As a result of this, from \eqref{eqn:Agmonnu} we see that we we can set $\nu^*(x,y)$ to be equal to the function described in the definition of $h^*(x,y)$ in Definition \ref{def:h(x,y)}.

In Theorem \ref{thm:Agmon} we are free to choose the value for $s$, and we will choose $s=1$. Then, we see that
\begin{align*}
D\backslash D_{1} = \{(x,y)\in D: \rho_{\nu}((x,y),\{\infty\}) \leq 1 \}
\end{align*}
consists of the region near the inner boundary of $D$ with width comparable to at most $L_1$. This is because we have ensured that $\nu(x,y) \geq cL_1^{-2}$ when the point $(x,y)$ is  within a distance $L_1$ of the boundary of $D$.

We finally need to choose $g(x,y)$ to ensure that \eqref{eqn:Agmong} holds, and so we need
\begin{align*}
|\nabla_{x,y} g(x,y)|^2 <  \nu^*(x,y) =  V(x,y) -\la.
\end{align*}
We can achieve this by setting $g(x,y)$  to be equal to the function $h^*(x,y)$ as in Definition \ref{def:h(x,y)}. This certainly satisfies the required derivative bound.

Thus, we can apply Theorem \ref{thm:Agmon} to get
\begin{align} \label{eqn:Agmonfinal} \nonumber
\int_{D_1} & u(x,y)^2 (\nu(x,y) - |\nabla_{x,y}h^*(x,y)|^2)e^{2h^*(x,y)} \ud x \ud y \\
& \leq 6\int_{D\backslash D_1} u(x,y)^2 \nu(x,y)e^{2h^*(x,y)} \ud x \ud y.
\end{align}
On $D\backslash D_1$, we know that $\nu(x,y) \leq L_1^{-2}$, and $h^*(x,y) \leq 1$. Therefore, by the $L^2$ bound on $u(x,y)$ from Proposition \ref{prop:uL2bound}, the right hand side of \eqref{eqn:Agmonfinal} is bounded by 
\begin{align*}
C_1L_1^{-2}L_1L_2 = C_1L_1^{-1}L_2.
\end{align*}
Since for $(x,y) \in D$, we have $ \nu(x,y) - |\nabla_{x,y}h^*(x,y)|^2 \geq c_1L_1^{-2}$, we can therefore conclude from \eqref{eqn:Agmonfinal} that
\begin{align*}
\int_{\Om_2}  u(x,y)^2 e^{2h^*(x,y)} \ud x \ud y \leq C_1L_1L_2
\end{align*}
as required. 
\end{proof}

\section{The Shape Of The Level Sets Of $u(x,y)$} \label{sec:shape}

We now return to the problem of studying the shape of the level sets of the first eigenfunction $u(x,y)$. As we have mentioned earlier, since the potential $V(x,y)$ is convex, a theorem of Brascamp and Lieb, \cite{BL2} tells us that $u(x,y)$ is log concave. In particular, this means that the superlevel sets of $u(x,y)$ are convex subsets of $\Om$.

We will use the results of the previous section to estimate the lengths of the projections of these level sets onto the $x$ and $y$-axis. In particular, in this section we will establish Theorem \ref{thm:shape} about the shape of the level sets. Throughout this section we let $c_1>0$ be a small absolute constant as in the statement of Theorem \ref{thm:shape}. The constant $c>0$ which appears in the propositions below is bounded away from $0$ and $1$ by satisfying
\begin{align*}
c_1 < c < 1-c_1,
\end{align*}
and all other constants will depend on the choice of $c_1$.

 We first use the bound on $A$ from Proposition \ref{prop:Aupper} to find an upper bound on the behaviour of the level sets of $u(x,y)$ in the $y$-direction.

\begin{prop} \label{prop:levelyupper}
Let $0<c<1$ be a fixed absolute constant. Then, for any fixed $x$, the cross-section of the superlevel set $\{(x,y)\in\Om:u(x,y) \geq c\}$ at $x$ consists of an interval of length at most $L_1$ multiplied by an absolute constant.
\end{prop}
\begin{proof}{Proposition \ref{prop:levelyupper}}
By Proposition \ref{prop:Aupper} we know that
\begin{align*}
A = \max_{x}H(x)  = \max_{x} \int_{\Om(x)}u(x,y)^2 \ud y \leq C_1L_1.
\end{align*}
If $u(x,y) \geq c$ for $y$ in an interval of length $CL_1$ for $C$ sufficiently large, this immediately gives a contradiction.
\end{proof}

We can also prove an upper bound on the length of the projection of the level sets of $u(x,y)$ in the $y$-direction.
\begin{prop} \label{prop:firstlocation}
For sufficiently small $\delta>0$ fixed, there exists an $\eta>0$ such that if the point $(x,y)$ is within a distance $\eta L_1$ of the level set $\{ (x,y)\in\Om: V(x,y) = 1+ \eta^{-1}L_1^{-2}\}$, then
\begin{align*}
 u(x,y) \leq \delta.
 \end{align*}
In particular, the level sets $\{(x,y)\in\Om: u(x,y) = c\}$ are at a distance comparable to $L_1$ away from the level set $\{ (x,y)\in\Om: V(x,y) = 1+CL_1^{-2}\}$, for some absolute constant $C>0$.
\end{prop}
\begin{rem}
The proof of this proposition follows closely the proof of Lemma 3.17 in \cite{GJ2}, where an analogous property has been established for the first eigenfunction of a two dimensional convex domain.
\end{rem}
Before proving this proposition, let us show the following corollary:
\begin{cor} \label{cor:firstlocationA}
Let $0<c<1$ be a fixed absolute constant. Then, the projection of the level set $\{(x,y)\in\Om:u(x,y) = c\}$ onto the $y$-axis has length bounded from above by an absolute constant multiplied by $L_1$.
\end{cor}  
\begin{proof}{Corollary \ref{cor:firstlocationA}}
By the definition of the length scale $L_1$ and the orientation of the level set $\Om_{L_1^{-2}} = \{(x,y)\in\Om:V(x,y) = 1+L_1^{-2}\}$  we used when defining $L_2$, we know that the projection of the level set $\Om_{L_1^{-2}}$ onto the $y$-axis has length comparable to $L_1$. Moreover, by the convexity of the potential $V(x,y)$, this is true for any level set $ \{(x,y)\in\Om:V(x,y) = 1+CL_1^{-2}\}$, for any absolute constant $C>0$. Therefore, the upper bound on the length of the projection of the level sets $\{(x,y)\in\Om:u(x,y) = c\}$ onto the $y$-axis follows from Proposition \ref{prop:firstlocation}.  
\end{proof}

\begin{proof}{Proposition \ref{prop:firstlocation}}
Let $(x',y')$ be a point which is within a distance  $\eta L_1$ of the level set  $\{(x,y)\in\Om :V(x,y) = 1+ \eta^{-1}L_1^{-2}\}$. After a rotation, we may assume that the nearest point of $\{(x,y)\in\Om : V(x,y) = 1+ \eta^{-1}L_1^{-2}\}$ to $(x',y')$ is equal to $(x',y_1)$, with $y_1<y'$ and $y'-y_1 <\eta L_1$.

We will need to use two properties of the potential $V(x,y)$. Firstly, by the simple eigenvalue bounds on $\la$ in Proposition \ref{prop:simpleeigenvalue} we have seen before that
\begin{align} \label{eqn:firstlocation1} 
 \Delta_{x,y}u(x,y) = (V(x,y) - \la)u(x,y) \geq -\frac{C_1^2}{L_1^2}u(x,y)
 \end{align}
for all values of $(x,y)$, for some absolute constant $C_1$. Moreover, $V(x,y)$ has convex sublevel sets and by the rotation we made above we have $V(x,y_1) = 1+\eta^{-1}L_1^{-2}$. Therefore,
\begin{align} \label{eqn:firstlocation2} 
 \Delta_{x,y}u(x,y) = (V(x,y) - \la)u(x,y) \geq \frac{1}{2\eta L_1^2}u(x,y),
 \end{align}
whenever $y \leq y'-\eta L_1 < y_1$.

We define the comparison function $v_1(x,y)$ by
\begin{align} \label{eqn:firstlocation3}
v_1(x,y) = \sin\left(\frac{C_1(y-y')}{2L_1} + \frac{C_1\eta}{2} + C_1\delta\right)
\end{align}
for $y \geq y'-\eta L_1$, and by
\begin{align} \label{eqn:firstlocation4}
v_1(x,y) = (\sin(C_1\delta))\exp\left(\frac{\delta}{2}+\frac{(y-y')}{2\delta L_1}\right)
\end{align}
for $y < y'-\eta L_1$. We make the choice $\eta = \delta^2$, and this ensures that $v_1(x,y)$ is continuous at $y = y'-\eta L_1$ for all values of $x$.

For $\delta>0$ sufficiently small, using $\sin(C_1\delta) >C_1\delta \cos(C_1\delta)$, we find that $\pa_y^2 v_1(x,y)$ has a negative delta function along $y = y'-\eta L_1$. Everywhere else, calculating $\Delta_{x,y} v_1(x,y)$ from its definition in \eqref{eqn:firstlocation3} and \eqref{eqn:firstlocation4}, and using the inequalities for $\Delta_{x,y}u(x,y)$ in \eqref{eqn:firstlocation1} and \eqref{eqn:firstlocation2}, we see that
\begin{align*}
\frac{ \Delta_{x,y} v_1(x,y)}{v_1(x,y)} \leq \frac{ \Delta_{x,y} u(x,y)}{u(x,y)}.
\end{align*}
Moreover, for those $(x,y) \in \pa\Om$, with $y \leq y' + \left(\frac{\pi}{C_1} - \eta -2\delta\right)L_1$ we have
\begin{align*}
v_1(x,y) > 0 = u(x,y)|_{\pa\Om}
\end{align*}
and for $(x,y) \in \Om$ with $y= y' + \left(\frac{\pi}{C_1} - \eta -2\delta\right)L_1$, we have
\begin{align*}
v_1\left(x,y' + (\pi/C_1 - \eta -2\delta)L_1\right) = 1 \geq u\left(x,y' + (\pi/C_1 - \eta -2\delta)L_1\right).
\end{align*}
Thus, applying the generalised maximum principle in Proposition \ref{prop:GMP} to those $(x,y)$ in $\Om$ with $y \leq y' + \left(\frac{\pi}{C_1} - \eta -2\delta\right)L_1$ , $v_1(x,y)$ is a positive supersolution, and in particular
\begin{align*}
 u(x',y') \leq v_1(x',y') = \sin\left(\frac{C_1\eta}{2}+C_1\delta\right) \leq C_2\delta.
\end{align*}
Thus, repeating the argument with a suitable multiple of $\delta$ gives the desired result.

\end{proof}

We now want to obtain a lower bound on the height of the level sets in the $y$-direction.
\begin{prop} \label{prop:levelylower}
Let $0<c<1$ be a fixed absolute constant. Then, the superlevel set $\{(x,y)\in\Om:u(x,y) \geq c\}$ has inner radius bounded below by an absolute constant multiplied by $L_1$. In particular, the projection of the level set $\{(x,y)\in\Om:u(x,y) = c\}$ onto the $y$-axis has length bounded from below by an absolute constant multiplied by $L_1$.
\end{prop}
\begin{rem}
The proof of this proposition only considers the parameter $L_1$, and does not use any properties of the eigenvalue or eigenfunction that depend on $L_2$. In particular, this means that we do not need to fix the orientation of the level set $\Om_{L_1^{-2}} = \{(x,y)\in \Om:V(x,y) = 1+ L_1^{-2}\}$, and we are free to rotate $\Om$ in the course of the proof.
\end{rem}
\begin{proof}{Proposition \ref{prop:levelylower}}
Let us consider the case $c=1/4$, and study the level set $\{(x,y)\in\Om: u(x,y) = \tfrac{1}{4}\}$. Suppose that the shortest projection of the set onto any direction is of length $\alpha$. By the convexity of the superlevel sets of $u(x,y)$, after a rotation and a translation, we may then assume that this level set lies between the two lines $y=0$ and $y=\alpha$.

We will use the comparison function
\begin{align*}
 W(x,y) \coloneqq \frac{1}{2}\sin \left(\frac{\pi}{6}+ \frac{2\pi}{3\alpha}y \right).
 \end{align*}
 This function is equal to $1/4$ when $y=0$ or $y = \alpha$, and satisfies
\begin{align} \label{eqn:levelylower1}
 (\Delta_{x,y} -V(x,y)+\la)W(x,y) = -\left(\frac{2\pi}{3\alpha}\right)^2W(x,y) + (\la-V(x,y))W(x,y). 
\end{align}
 Since $V(x,y) \geq 1$, by the straightforward eigenvalue bound on $\la$ from Proposition \ref{prop:simpleeigenvalue} we have
\begin{align*}
\la - V(x,y) \leq C^2L_1^{-2},
\end{align*}
for an absolute constant $C>0$. Therefore, from \eqref{eqn:levelylower1} we obtain
\begin{align} \label{eqn:levelylower2}
 (\Delta_{x,y} -V(x,y)+\la)W(x,y) \leq  \left(-\left(\frac{2\pi}{3\alpha}\right)^2 + C^2L_1^{-2}\right)W(x,y) .
\end{align}
Let us assume that
\begin{align} \label{eqn:levelylower3}
\alpha < \frac{2\pi L_1}{3C}.
\end{align}
Then, from \eqref{eqn:levelylower2} we see that
\begin{align*}
 (\Delta_{x,y} -V(x,y)+\la)W(x,y) < 0,
\end{align*}
while $ (\Delta_{x,y} -V(x,y)+\la)u(x,y)  = 0$ in $\Om$. Also, for all points $(x,y)$ with $y=0$, $\alpha$ we have
\begin{align*}
 u(x,y) \leq W(x,y) = \frac{1}{4},
\end{align*}
and $u(x,y) = 0 < W(x,y)$  for $(x,y) \in \pa \Om$, with $0 \leq y \leq \alpha$. Therefore, by the generalised maximum principle in Proposition \ref{prop:GMP} we find that
\begin{align*}
u(x,y) \leq W(x,y) \qquad \text{ for }(x,y)\in D \text{ with } 0 \leq y \leq \alpha.
\end{align*}
However, $W(x,y) \leq \tfrac{1}{2}$, while $u(x,y)$ attains its maximum of $1$ at some point $(x,y)$ with $0 \leq y \leq \alpha$. This gives a contradiction, and so from \eqref{eqn:levelylower3} we must have
\begin{align*}
\alpha > \frac{2\pi L_1}{3C}. 
\end{align*}
Therefore the projection of the superlevel set $\{(x,y)\in \Om: u(x,y) \geq \tfrac{1}{4}\}$ onto any direction has length at least comparable to $L_1$, and this gives us the required lower bound on the inner radius of this superlevel set. We can also repeat the argument above for the superlevel set $\{(x,y)\in\Om: u(x,y) \geq c\}$ for any fixed absolute constant $c$ with $c_1<c<1-c_1$ to obtain the same result.
\end{proof}

\begin{cor} \label{cor:levelylower}
As an immediate consequence of Proposition \ref{prop:levelylower}, we see that
\begin{align*}
A = \max_{x} \int_{\Om(x)}u(x,y)^2 \ud y \geq \tilde{c}L_1,
\end{align*}
for an absolute constant $\tilde{c}>0$.
\end{cor}

Combining Propositions \ref{prop:levelyupper} and \ref{prop:levelylower}, the height of the level set $\{(x,y)\in\Om:u(x,y) = c\}$ in the $y$-direction is comparable to $L_1$. We now turn to the studying the length of the level sets of $u(x,y)$ in the $x$-direction. We first use Corollary \ref{cor:H(x)decay} to obtain an upper bound on the length of the level sets.

\begin{prop} \label{prop:levelxupper}
Let $0<c<1$ be a fixed absolute constant. Then, the projection of the level set $\{(x,y)\in\Om: u(x,y) = c\}$ onto the $x$-axis has length bounded by an absolute constant multiplied by $L_2$.
\end{prop}
\begin{proof}{Proposition \ref{prop:levelxupper}}
Suppose that the length of the projection of $\{(x,y)\in\Om: u(x,y)  = c\}$ onto the $x$-axis is bounded below by $2CL_2$, where $C>0$ is a large absolute constant that we will specify later in the proof. For each fixed $x$, the cross-section of the superlevel set $\{(x,y) \in\Om: u(x,y) \geq c\}$ at $x$ consists of an interval. Since the superlevel set is convex, the length of this interval is greater than half of its maximum length for $x$ lying in an interval of length $CL_2$. 

By Proposition \ref{prop:levelylower}, this maximum length is bounded below by $2C_1L_1$ for an absolute constant $C_1>0$. In other words, $u(x,y) \geq c$ for all $(x,y)$ in a rectangle of height $C_1L_1$ and width $CL_2$.

As a result of this, we have 
\begin{align} \label{eqn:levelxupper1}
H(x) = \int_{\Om(x)}u(x,y)^2 \ud y \geq c^2C_1L_1,
\end{align}
for all $x$ in an interval of length $CL_2$. By Proposition \ref{prop:Aupper}, $A$ is bounded by an absolute constant multiplied by $L_1$, and by Corollary \ref{cor:H(x)decay} we have the bound
\begin{align}  \label{eqn:levelxupper2}
H(x) \leq A e^{-c|x-x^*|/L_2}.
\end{align}
Therefore, combining \eqref{eqn:levelxupper1} and \eqref{eqn:levelxupper2}, we obtain a contradiction if we choose $C$ to be sufficiently large. This completes the proof of the proposition.
\end{proof}

To complete the proof of Theorem \ref{thm:shape} we finally want to obtain a comparable lower bound on the length of the level set of $u(x,y)$ in the $x$-direction. To do this we will use the $L^2$-bound on the first derivative $\pa_xu(x,y)$ from Proposition \ref{prop:xdecayupper}.

\begin{prop} \label{prop:levelxlower}
Let $0<c<1$ be a fixed absolute constant. Then, the projection of the level set $\{(x,y)\in\Om: u(x,y) = c\}$ onto the $x$-axis has length bounded from below by an absolute constant multiplied by $L_2$.
\end{prop}
\begin{proof}{Proposition \ref{prop:levelxlower}}
We first prove the proposition for $c=1/4$. By applying Proposition \ref{prop:levelylower} with $c=\tfrac{1}{2}$, there exists a point $x=x_0$, and an interval $J$ of length equal to $2c^*L_1$ for a constant $c^*>0$, such that $u(x_0,y) \geq \tfrac{1}{2}$ for all $y$ in $J$. Therefore, 
\begin{align} \label{eqn:levelxlower1}
\int_{J}u(x_0,y)^2 \ud y \geq \tfrac{1}{4}c^*L_1,
\end{align}
Extending $u(x,y)$ to be zero outside of $\Om$, for any other $x$, we can write
\begin{align*}
 u(x,y) = u(x_0,y) + \int^x_{x_0}\pa_tu(t,y) \ud t,
 \end{align*}
and so, 
\begin{align*}
 u(x,y)^2 \geq \tfrac{3}{4}u(x_0,y)^2 - C_1|x-x_0|\int_{I(x)}(\pa_tu(t,y))^2 \ud t,
 \end{align*}
 where $I(x)$ consists of those points between $x_0$ and $x$.  Integrating this over $y \in J$, we find that 
 \begin{align} \label{eqn:levelxlower2}
 \int_{J} u(x,y)^2 \ud y \geq \tfrac{3}{4} \int_{J} u(x_0,y)^2 \ud y - C_1|x-x_0|\int_{J} \int_{I(x)}(\pa_tu(t,y))^2 \ud t \ud y
 \end{align}
 By \eqref{eqn:levelxlower1}, the first term on the right hand side of \eqref{eqn:levelxlower2} is bounded from  below by $\tfrac{3}{8}c^*L_1$. Provided $|x-x_0| \leq c_2L_2$ for $c_2>0$ sufficiently small, we can use Proposition \ref{prop:xdecayupper} to show that the second term on the right hand side of \eqref{eqn:levelxlower2} is bounded above by
 \begin{align*}
 C_1|x-x_0|L_1L_2^{-1},
 \end{align*}
 for an absolute constant $C_1>0$. Thus, if $|x-x_0| \leq c_3L_2$ for $c_3>0$ sufficiently small, we can ensure that
 \begin{align*}
 \int_{J}u(x,y)^2 \ud y \geq \tfrac{3}{16}c^*L_1 - \tfrac{1}{16}c^*L_1 = \tfrac{1}{8} c^*L_1.
 \end{align*}
 Since the interval $J$ has length equal to $c^*L_1$, this means that for each $x$ with $|x-x_0|\leq c_3L_2$,  $u(x,y)$ must be at least $\tfrac{1}{4}$ at some point $y\in J$. In particular, the level set $\{(x,y)\in\Om: u(x,y) = 1/4\}$ must have length in the $x$-direction of at least $c_3L_2$ as required. A lower bound on the length in the $x$-direction of the other level sets of $u(x,y)$ follows in an analogous way.

\end{proof}

Combining Corollary \ref{cor:firstlocationA}  and Proposition \ref{prop:levelylower} concerning the height of the level sets in the $y$-direction with Propositions \ref{prop:levelxupper} and \ref{prop:levelxlower} concerning the length of the level sets in the $x$-direction we have established the following: For any $c$ with $c_1<c<1-c_1$, the projections of the level sets $\{(x,y)\in\Om: u(x,y) = c\}$ onto the $y$ and $x$-axes are of lengths comparable to $L_1$ and $L_2$ respectively, and moreover, the inner radius of the corresponding superlevel set is comparable to $L_1$ while the diameter is comparable to $L_2$. This implies that the level sets have the desired shape and  completes the proof of Theorem \ref{thm:shape}.

\section{The Location Of The Maximum Of $u(x,y)$} \label{sec:max}

In Theorem \ref{thm:shape} we have described the shape of the level sets $\{(x,y)\in\Om:u(x,y)=c\}$ where $c$ is bounded away from $0$ and $1$ by $c_1< c< 1-c_1$. In Proposition \ref{prop:Agmon}, we gave an indication of the behaviour of $u(x,y)$ as $c$ becomes small. In this section, we instead want to focus on the behaviour of the eigenfunction near its maximum.

The main aim in this section is to prove the following:
\begin{prop} \label{prop:location}
Suppose that the eigenfunction $u(x,y)$ attains its maximum at the point $(x^*,y^*)$. Then, there exists an absolute constant $c^*>0$ such that
\begin{align*}
V(x^*,y^*) - \la \leq -c^*L_1^{-2}.
\end{align*}
\end{prop}
\begin{proof}{Proposition \ref{prop:location}}
To prove this proposition, we first notice that from Proposition \ref{prop:firstlocation}, we can restrict our attention to the region where $V(x,y) \leq 1+CL_1^{-2}$, for a sufficiently large absolute constant $C>0$. Before we can prove the sharper estimate on the location of the maximum in Proposition \ref{prop:location} we need more information about the first derivatives of $u(x,y)$. We will use Proposition \ref{prop:firstlocation} to obtain a pointwise bound on the first derivatives of $u(x,y)$ near its maximum. 

 To do this, we need to introduce the following function:
\begin{defn} \label{def:J(r)}
Let $K_0(r)$ be the $0$th modified Bessel function of the second kind. Then, for $r>0$, we define the function $J(r)$ as follows: Let $0<c_1<c_2$ be small absolute constants. We first set 
\begin{align*}
 J(r) \coloneqq K_0(r/L_1),
 \end{align*}
for $0< r \leq c_1L_1$, and then we require that $J(r)$ decays smoothly to $0$ on a length scale comparable to $L_1$ for $c_1L_1 \leq r \leq c_2L_1$, and is identically $0$ for $r>c_2L_1$.
\end{defn}
\begin{lem}[Properties of $J(r)$] \label{lem:J(r)}
For $0<r \leq c_1L_1$, $J(r)$ satisfies the equation,
\begin{align*}
 \frac{1}{r^{2}} \left( r^2 \frac{d^2}{dr^2} + r\frac{d}{dr} \right) J(r) = L_1^{-2} J(r),
 \end{align*}
and for $c_1L_1\leq r \leq c_2L_1$,
\begin{align*}
 \frac{d^m}{dr^m}J(r) \leq CL_1^{-m}
 \end{align*}
for $m\leq 3$. Moreover, $J(r)$ has a singularity equal to a multiple of $\log r$ as we approach $r=0$.
\end{lem}
\begin{proof}{Lemma \ref{lem:J(r)}}
These properties follow immediately  from the definition of $J(r)$ in Definition \ref{def:J(r)} and the corresponding properties of the modified Bessel function $K_0(r)$.
\end{proof}

We will use the function $J(r)$ with $r$ defined by
\begin{align} \label{eqn:rdefn}
r^2 \coloneqq (x-x')^2 + (y-y')^2
\end{align}
to obtain a pointwise bound on the first derivatives of $u(x,y)$.
\begin{prop} \label{prop:pointwisederivy}
Fix an absolute constant $c$ with $0<c<1$. There exists an absolute constant $C>0$ such that for any point $(x',y')$ with $u(x',y') \geq c$, we have the bound
\begin{align*}
 |\nabla_{x',y'}u(x',y')| \leq CL_1^{-1}.
 \end{align*}
\end{prop}
\begin{proof}{Proposition \ref{prop:pointwisederivy}}
The strategy of the proof to obtain this pointwise estimate for $\nabla_{x',y'}u(x',y')$ is to use the function $J(r)$ together with the eigenfunction equation
\begin{align*}
 -\Delta_{x,y}u(x,y) + (V(x,y) - \la)u(x,y) = 0,
 \end{align*}
 to obtain an expression for the first derivatives of $u(x',y')$.

We fix $(x',y')$ as in the statement of the proposition, and by Proposition \ref{prop:firstlocation}, we see that by choosing $c_1<c_2$ sufficiently small in the definition of $J(r)$, the support of $J(r)$ is contained in the region where $V(x,y) \leq 1+C_1L_1^{-2}$.

With $r$ as in \eqref{eqn:rdefn}, we begin by considering the integral
\begin{align} \label{eqn:J(r)1}
  \lim_{\eps\to0}\int_{r>\eps}\nabla_{x,y}J(r).\nabla_{x,y}\pa_xu(x,y) \ud x \ud y.
  \end{align}
  Since $J(r)$ has a singularity of the form $\log r$ at $r=0$, $\nabla_{x,y}J(r)$ is integrable and the above limit is well-defined. 
  
  Integrating by parts one time to move the derivative away from $\pa_{x}u(x,y)$, we obtain a boundary term at $r=0$, and the integral in \eqref{eqn:J(r)1} becomes
  \begin{align} \label{eqn:J(r)2}
 \pa_xu(x',y') - \lim_{\eps\to0}\int_{r>\eps} \Delta_{x,y}J(r) \pa_x u(x,y) \ud x \ud y.
 \end{align}
 Note that by the support properties of $J(r)$, there are no other boundary terms appearing from this integration by parts.
  
  By Lemma \ref{lem:J(r)}, $\Delta_{x,y}J(r) = L_1^{-2}J(r)$ for $r\leq c_1L_1$ and is $\leq C_1L_1^{-2}$ elsewhere. Thus, we can integrate by parts again in the integral in \eqref{eqn:J(r)2} to get
  \begin{align} \label{eqn:J(r)3}
 -  \lim_{\eps\to0}\int_{r>\eps} \pa_x\Delta_{x,y}J(r)  u(x,y) \ud x \ud y.
  \end{align}
The function  $ \Delta_{x,y}J(r) = L_1^{-2}J(r)$ only has a logarithmic singularity at $r=0$, and so for this integral we do not get a boundary term at $r=0$ in the integration by parts. Instead, since $J(r)$ is supported in a region of area $L_1^2$ and $u(x,y) \leq 1$ everywhere, this integral is bounded by $CL_1^{-1}$ for an absolute constant $C>0$.

Using this bound we see that the integral in \eqref{eqn:J(r)1} is equal to $\pa_{x}u(x',y')$ plus a contribution which is bounded by $CL_1^{-1}$.

We can also write the integral in \eqref{eqn:J(r)1} as
\begin{align*}
 \lim_{\eps\to0}\int_{r>\eps}\pa_xJ(r)\pa^2_xu(x,y) + \pa_y J(r) \pa_y\pa_xu(x,y) \ud x \ud y.
  \end{align*}
  Using the eigenfunction equation, we can rewrite this as
  \begin{align} \label{eqn:J(r)4}
   \lim_{\eps\to0}\int_{r>\eps}-\pa_xJ(r)\pa^2_yu(x,y) +\pa_xJ(r)(V(x,y) - \la)u(x,y) + \pa_y J(r) \pa_y\pa_xu(x,y) \ud x \ud y . 
  \end{align} 
We consider the contribution to this integral from
\begin{align} \label{eqn:J(r)5}
  \lim_{\eps\to0}\int_{r>\eps}-\pa_xJ(r)\pa^2_yu(x,y)  + \pa_y J(r) \pa_y\pa_xu(x,y) \ud x \ud y.
\end{align}
If we integrate by parts in $y$ in the first term and in $x$ in the second term, we find that that the boundary terms at $r=0$ vanish, and there are no other boundary terms. Therefore, the integral in \eqref{eqn:J(r)5} is equal to
\begin{align*}
  \lim_{\eps\to0}\int_{r>\eps}\pa_y\pa_xJ(r)\pa_yu(x,y)  - \pa_x\pa_y J(r) \pa_yu(x,y) \ud x \ud y = 0.
\end{align*}
Thus, from \eqref{eqn:J(r)4}, the original integral in \eqref{eqn:J(r)1} becomes
\begin{align*}
  \lim_{\eps\to0}\int_{r>\eps} \pa_xJ(r)(V(x,y) - \la)u(x,y)  \ud x \ud y =  \int_{\Om} \pa_xJ(r)(V(x,y) - \la)u(x,y)  \ud x \ud y .
\end{align*}
This means that we have shown that 
\begin{align} \label{eqn:J(r)6}
\pa_xu(x',y') =   \int_{\Om} \pa_xJ(r)(V(x,y) - \la)u(x,y)  \ud x \ud y,
\end{align}
plus a contribution which is bounded by $CL_1^{-1}$.

Since $u(x',y') \geq c$, we know from Proposition \ref{prop:firstlocation}  that $|V(x,y) - \la| \leq C_1L_1^{-2}$ on the support of $J(r)$. Therefore, the integral on the right hand side of \eqref{eqn:J(r)6} is bounded by
\begin{align*}
C_1L_1^{-2} \int_{\Om} \left| \pa_xJ(r) \right| u(x,y) \ud x \ud y \leq CL_1^{-1}.
\end{align*}
This gives the required bound for $\pa_xu(x',y')$, and the bound for $\pa_yu(x',y')$ follows in exactly the same way. 
\end{proof}

We recall from Theorem \ref{thm:shape} that the level sets of $u(x,y)$ are of height comparable to $L_1$ in the $y$-direction and of length comparable to $L_2$ in the $x$-direction. Therefore, this is consistent with the derivative bound for $\pa_yu(x,y)$ from Proposition \ref{prop:pointwisederivy} above. However, since in general we have $L_2\gg L_1$, we want to improve the bound given for $\pa_xu(x,y)$.

To do this we first prove a corollary of Propositions \ref{prop:firstlocation} and \ref{prop:pointwisederivy} about the location of the level sets of $u(x,y)$ in the $x$-direction.
\begin{cor} \label{cor:firstlocation}
Fix an absolute constant $c$, with $0<c<1$. Then, there exists an absolute constant $C>0$ such that for any point $(x,y)$ in the level set $\{(x,y)\in\Om:u(x,y) = c\}$, there exists points $(x_1,y_1)$ with $x_1-x$ both positive or negative that $|x_1 - x|$ is comparable to $L_2$  and $V(x_1,y_1) \leq 1+CL_1^{-2}$.
\end{cor}
\begin{proof}{Corollary \ref{cor:firstlocation}}
Let $(x',y')$ be the left most point of the level set $\{(x,y)\in\Om:u(x,y) = c\}$. That is, $u(x',y') = c$, and $u(x,y) < c$ for any point $(x,y)$ with $x<x'$. Then, by the bound on $\pa_yu(x,y)$ from Proposition \ref{prop:pointwisederivy}, we know that
\begin{align*}
u(x',y) \geq \tfrac{1}{2}c,
\end{align*}
for all $y$ in an interval $J$ of length comparable to $L_1$. In particular,
\begin{align*}
\int_{J}u(x',y)^2 \ud y \geq \tilde{c}L_1,
\end{align*}
for some absolute constant $\tilde{c}>0$. Therefore, using the $L^2$-bound on $\pa_xu(x,y)$ from Proposition \ref{prop:xdecayupper}, exactly as in the proof of Proposition \ref{prop:levelxlower}, we find that
\begin{align*}
\int_{J}u(x,y)^2 \ud y \geq \tilde{c}L_1/2,
\end{align*}
for all $x<x'$ with $x'-x>c_2L_2$, for an absolute constant $c_2>0$. Thus, for any such $x=x_1$, there exists a $y_1 \in J$ such that $u(x_1,y_1)$ is bounded below by an absolute constant. In particular, by Proposition \ref{prop:firstlocation}, this means that $V(x_1,y_1) \leq 1+ CL_1^{-2}$, and this completes the proof of the corollary. 

\end{proof}

This corollary allows us to partially improve the estimate on $\pa_xu(x,y)$ from Proposition \ref{prop:pointwisederivy}.  We recall that the maximum of $u(x,y)$ of $1$ is achieved at $(x,y) = (x^*,y^*)$.
\begin{prop} \label{prop:pointwisederivx}
Consider the level set $\{(x,y) \in\Om: u(x,y) = c\}$ for a fixed constant $c$, with $c_1<c<1-c_1$. Then, on part of the upper and lower boundaries of this level set, with $x$ in an interval around $x=x^*$ of length comparable to $L_2$, we have the pointwise derivative bound
\begin{align*}
 |\pa_xu(x,y)| \leq CL_2^{-1},
 \end{align*}
for some absolute constant $C>0$.
\end{prop}
\begin{proof}{Proposition \ref{prop:pointwisederivx}}
For fixed $c$, with $c_1<c<1-c_1$, the level set $\{(x,y)\in\Om: u(x,y) = c\}$ extends a distance comparable to $L_2$ in the $x$-direction on either side of the point $x^*$. Let $y= g(x)$ be a parametrisation of the upper boundary of the level set. Since this level set is the boundary of a convex set of height comparable to $L_1$ in the $y$-direction, this means that 
\begin{align} \label{eqn:pointwisederivx1}
|g'(x)| \leq C_1L_1L_2^{-1}
\end{align}
for all $x$ in an interval around $x^*$ of length comparable to $L_2$. On this part of the level set we have
\begin{align*}
u(x,g(x)) = c,
\end{align*}
and so differentiating this with respect to $x$ gives
\begin{align*}
\pa_x u(x,g(x)) = - g'(x)\pa_yu(x,g(x)).
\end{align*}
Combining the estimate for $g'(x)$ in \eqref{eqn:pointwisederivx1} with the pointwise bound on $\pa_yu(x,y)$ from Proposition \ref{prop:pointwisederivy} gives the required bound for $\pa_xu(x,y)$.

\end{proof}

To prove that the maximum of $u(x,y)$ has the required properties of Proposition \ref{prop:location}, we will also need to improve the derivative estimate of Proposition \ref{prop:pointwisederivy} for those points $(x,y)$ near the maximum. To do this we will prove the following two propositions.
\begin{prop} \label{prop:yloweruniform}
Let $\eps>0$ be sufficiently small. Then, the convex superlevel set $\{(x,y)\in\Om: u(x,y) \geq 1-\eps\}$ has inner radius at least $c {\eps}^{1/2}L_1$, for a small absolute constant $c>0$, which is independent of $\eps$.
\end{prop}

\begin{prop} \label{prop:maxderivy}
Let $u(x',y') = 1-\eps$, where $\eps>0$ is sufficiently small. Then, there exists an absolute constant $C>0$, which is independent of $\eps$, such that
\begin{align*}
 |\nabla_{x,y}u(x',y')| \leq C{\eps}^{1/2}L_1^{-1}.
 \end{align*}
In particular, this is an improvement on Proposition \ref{prop:pointwisederivy} for small $\eps$.
\end{prop}
\begin{rem}
Consider a function $f(x)$ defined on the interval $[-L_1,L_1]$ by
\begin{align*}
f(y) \coloneqq 1 - L_1^{-2}y^2. 
\end{align*}
Then,  $f(y) = 1-\eps$ for $y = \pm \eps^{1/2}L_1$. In other words, the interval on which $f(y) \geq 1-\eps$ has length $2\eps^{1/2}L_1$. Thus, the lower bound on the inner radius of the superlevel set of $u(x,y)$ in Proposition \ref{prop:yloweruniform}  is consistent with the eigenfunction $u(x,y)$ being bounded from below by such a parabola as we move away from the maximum in the $y$-direction.

Also,  at the two points where $f(y)$ is equal to $1-\eps$ we have the derivative bound
\begin{align*}
f'(y) = -2L_1^{-2}y  = \mp 2\eps^{1/2}L_1^{-1}. 
\end{align*}
Therefore, the derivative bound in Proposition \ref{prop:maxderivy} is again consistent with the eigenfunction $u(x,y)$ being bounded from below by such a parabola as we move away from the maximum in the $y$-direction. It is also consistent with having a bound comparable to $L_1^{-2}$ on the second derivatives of the eigenfunction.
\end{rem}

\begin{proof}{Proposition \ref{prop:yloweruniform}}
Suppose that the proposition does not hold. Then, after a translation and rotation, we may assume that the level set $\{(x,y)\in\Om:u(x,y) = 1-\eps\}$ lies between the lines $y=\pm\alpha$, where $\alpha < c_1 \eps^{1/2}L_1$ for a small absolute constant $c_1$ to be chosen later. Note that we are considering the length scale $L_1$, and we will not use any of the properties of $\la$ and $u(x,y)$ that depend on $L_2$. This means that we do not have to fix the orientation of $\Om_{L_1^{-2}} = \{(x,y)\in\Om:V(x,y)=1+L_1^{-2}\}$ and so there is no issue in applying the rotation above.

We will use the comparison function
\begin{align*}
 v_2(x,y) \coloneqq \left(1-\tfrac{1}{2}\eps \right) \sin\left(\frac{\pi}{2} + \frac{{\eps}^{1/2}y}{C_1\alpha}\right),
 \end{align*}
where $C_1>0$ is chosen so that
\begin{align*}
 v_2(x,y) \geq 1-\eps
\end{align*}
for  all $(x,y)$ with $y = \pm\alpha$. This means that
\begin{align} \label{eqn:yloweruniform1}
u(x,y) \leq v_2(x,y)
\end{align}
for any $(x,y)\in\Om$ with $y=\pm \alpha$. Also, for all $(x,y)\in\pa\Om$ with $-\alpha \leq y \leq \alpha$, we see that
\begin{align} \label{eqn:yloweruniform1A}
u(x,y) = 0 \leq v_2(x,y).
\end{align}
 Moreover, the function $v_2(x,y)$ satisfies
\begin{align*}
\left(\Delta_{x,y} - V(x,y) + \la \right)v_2(x,y) = - \left(\frac{\eps}{C_1^2\alpha^2}\right)v_2(x,y) + (\la-V(x,y))v_2(x,y).
\end{align*}
We know that $\la-V(x,y) \leq C_2^2L_1^{-2}$, for some absolute constant $C_2>0$, and so
\begin{align} \label{eqn:yloweruniform2}
\left(\Delta_{x,y} - V(x,y) + \la\right) v_2(x,y) <\left(- \left(\frac{\eps}{C_1^2\alpha^2}\right) +C_2^2L_1^{-2}\right)v_2(x,y) .
\end{align}
Provided $\alpha <c_1\eps^{1/2}L_1$, for $c_1$ sufficiently small (depending only on $C_1$ and $C_2$), we can ensure from \eqref{eqn:yloweruniform2} that
\begin{align} \label{eqn:yloweruniform3}
\left(\Delta_{x,y} - V(x,y) + \la\right)v_2(x,y) < 0.
\end{align}
Combining \eqref{eqn:yloweruniform1}, \eqref{eqn:yloweruniform1A} and \eqref{eqn:yloweruniform3}, we see that by the generalised maximum principle in Proposition \ref{prop:GMP} that
\begin{align*}
u(x,y) \leq v_2(x,y) \qquad \text{ for } (x,y)\in\Om \text{ with } -\alpha \leq y \leq \alpha.
\end{align*}
However, $v_2(x,y) \leq 1-\frac{1}{2}\eps$ everywhere, whereas we know that $u(x,y)$ attains its maximum of $1$ for some $(x,y) \in\Om$ with $-\alpha \leq y \leq \alpha$. This contradiction completes the proof of the proposition. 

\end{proof}

\begin{proof}{Proposition \ref{prop:maxderivy}}

In the proof of Proposition \ref{prop:maxderivy}, as well as Proposition \ref{prop:yloweruniform},  we will also make use of the following proposition:
\begin{prop} \label{prop:A4}
Suppose that the function $v(x,y)$ satisfies
\begin{align*}
\Delta_{x,y} v(x,y) + W_1(x,y)v(x,y) = W_2(x,y),
\end{align*}
in a convex domain $D$, with $|W_1(x,y)|$, $|W_2(x,y)| \leq C_1$. Let $(x_0,y_0) \in \pa D$, and let $B_s$ be the disc of radius $s$ around $(x_0,y_0)$. If  we have $v \geq 0$ in $D$, then
\begin{align*}
|\nabla_{x,y}v(x_0,y_0)| \leq C \sup_{B_{1/4}\cap D}v.
\end{align*}
\end{prop}
\begin{proof}{Proposition \ref{prop:A4}}
This proposition is a variation of Proposition A.4 in \cite{GJ2}, and is proved in exactly the same way.
\end{proof}

We can now begin the proof of Proposition \ref{prop:maxderivy}. After a translation, we may assume that $u(x,y) = 1$ at the point $(x,y) = (0,0)$, and we define the function $u_1(x,y)$ by
\begin{align*}
 u_1(x,y) \coloneqq u\left({\eps}^{1/2}L_1x,{\eps}^{1/2}L_1y \right). 
\end{align*}
This satisfies the equation
\begin{align*}
 -\Delta_{x,y}u_1(x,y) + \eps L_1^2\left(V\left({\eps}^{1/2}L_1x,{\eps}^{1/2}L_1y \right) - \la\right)u_1(x,y) = 0.
 \end{align*}
We next define $u_2(x,y)$ by
\begin{align*}
 u_2(x,y) \coloneqq \eps^{-1}(u_1(x,y) - (1-\eps)).
 \end{align*}
 Consider the convex superlevel set $\{(x,y)\in\Om:u_1(x,y) \geq 1-\eps\}$. By Proposition \ref{prop:yloweruniform} and the definition of the function $u_1(x,y)$, this set has inner radius bounded from below by an absolute constant. 
Moreover, on the boundary of this set the function $u_2(x,y)$ is equal to $0$ and inside the superlevel set it takes all values between $0$ and $1$. It satisfies the equation
\begin{align} \label{eqn:maxderivy1} \nonumber
 -\Delta_{x,y}u_2(x,y) & +  \eps L_1^2\left(V\left({\eps}^{1/2}L_1x,{\eps}^{1/2}L_1y \right)  - \la\right)u_2(x,y) \\
& + (1-\eps)L_1^2\left(V\left({\eps}^{1/2}L_1x,{\eps}^{1/2}L_1y \right)  - \la\right) = 0. 
\end{align}
From Proposition \ref{prop:firstlocation}, we certainly know that inside the convex set $\{(x,y)\in\R^2:u_1(x,y) = 1-\eps\}$ we have the bounds
\begin{align*}
\left|L_1^2 \left(V\left({\eps}^{1/2}L_1x,{\eps}^{1/2}L_1y \right)  - \la\right)\right| \leq C_1,
\end{align*}
for an absolute constant $C_1>0$. Therefore, in \eqref{eqn:maxderivy1} we have the bounds
\begin{align*}
\left|\eps L_1^2\left(V\left({\eps}^{1/2}L_1x,{\eps}^{1/2}L_1y \right)  - \la\right)\right| \leq C_1, \qquad \left|(1-\eps)L_1^2\left(V\left({\eps}^{1/2}L_1x,{\eps}^{1/2}L_1y \right)  - \la\right)\right|  \leq C_1.
\end{align*}
We can thus use Proposition \ref{prop:A4}  with $D = \{(x,y)\in\R^2:u_1(x,y) \geq 1-\eps\}$ to conclude that
\begin{align*}
|\nabla_{x,y}u_2(x,y)| \leq C,
\end{align*}
for $(x,y)$ on the boundary of $D$. Recalling the definitions of $u_1(x,y)$ and $u_2(x,y)$, this shows that
\begin{align*}
 |\nabla_{x,y}u(x',y')| \leq C \eps(\sqrt{\eps}L_1)^{-1} = C\sqrt{\eps}L_1^{-1}
\end{align*}
as required. 
\end{proof}

We need to write down one more consequence of the log concavity of $u(x,y)$, and then we can start to complete the proof of Proposition \ref{prop:location}.
\begin{lem} \label{lem:logcon}
We have an upper bound on $\pa_x^2u(x,y)$ of the form
\begin{align*}
 \pa_x^2u(x,y) \leq \frac{(\pa_xu(x,y))^2}{u(x,y)},
 \end{align*}
and an analogous upper bound for $\pa^2_yu(x,y)$.
\end{lem}
\begin{proof}{Lemma \ref{lem:logcon}}
Differentiating the function $\log u(x,y)$ twice with respect to $x$, we find that
\begin{align} \label{eqn:logcon1}
 \pa_x^2(\log u(x,y)) =\frac{ \pa_x^2u(x,y)}{u(x,y)} - \frac{(\pa_xu(x,y))^2}{u(x,y)^2}.
 \end{align}
However, the eigenfunction $u(x,y)$ is log concave, and so
\begin{align} \label{eqn:logcon2}
 \pa_x^2(\log u(x,y)) \leq 0.
 \end{align}
Combining \eqref{eqn:logcon1} and \eqref{eqn:logcon2} gives the desired bound.
\end{proof}

To complete the proof of Proposition \ref{prop:location} we split into two cases.

\subsection{Case $1$: $L_2 \gg L_1$}

We will first assume that $L_2 \gg L_1$, or in other words, $L_2 \geq \tilde{C}L_1$ for a large absolute constant $\tilde{C}>0$, which we will specify below.

After a translation, we may assume that $u(x,y)$ attains its maximum at the point $(0,0)$. Let $u(0,-\alpha L_1)$, $u(0,\beta L_1) = 1/2$, where we know from Theorem \ref{thm:shape} that $\alpha$ and $\beta$ are both comparable to $1$. Moreover, without loss of generality, we may assume that $\pa_yV(0,0) \geq 0$. We want to study the integral
\begin{align} \label{eqn:max1}
\int_0^{\beta L_1} (\beta L_1 - y)(V(0,y) -\la)u(0,y) \ud y.
\end{align}
We will prove the following two lemmas:
\begin{lem} \label{lem:max1lower}
The integral in \eqref{eqn:max1} is bounded from below by
\begin{align*} 
\tfrac{1}{2}\beta^2L_1^2(V(0,0) - \la).
\end{align*}
\end{lem}
\begin{lem} \label{lem:max1upper}
The integral in \eqref{eqn:max1} is bounded from above by $-\tfrac{1}{4}$.
\end{lem}
Combining Lemmas \ref{lem:max1lower} and \ref{lem:max1upper}, we see that
\begin{align*}
V(0,0) - \la \leq -\tfrac{1}{2}\beta^{-2}L_1^{-2},
\end{align*}
and so we have established Proposition \ref{prop:location} in the case where $L_2\gg L_1$. Thus, we are left to prove these two lemmas. 

\begin{proof}{Lemma \ref{lem:max1lower}}
Since $\pa_yV(0,0) \geq 0$, and $V(x,y)$ is convex, we must have $\pa_yV(0,y) \geq0$ for all $y \geq0$. Thus, for $0 \leq y\leq \beta L_1$ we have 
\begin{align*}
V(0,y)-\la \geq V(0,0) -\la
\end{align*}
and $u(0,y) \geq \tfrac{1}{2}$. The lower bound then follows immediately.
\end{proof}

\begin{proof}{Lemma \ref{lem:max1upper}}
Since $u(x,y)$ satisfies the eigenfunction equation, we can rewrite \eqref{eqn:max1} as
\begin{align} \label{eqn:max1upper1}
 \int_0^{\beta L_1} (\beta L_1 - y)\Delta_{x,y}u(0,y) \ud y.
 \end{align}
Let us first consider the term containing a factor of $\pa_y^2u(0,y)$. Since $\pa_yu(0,0) = 0$, integrating by parts, this becomes
\begin{align*}
 \int_0^{\beta L_1} \pa_yu(0,y) \ud y = -u(0,0) + u(0,\beta L_1) = -\tfrac{1}{2}.
 \end{align*}
We are left to bound the contribution to \eqref{eqn:max1upper1} from
\begin{align} \label{eqn:max1upper2}
\int_0^{\beta L_1} (\beta L_1 - y)\pa_x^2u(0,y) \ud y,
\end{align}
and to do this we will use Proposition \ref{prop:maxderivy} and Lemma \ref{lem:logcon}.

We first fix $0<c_1<\beta$ such that $u(0,c_1L_1) = 1- c_2$ for a small constant $c_2>0$. Then, by Proposition \ref{prop:maxderivy}, we have the bound
\begin{align*}
|\pa_xu(0,y)| \leq  Cc_2^{1/2}L_1^{-1}
\end{align*}
for all $y$ with $0 \leq y \leq c_1L_1$, with $C$ independent of $c_2$. In particular, by Lemma \ref{lem:logcon}, we have
\begin{align*}
\pa_x^2u(x,y) \leq Cc_2L_1^{-2},
\end{align*}
and so
\begin{align}  \label{eqn:max1upper3}
\int_0^{c_1 L_1}(\beta L_1 - y)\pa_x^2 u(0,y) \ud y  \leq \frac{1}{8},
\end{align}
provided $c_1>0$ is sufficiently small. We now need to consider the part of the integral in \eqref{eqn:max1upper2} with $y$ between $c_1L_1$ and $\beta L_1$. For $y$ in this range, we can use the derivative bound on $\pa_xu(x,y)$ from Proposition \ref{prop:pointwisederivx}, which after applying Lemma \ref{lem:logcon} gives
\begin{align*}
\pa_x^2u(0,y) \leq CL_2^{-2}.
\end{align*}
Therefore, provided $L_2/L_1$ is sufficiently large, we also have the bound
\begin{align} \label{eqn:max1upper4}
\int_{c_1L_1}^{\beta L_1} (\beta L_1 - y)\pa_x^2u(0,y) \ud y \leq \frac{1}{8}.
\end{align}
Combining \eqref{eqn:max1upper3} and \eqref{eqn:max1upper4}, we see that the integral in \eqref{eqn:max1upper2} is bounded above by $\tfrac{1}{4}$, and hence
\begin{align*}
\int_0^{\beta L_1} (\beta L_1 - y) \Delta_{x,y}u(0,y) \ud y < -\frac{1}{4}.
\end{align*}
This completes the proof of the lemma.

\end{proof}
As we discussed after the statement of Lemmas \ref{lem:max1lower} and \ref{lem:max1upper}, this completes the proof of Proposition \ref{prop:location} in the case where $L_2\gg L_1$.

\subsection{Case $2$: $L_1$ and $L_2$ are comparable}

In this case, we assume that $V(x,y)$ attains its minimum of $1$ at $(0,0)$, and we rescale the eigenfunction $u(x,y)$ by $L_1$ in the $x$ and $y$-directions. Then, $\tilde{u}(x,y) = u(L_1x,L_1y)$ satisfies the equation 
\begin{align*}
\Delta_{x,y}\tilde{u}(x,y) = \tilde{F}(x,y)\tilde{u}(x,y),
\end{align*}
where $\tilde{F}(x,y) = L_1^2\left(V(L_1x,L_1y) - \la\right)$. We know that $\tilde{u}(x,y)$ must attain its maximum at some point inside the region where $\tilde{F}(x,y) \leq 0$. We now want to improve this estimate on the location of the maximum of $\tilde{u}(x,y)$.

\begin{lem} \label{lem:umaxL1}
In the case where $L_1$ and $L_2$ are comparable there exists a small absolute constant $\eps>0$ such that $\tilde{u}(x,y)$ attains its maximum at a distance at least $\eps^{2/3}$ from the boundary of the region where $\tilde{F}(x,y) \leq 0$.
\end{lem}
The function $\tilde{F}(x,y)$ is convex, has a minimum of $-c$ for some $c>0$, and the inner radius and diameter of the region where $\tilde{F}(x,y) \leq 0$ are comparable to $1$. Therefore, this lemma implies that
\begin{align*}
\tilde{F}(x,y) \leq -c_1 
\end{align*}
for some constant $c_1>0$ at the point where $\tilde{u}(x,y)$ attains its maximum. Returning to $u(x,y)$ and $F(x,y) = V(x,y) - \la$, this proves Proposition \ref{prop:location} in the case where $L_1$ and $L_2$ are comparable. Thus, we are left to prove Lemma \ref{lem:umaxL1}.
\begin{proof}{Lemma \ref{lem:umaxL1}}
Suppose that $\tilde{u}(x,y) = 1$ at a point $(x,y)$ within $\eps^{2/3}$ of the boundary of the region $\{(x,y):\tilde{F}(x,y) \leq 0\}$, where $\eps>0$ is a small constant that we will specify later. 

The function $\tilde{F}(x,y)$ attains a negative minimum, is convex, and is negative on a region with diameter comparable to $1$. Therefore, $|\nabla_{x,y}\tilde{F}(x,y) | \geq c_1$ on the set where $\tilde{F}(x,y)  = 0$. In particular, we have the lower bound
\begin{align} \label{eqn:umaxL11}
\tilde{F}(x,y) \geq \eps^{2/3}
\end{align}
when we are at a distance comparable to $\eps^{2/3}$ outside the region where $\tilde{F}(x,y) \leq 0$. Also, by the pointwise derivative bounds on $u(x,y)$ from Proposition \ref{prop:pointwisederivy}, $\tilde{u}(x,y)$ is comparable to $1$ at a distance of $\eps^{2/3}$ from its maximum.

By the log concavity of $u(x,y)$, we know that 
\begin{align*}
 \Delta_{x,y}\log u(x,y) = \frac{\Delta_{x,y}u(x,y)}{u(x,y)} - \frac{|\nabla_{x,y}u(x,y)|^2}{u(x,y)^2}  \leq 0.
 \end{align*}
Rearranging, and using the eigenfunction equation, this tells us that
\begin{align} \label{eqn:umaxL12}
 |\nabla_{x,y}u(x,y)|^2 \geq (V(x,y) - \la)u(x,y)^2 = F(x,y)u(x,y)^2.
\end{align}
Thus, from \eqref{eqn:umaxL11} and \eqref{eqn:umaxL12}, we have the lower bound
\begin{align} \label{eqn:umaxL13}
|\nabla_{x,y}\tilde{u}(x,y)| \geq \tilde{c}\eps^{1/3}
\end{align}
for some point $(x_1,y_1)$ which is at a distance comparable to $\eps^{2/3}$ from the point where $\tilde{u}(x,y)$ attains its maximum.

However, by Proposition \ref{prop:maxderivy}, we know that when $\tilde{u}(x',y') = 1-\eps$, we have the derivative bound
\begin{align} \label{eqn:umaxL14}
|\nabla_{x,y}\tilde{u}(x',y')| \leq C\eps^{1/2}.
\end{align}
For $\eps>0$ sufficiently small, we have $\tilde{c}\eps^{1/3} > C\eps^{1/2}$, and so from \eqref{eqn:umaxL13} and \eqref{eqn:umaxL14}, we see that $\tilde{u}(x_1,y_1) < 1-\eps$.

In other words, the function $\tilde{u}(x,y)$ changes from $1$ to $1-\eps$ on a line segment of length comparable to $\eps^{2/3}$. However, using Proposition \ref{prop:maxderivy} again, we know that
\begin{align*}
|\nabla_{x,y}\tilde{u}(x,y)| \leq C\eps^{1/2}
\end{align*}
whenever $\tilde{u}(x,y) \geq 1- \eps$, and so $\tilde{u}(x,y)$ can only change by an amount comparable to 
\begin{align*}
\eps^{1/2}\eps^{2/3} = \eps^{7/6},
\end{align*}
on this line segment of length $\eps^{2/3}$. For $\eps>0$ sufficiently small, we see that $\eps^{7/6} \ll \eps$, and so this gives us a contradiction.
\end{proof}

As we discussed after the statement of Lemma \ref{lem:umaxL1} this also completes the proof of Proposition \ref{prop:location} in the case where $L_1$ and $L_2$ are comparable.
\end{proof}

Let us finish by giving two consequences of the location of the maximum of $u(x,y)$ derived in Proposition \ref{prop:location}. The first is to show that the lower bound on the inner radius of the superlevel set $\{(x,y)\in\Om:u(x,y) \geq 1-\eps\}$ given in Proposition \ref{prop:yloweruniform} is sharp.

\begin{cor} \label{cor:yupperuniform}
For $\eps>0$ sufficiently small, the superlevel set  $\{(x,y)\in\Om:u(x,y) \geq 1-\eps\}$ has inner radius at most $C_1\eps^{1/2}L_1$, where $C_1>0$ is an absolute constant.
\end{cor}
\begin{proof}{Corollary \ref{cor:yupperuniform}}
By Proposition \ref{prop:location}, we know that $V(x,y) -\la \leq -c^*L_1^{-2}$ at the maximum of $u(x,y)$. Moreover, by Proposition \ref{prop:firstlocation}, inside the level set $\{(x,y)\in\Om:u(x,y) =1/2\}$, we have the bound
\begin{align*}
V(x,y) - \la \leq CL_1^{-2}.
\end{align*}
Since this level set has height  comparable to $L_1$ in the $y$-direction and length comparable to $L_2$ in the $x$-direction, by the convexity of the potential, we have
\begin{align} \label{eqn:yupperuniform1}
V(x,y) - \la \leq -\frac{1}{2}c^*L_1^{-2}
\end{align}
on a region of height $c_1L_1$ and length $c_2L_2$ in the $y$ and $x$-directions around the maximum.

Suppose that the superlevel set $\{(x,y)\in\Om:u(x,y) \geq 1-\eps\}$  has inner radius at least $\alpha$, where $\alpha = C_1\eps^{1/2}L_1$ for some large absolute constant $C_1>0$. Then, this superlevel set contains a circle of radius $\alpha$, and after a translation, centre at $(0,0)$.

Let $J_0(r)$ be the $0$th Bessel function of the first kind for $r>0$. This satisfies $J_0(0) = 1$, $J_0'(0) = 0$ and $J_0''(0) = -1/2$, as well as the equation
\begin{align} \label{eqn:yupperuniform2}
 r^2J_0''(r) + rJ_0'(r) = -r^2J_0(r) .
 \end{align}
 Setting $r^2 = x^2+y^2$, we will use the comparison function
 \begin{align*}
 v(x,y) = (1+\eps)J_0(C\eps^{1/2}\alpha^{-1}r)
 \end{align*}
for $x^2+y^2 \leq \alpha$. Here $C>0$ is chosen so that $(1+\eps)J_0\left(C\eps^{1/2}\right) \leq 1-\eps$. This is possible for $\eps>0$ sufficiently small, since for small $r$, $J_0(r)$ satisfies
\begin{align*}
 J_0(r) = 1-\tfrac{1}{2}r^2 + O(r^4).
 \end{align*}
 In particular, this ensures that
 \begin{align} \label{eqn:yupperuniform3}
 v(x,y) \leq u(x,y),
 \end{align}
 for $x^2+y^2 = \alpha^2$. By \eqref{eqn:yupperuniform2}, the function $v(x,y)$ also satisfies the equation
 \begin{align*}
  \Delta_{x,y}v(x,y) = -\frac{C^2\eps}{\alpha^2}v(x,y) .
 \end{align*}
 Thus,
 \begin{align}  \label{eqn:yupperuniform4}
 \Delta_{x,y}v(x,y) - (V(x,y) - \la)v(x,y) = -\frac{C^2\eps}{\alpha^2}v(x,y) - (V(x,y) - \la)v(x,y).
 \end{align}
 Provided that we take $\eps>0$ sufficiently small, we can ensure from \eqref{eqn:yupperuniform1} that
 \begin{align*}
 V(x,y) - \la \leq -\frac{1}{2}c^*L_1^{-2}
 \end{align*}
for $x^2 + y^2 \leq \alpha^2$. Therefore, provided $\alpha= C_1\eps^{1/2}L_1$ for $C_1$ sufficiently large, and $x^2+y^2 \leq \alpha^2$, we have
\begin{align*}
 -\frac{C^2\eps}{\alpha^2} - (V(x,y) - \la) \geq \frac{1}{4}c^*L_1^{-2} \geq 0,
\end{align*}
and so from \eqref{eqn:yupperuniform4} we see that
\begin{align}  \label{eqn:yupperuniform5}
 \Delta_{x,y}v(x,y) - (V(x,y) - \la)v(x,y) \geq 0
\end{align}
for $x^2+y^2 \leq \alpha^2$. Combining \eqref{eqn:yupperuniform3} and \eqref{eqn:yupperuniform5}, we can apply the generalised maximum principle from Proposition \ref{prop:GMP} to conclude that
\begin{align*}
v(x,y) \leq u(x,y)
\end{align*}
whenever $x^2 + y^2 \leq \alpha^2$. However, $v(0,0) = 1+\eps$, while $u(x,y) \leq 1$ everywhere, and so this gives us a contradiction. 
\end{proof}

The second consequence of Proposition \ref{prop:location} is to improve the pointwise bound on $\pa_xu(x,y)$ from Proposition \ref{prop:pointwisederivx} in the case where $L_2 \gg L_1$.

\begin{cor} \label{cor:2pointwisederivx}
There exists a constant $c>0$ such that we have the derivative bound
\begin{align*}
|\pa_xu(x,y)| \leq CL_2^{-1},
\end{align*}
for an absolute constant $C$, for all $(x,y)$ in a rectangle of side lengths $cL_2$ and $cL_1$ around the maximum of $u(x,y)$.
\end{cor}
\begin{proof}{Corollary \ref{cor:2pointwisederivx}}
From Corollary \ref{cor:yupperuniform} above, the superlevel sets $\{(x,y)\in\Om:u(x,y) \geq 1-\eps\}$  have inner radius bounded by $C\eps^{1/2}L_1$. Let the maximum of $u(x,y)$ be attained at $(0,0)$.  Then, we saw in the proof of Corollary \ref{cor:yupperuniform} that the sublevel set
\begin{align*}
\{(x,y) \in \Om: V(x,y) - \la \leq -\tfrac{1}{2}c^*L_1^{-2}\},
\end{align*}
contains a rectangle $R$, with centre at $(0,0)$, and of side lengths comparable to $L_2$ and $L_1$ in the $x$ and $y$-directions. We then construct a set $U \subset \Om$ as follows: It consists of the part of the superlevel set $\{(x,y)\in\Om:u(x,y) \geq 1-\tilde{c}\}$ with $x$ restricted to an interval of length $L_2$ around $0$, and $\tilde{c}>0$ sufficiently small so that $U$ is contained within the middle half of the rectangle $R$. 

The boundary of this set $U$ then consists of parts of the upper and lower boundaries of the level set $\{(x,y)\in\Om:u(x,y) = 1-\tilde{c}\}$, and two vertical lines with $x$ fixed. Moreover, by choosing $\tilde{c}$ to be sufficiently small, $U$ is contained between the two lines $y=\pm\tfrac{1}{2} c_1L_1$. We then define a comparison function $W(x,y)$ by
\begin{align*}
W(x,y) = \frac{1}{c_2L_2}\cosh \left(\frac{x\log(L_2/L_1)}{c_3L_2}\right)\cos \left(\frac{\pi y}{2c_1L_1}\right).
\end{align*}
Here $c_2$ and $c_3$ are small absolute constants depending on $c_1$ that we will specify below.
Firstly, we choose $c_2>0$ sufficiently small so that for all $|y| \leq c_1L_1/2$, we have
\begin{align*}
 W(x,y) \geq C_1L_2^{-1}.
 \end{align*}
This absolute constant $C_1$ is chosen so that
\begin{align} \label{eqn:2pointwisederivx1}
|\pa_xu(x,y)| \leq W(x,y)
\end{align}
for all points $(x,y)$ on the curved portion of $\pa U$ consisting of part of the upper and lower boundaries of $\{(x,y)\in\Om:u(x,y) = 1-\tilde{c}\}$. This is possible due to Proposition \ref{prop:pointwisederivx}.

We now let $x = cL_2$, where $c>0$ is chosen so that $x=\pm2cL_2$ is contained in the projection of the set $U$ onto the $x$-axis. Then, for all $|y| \leq c_1L_1/2$, we have the lower bound 
\begin{align*}
W(cL_2,y) \geq   \frac{1}{2c_2L_2}\cosh \left(\frac{c}{c_3}\log(L_2/L_1)\right)\geq  \frac{1}{4c_2L_2}\exp\left(\frac{c}{c_3}\log(L_2/L_1)\right) = \frac{1}{4c_2L_2}\left(\frac{L_2}{L_1}\right)^{c/c_3}. 
\end{align*}
We can thus choose $c_3>0$ sufficiently small, depending on $c$ only, so that
\begin{align*}
W(cL_2,y) \geq C_2L_1^{-1}.
\end{align*}
Here $C_2$ is chosen so that for $|x|\geq cL_2$, $(x,y)\in U$, we have
\begin{align} \label{eqn:2pointwisederivx2}
|\pa_xu(x,y)| \leq W(x,y).
\end{align}
This is possible due to Proposition \ref{prop:pointwisederivy}.

The function $W(x,y)$ satisfies the equation
\begin{align} \label{eqn:2pointwisederivx3}
 \Delta_{x,y}W(x,y) = \left(\left( \frac{\log(L_2/L_1)}{c_3L_2}\right)^2 - \left(\frac{\pi^2}{4c_1^2L_1^2}\right)\right)W(x,y) \leq -\frac{\pi^2}{8c_1^2L_1^2}W(x,y), 
\end{align}
provided $L_2/L_1$ is sufficiently large.

The first derivative $\pa_xu(x,y)$ satisfies
\begin{align} \label{eqn:2pointwisederivx4}
(-\Delta_{x,y} + V(x,y) - \la)\pa_xu(x,y) = -\pa_xV(x,y)u(x,y).
\end{align}
By the convexity of $V(x,y)$, we have the bound $|\pa_xV(x,y)u(x,y)| \leq C_3L_2^{-1}L_1^{-2}$ for all $(x,y) \in U$. Also, $|V(x,y) - \la| \leq C_4L_1^{-2}$.

We will apply the maximum principle to the functions
\begin{align*}
\Psi_{\pm}(x,y) \coloneqq   ( (\pa_xu)_\pm(x,y)+L_2^{-1})/W(x,y),
\end{align*}
where $\pm$ signifies taking the positive or negative part of the function. 

Let $U_{\pm}$ be any connected component of the support of $(\pa_{x}u)_{\pm}$ in $U$. Then, inside $U_{\pm}$, the functions $\Psi_{\pm}(x,y)$ satisfy
\begin{align*} 
 \Delta_{x,y}\Psi_{\pm}(x,y) &+ 2\nabla_{x,y}\log W(x,y).\nabla_{x,y}\Psi_{\pm}(x,y) =\\ \nonumber
  &W(x,y)^{-1}\left( \Delta_{x,y}\left(\left(\pa_x u\right)_{\pm}(x,y) + L_2^{-1}\right)- \left(\left(\pa_xu\right)_{\pm}(x,y)+L_2^{-1}\right)W(x,y)^{-1} \Delta_{x,y}W(x,y)\right),
\end{align*}
which by \eqref{eqn:2pointwisederivx3} and \eqref{eqn:2pointwisederivx4} implies
\begin{align} \label{eqn:2pointwisederivx5}
 & \Delta_{x,y}\Psi_{\pm}(x,y) + 2\nabla_{x,y}\log W(x,y).\nabla_{x,y}\Psi_{\pm}(x,y) \\ \nonumber
 & \geq   W(x,y)^{-1}\left(\pa_xV(x,y)u(x,y)+(V(x,y) - \la)\left((\pa_xu)_{\pm}(x,y)+L_2^{-1}\right) + \tfrac{1}{8}\pi^2 c_1^{-2}L_1^{-2}((\pa_xu)_{\pm}(x,y)+L_2^{-1})\right).
\end{align}
By the bound above on $|\pa_xV(x,y)u(x,y)|$, provided $c_1>0$ is sufficiently small, the right hand side of \eqref{eqn:2pointwisederivx5} is $\geq 0$. Combining this with the bounds from \eqref{eqn:2pointwisederivx1} and  \eqref{eqn:2pointwisederivx2} on the boundary of $U$, we can apply the maximum principle to conclude that
\begin{align*}
|\pa_xu(x,y)| \leq W(x,y).
\end{align*}
The function $W(x,y)$ satisfies $W(0,y) \leq CL_2^{-1}$, and we can repeat the argument above with $W(x,y)$ shifted by an amount comparable to $L_2$ in the $x$-direction. This gives us the required bound on $\pa_xu(x,y)$ and concludes the proof of the corollary. 
\end{proof}

\noindent
\setlength{\tabcolsep}{0ex}
\begin{tabular*}{\textwidth}{@{\extracolsep{\fill}}r}
  \begin{tabular}{l}\\
   \tmtextsc{{Department of Mathematics, Princeton University, Fine Hall, Washington Road,}} \\ 
   \tmtextsc{{Princeton, NJ 08544}}\\\\
   \tmtextsc{\it E-mail address: { \bf tdbeck@math.princeton.edu}}\\
    \normalsize
  \end{tabular}
\end{tabular*}

\end{document}